\documentclass{elsarticle}

\makeatletter
\def\ps@pprintTitle{%
 \let\@oddhead\@empty
 \let\@evenhead\@empty
 \def\@oddfoot{\centerline{\thepage}}%
 \let\@evenfoot\@oddfoot}
\makeatother

\usepackage{amsmath}
\usepackage{amssymb}  
\usepackage{listings}   
\usepackage{mathtools}
\usepackage{enumitem}
\usepackage{amsthm}
\usepackage[a4paper,top=4cm,bottom=4cm,left=4.5cm,right=4.5cm]{geometry}
\usepackage{multirow}
\usepackage{booktabs}
\usepackage{bigstrut}
\usepackage[hidelinks]{hyperref}
\usepackage{tikz-cd}
\usepackage[all]{xy}
\usepackage{mathrsfs}
\usepackage{mathtools}
\usepackage{color}  
\usepackage{ulem}   
\DeclareRobustCommand{\erase}{\bgroup\markoverwith{\textcolor{red}{\rule[.5ex]{2pt}{0.4pt}}}\ULon}  
\newcommand{\ran}[1]{X\times A_K^n(#1)}

\newcommand{\wt}[1]{\widetilde{#1}}

\newcommand{\calg}[1]{\widehat{\overline{#1}}}
\newcommand{\Robb}[1]{\mathrm{Rob}_{{#1}\times A_K^n(I)/{#1}}}
\newcommand{\Rob}[1]{\mathrm{Rob}_{{#1}\times A_K^n(I)}}

\def\spa{\partial_{s}}

\def\exp{\mathfrak{Exp}}
\def\weq{\overset{w}{\sim }}

\def\mupi{\mu_{p^{\infty}}}
\def\mupk{\mu_{p^k}}
\def\1{\textbf{1}}
\def\cS{\mathcal{S}}

\def\la{\langle}
\def\ra{\rangle}
\def\wh{\widehat}
\def\an{\mathrm{an}}

\def\Mat{\mathrm{Mat}}

\def\tpa{t_i\partial_{t_i}}
\def\NID{\mathrm{NID}}
\def\NLD{\mathrm{NLD}}

\def\lra{\longrightarrow}
\def\res{\mathrm{res}}
\def\dlog{\mathrm{dlog}}
\def\diag{\mathrm{diag}}

\def\Gal{{\mathrm{Gal}}} 
\def\sup{{\mathrm{sup}}} 
\def\sp{{\mathrm{sp}}}
\def\End{{\mathrm{End}}} 
\def\Im{{\mathrm{im}}} 
\def\coker{{\mathrm{coker}}} 
\def\Der{{\mathrm{Der}}}
\def\OO{{\mathcal{O}}}
\def\AA{\mathscr{A}}
\def\BB{\mathscr{B}}
\def\BBB{\mathbb{B}}
\def\HH{\mathscr{H}}
\def\ZZ{{\mathbb Z}} 
\def\RR{{\mathbb R}} 

\def\QQ{{\mathbb Q}} 

\def\OO{{\mathcal{O}}}
\def\Sp{{\mathscr{M}}}
\def\Spec{{\mathrm{Spec}}} 

\def \ZP{{\mathbb{Z}_p}}
\def \GL{{\mathrm{GL}}} 
\def \upd{{\overline{d}}}
\def \lowd{{\underline{d}}}
\def \Aut{{\mathrm{Aut}}} 

\def \cotimes{{\widehat{\otimes}}}
\def \dagotimes{{\otimes^\dagger}}

\theoremstyle{plain} 

\newtheorem{thm}{\indent\sc Theorem}[section]
\newtheorem{lem}[thm]{\indent\sc Lemma}
\newtheorem{cor}[thm]{\indent\sc Corollary}
\newtheorem{prop}[thm]{\indent\sc Proposition}
\newtheorem*{claim}{\indent\sc Claim}
\newtheorem{conj}[thm]{\indent\sc Conjecture}

\theoremstyle{definition} 
\newtheorem{convention}[thm]{\indent\sc Convention}
\newtheorem{defn}[thm]{\indent\sc Definition}
\newtheorem{rmk}[thm]{\indent\sc Remark}

\allowdisplaybreaks[4]

\begin{document}

\begin{frontmatter}
\pagenumbering{arabic}
\title{{\large\textbf{On generalized Fuchs theorem over relative $p$-adic polyannuli}}}
\author{Peiduo Wang}
\cortext[corresponding author]{Address: 153-8914 Komaba 3-8-1, Meguro, Tokyo, Japan}
\ead{wang-peiduo617@g.ecc.u-tokyo.ac.jp}
\address{Graduate School of Mathematical Sciences, the University of Tokyo}
\date{}
\begin{abstract}
    In this paper, we study coherent locally free (logarithmic-)$\nabla$-modules on relative $p$-adic polyannuli satisfying the Robba condition and prove several criteria for decomposition of such (logarithmic-)$\nabla$-modules. Firstly we prove the $p$-adic Fuchs theorem for absolute logarithmic $\nabla$-modules where the exponents have non-Liouville differences, which generalizes a result of Shiho. Secondly, we prove a generalized $p$-adic Fuchs theorem for relative $\nabla$-modules which are semi-constant on fibers. We also prove a generalized $p$-adic Fuchs theorem for absolute $\nabla$-modules, when the derivation on the base has some specific form. In the appendix, we prove the coincidence of two definitions of exponents due to Christol-Mebkhout and Dwork and prove that the set of exponents forms exactly one weak equivalence class.
\end{abstract}

\end{frontmatter}
\section*{Introduction}
Let $K$ be a complete nonarchimedean valuation field of mixed characteristic $(0,p)$. In \cite{CM2},
Christol and Mebkhout gave an intrinsic definition of the exponent of a finite free $\nabla$-module satisfying the Robba condition on one dimensional open annuli over $K$. They also showed that if the exponent has $p$-adic non-Liouville differences, then there exists a canonical decomposition of this $\nabla$-module with respect to the partition of its exponent into $\ZZ$-cosets. This is called the $p$-adic Fuchs theorem. Dwork gave another definition of exponents and another proof of $p$-adic Fuchs theorem on one dimensional open annuli in \cite{dwork1997exponents}. This method, after some simplification, is written in \cite{Ked1} Chapter 13. Note that the equivalence between the strategy of Christol and Mebkhout and the strategy of Dwork has not yet been verified, while that of Dwork is the one commonly adopted and applied throughout the main body of this paper.
 
In \cite{Ked2} and \cite{kedlaya2017corrigendum}, Kedlaya and Shiho proved a generalized version of one dimensional $p$-adic Fuchs theorem, by loosing the condition on exponents from having $p$-adic non-Liouville differences to a weaker one, namely, having Liouville partition, and yet still gives a decomposition of such a $\nabla$-module with respect to a Liouville partition of exponents. The author of this paper further generalized this result to high dimensional polyannuli in \cite{w1}.

There are two prior studies on relative and absolute $\nabla$-modules over relative (poly)annuli.
Let $X$ be a $K$-rigid space.
Firstly, in \cite{Kedlaya2022monodromy}, Kedlaya generalized some theories of $p$-adic differential equations on open annuli relative to $X$. He defined regular relative connections and $p$-adic exponents for coherent locally free modules with such a connection. The terminology ``regular relative connections'' will be referred to as Robba condition in this paper. However, some of his proofs are not clear and he did not prove the $p$-adic Fuchs theorem. In this paper, we give a reasonable definiton of $p$-adic exponents for modules with relative connections and prove the (generalized) $p$-adic Fuchs  theorem under certain circumstances.

Now we further assume that $X$ is smooth and let $\Sigma$ be a subset of $\overline{K}^n$. In \cite{shiho2010logarithmic} Section 1 and Section 2, Shiho studied $\Sigma$-unipotent logarithmic-$\nabla$-modules on quasi-open polyannuli relative to $X$. In particular, when $X$ has one-point Shilov boundary, he proved a proposition called ``generization'' which roughly speaking, asserts that for $\Sigma$ which has both non-Liouville differences and non-integer differences, the property of $\Sigma$-unipotence is ``generic on the base'' (see \cite{shiho2010logarithmic} Proposition 2.4).
We interpret his result as the $p$-adic Fuchs theorem for absolute logarithmic-$\nabla$-modules on relative polyannuli whose base is an affinoid space with one-point Shilov boundary.
One of our main results is a generalization of this result to the case where the Shilov boundary of the base $X$ is not necessary the singleton. Let $P$ be an absolute logarithmic-$\nabla$-module on open polyannuli relative to $X$. Then, the strategy of the proof is as follows: if $K$ is algebraically closed, then every smooth rigid space is locally a finite \'etale cover of the unit polydisc $\BBB$, which has one-point Shilov boundary. We then use Shiho's result to decompose the pushforward of $P$, which is an absolute $\nabla$-module over open polyannuli relative to $\BBB$. We can show that this decomposition is not only a decomposition of the pushforward of $P$ but also a decomposition of $P$ itself. Finally we apply Galois descent to obatin the result for arbitrary $K$.

We also consider generalized $p$-adic Fuchs theorem for coherent locally free modules with relative connections. Let $X$ be a $K$-smooth dagger curve and $\Sigma$ be a set having non-integer differences. We study $\Sigma$-semi-constant relative $\nabla$-modules on polyannuli relative to $X$ and show that $\Sigma$-semi-constancy is also ``generic on the base'' when $X$ is a disc. Moreover, we prove the generalized $p$-adic Fuchs theorem for relative $\nabla$-modules which are $\Sigma$-semi-constant on fibers.

Let us explain the content of this paper. In Section 1, we introduce notions and basic properties of rigid spaces, dagger spaces, $\nabla$-modules and generalized $p$-adic Fuchs theorem for absolute polyannuli. 
In Section 2, we firstly prove two Quillen-Suslin type conclusion to help us to reduce claims for general (logarithmic-)$\nabla$-modules to those for free ones. Then we introduce the definition of Robba condtion and $p$-adic exponent for coherent locally free $\nabla$-modules over relative polyannuli with connection relative to $X$ and prove some properties of them. In particular, we show that for any connected rigid or dagger space, $p$-adic exponent is well defined up to weak equivalence.
In Section 3, we study the invariance of $p$-adic exponents and decompositions along pushforward by finite \'etale morphisms.
In section 4, we firstly introduce terminologies and properties of logarithmic-$\nabla$-modules and $\Sigma$-unipotence. Next we prove that any smooth rigid or dagger space over an algebraically closed field is locally a finite \'etale cover of the unit polydisc. After that, we prove a Galois descent theorem. Finally, we use these arguments along with the theory developed in previous sections to prove a stonger version of the ``generization'' proposition by Shiho, which we call the $p$-adic Fuchs theorem over relative polyannuli. 
In Section 5, we introduce terminologies and properties of $\Sigma$-semi-constancy for modules with relative connections. We prove the generalized $p$-adic Fuchs theorem for relative $\nabla$-modules which are $\Sigma$-semi-constant on fibers, using push-forward technique developed in Section 3. We also introduce terminologies of $\xi$-constancy on the base for absolute $\nabla$-modules over relative polyannuli and prove the generalized $p$-adic Fuchs theorem for such modules.
In the appendix, firstly we show that in the case of polyannuli (which directly yields the case of relative polyannuli), the set of exponents forms exactly one weak equivalence class. Then we verify the equivalence of definition of $p$-adic exponent in \cite{CM2} and that in \cite{dwork1997exponents}. It is mentioned in the Notes of \cite{Ked1}, Chapter 13 and \cite{dwork1997exponents}, Remark 4.5 that this problem is not entirely trivial.

\section*{Acknowledgments}

This article is the doctor thesis of the author at the University of Tokyo. He would like to express his sincere gratitude to his supervisor Atsushi Shiho, for his indispensable guidance, insightful discussions, unwavering support, and heartfelt encouragement throughout the course of this research. 

This work was partially supported by FoPM (Forefront Physics and Mathematics Program to Drive Transformation) Program of the Graduate School of Science, the University of Tokyo. This work was supported by the Grant-in-Aid for Scientific Research (KAKENHI No. 24KJ0682) and the Grant-in-Aid for JSPS DC2 fellowship.
\subsection*{Notations}

\begin{enumerate}
    \item Throughout this paper, $K$ always denotes a complete nonarchimedean valuation field of mixed characteristic $(0,p)$,  $\mathcal{O}_K$ and $\kappa_K$ denotes its valuation ring and its residue field, respectively. Also we fix an algebraic closure $\overline{K}$ of $K$, and denote the completion of $\overline{K}$ by $\calg{K}$.
    \item For a $K$-Banach algebra $R$, use $\Sp(R)$ to denote its Berkovich spectrum. For a point $x\in \Sp(R)$, we denote the complete residue field at $x$ by $\HH(x)$.
    \item For a positive integer $k$, $\mupk$ denotes the group of $p^k$-th roots of unity in $\overline{K}$, and $\mupi=\bigcup_{k\geq 0}\mupk$ denote the group of $p$-power roots of unity.
    \item For a $K$-affinoid or $K$-dagger algebra $A$, $A^\circ$ denotes the ring of topologically bounded elements in $A$, $A^{\circ\circ}$ denotes the set of topologically nilpotent elements of $A$, which is an ideal of $A^\circ$, and $\widetilde{A}:=A^\circ/A^{\circ\circ}$ denotes the reduction of $A$. 
    \item We say a polysegment $I\subset \RR_{>0}^n$ is aligned if each endpoint at which it is closed belongs to $\sqrt{|K^\times|}$.
\end{enumerate}

\section{Preliminaries}
In this section, first we recall the notion of rigid spaces and dagger spaces. Then we recall properties of $p$-adic exponents and generalized $p$-adic Fuchs theorem over polyannuli, following \cite{w1}.
\begin{convention}
For two elements $\alpha=(\alpha_1,\dots,\alpha_m)$, $\beta=(\beta_1,\dots,\beta_m)$ in $\RR^m $, we write $\alpha<\beta$ (resp. $\alpha>\beta$, $\alpha\leq \beta$, $\alpha\geq\beta$) if $\alpha_i<\beta_i$ (resp. $\alpha_i>\beta_i$, $\alpha_i\leq \beta_i$, $\alpha_i\geq\beta_i$) for all $1\leq i\leq m$.
For $i=(i_1,\dots,i_m)\in \ZZ^m$ and a $m$-tuple of variables $t=(t_1,\dots,t_m)$, we denote $t_1^{i_1}\cdots t_m^{i_m}$ by $t^i$, denote  
$|i_1|+\dots+|i_m|$ by $|i|$ and $i_1!\cdots i_m!$ by $i!$. For another $m$-tuple $j=(j_1,\dots,j_m)$, we denote $\binom{i_1}{j_1}\cdots\binom{i_m}{j_m}$ by $\binom{i}{j}$.
\end{convention}

\subsection{Affinoid and dagger algebras}
Firstly in this subsection, we give a brief introduction of affinoid algebras, dagger algebras and corresponding spaces. For readers wishing for more details, a better introduction is given in \cite{KedF} and \cite{Gro}.
\begin{defn}
    For each $\rho=(\rho_1,\dots,\rho_m)\in\RR_{>0}^m$ and an $m$-tuple of variables $s=(s_1,\dots,s_m)$, denote the ring
    $$\left \{ f=\sum_{i\in{\ZZ_{\geq 0}^m}}f_is^i\in K[[s]]: \lim_{|i|\to \infty}|f_i|\rho^i=0\right \},$$ 
 by $K\langle\rho_1^{-1}s_1,\dots,\rho^{-1}_ms_m \rangle$ or simply by $K\langle \rho^{-1}s \rangle$ if $m$ is clear from the context. The corresponding affinoid space is denoted by $\BBB^m_{\rho,K}$ ($\BBB^m_K$ if $\rho=1$ and $\BBB_\rho^m$ if $K$ is clear from the context) and called the $m$-dimensional closed polydisc of radius $\rho$ over $K$. A commutative $K$-Banach algebra $R$ is said to be $K$-affinoid if there exists a surjection $K\la \rho^{-1}s\ra\to R$ for some $\rho$. If furthermore we can choose $\rho=(1,\dots,1)$, then $R$ is said to be strictly $K$-affinoid.
The ring
 $$K\langle s \rangle^\dagger = \bigcup_{\rho>1} K\langle \rho^{-1}s \rangle $$
is called the Monsky-Washnitzer algebra (over $K$). It is a topological ring with respect to $1$-Gauss norm, called the affinoid topology. A dagger algebra $R$ is a topological $K$-algebra isomorphic to a quotient of $K\langle s \rangle^\dagger$. The completion of $R$ with respect to this topology is $K$-affinoid and we denote it by $R'$.
\end{defn}

\begin{defn}
For a dagger algebra $R$ with surjection $f:K\langle s \rangle^\dagger\to R$, a fringe algebra is a strict $K$-affinoid algebra of the form $f(K\langle \rho^{-1}s \rangle)$ for some $\rho\in\sqrt{|K^\times|}$. The corresponding affinoid space is called a fringe space of $\Sp(R)$.
\end{defn}

Now recall the notion of polyannuli. We denote the Berkovich affine $n$-space $(\Spec K[t_1,\dots,t_n])^{\mathrm{an}}$ over $K$ by $\mathbb{A}_K^{n,\mathrm{an}}$.
\begin{defn}
For a polysegment $I=\prod_{i=1}^n I_i\subset\mathbb{R}_{> 0}^n$, the polyannulus with radius $I$ over $K$ is the subspace of $\mathbb{A}_K^{n,\mathrm{an}}$ defined by
$$A^n_K(I):=\left\{ x\in \mathbb{A}_K^{n,\mathrm{an}}: t_i(x)\in I_i ,\ 1\leq i\leq n \right\},$$
and we call such a subspace an open (resp. closed) polyannulus if $I$ is open (resp. closed).
The ring of global sections of the structure sheaf on $A_K^n(I)$ is the ring
$$\left \{ f=\sum_{i\in{\ZZ^n}}f_it^i\in K[[t,t^{-1}]]: \lim_{|i|\to \infty}|f_i|\rho^i=0\quad \forall \rho\in I\right \},$$ 
and we denote it by $K_{I,n}$ or simply by $K_I$ if $n$ is clear from the context.
The ring of global sections of the overconvergent structure sheaf on $A_K^n(I)$ is given by
$\varinjlim_{J\supset I} K_J$, where $J$ runs over all polysegments containing $I$ in its interior. We denote this direct limit ring by $K^\dagger_{n,I}$ or simply by $K^\dagger_I$ if $n$ is clear from the context.
For $\rho\in I$, we define the $\rho$-Gauss norm of $f=\sum_{i\in \ZZ^n}f_i t^i\in K_{I,n}$ or $K^\dagger_{I,n}$ by $|f|_\rho:=\max_{i\in\ZZ^n} \{|f_i|\rho^i\}$.
\end{defn}
When $I$ is closed, $K_{I}$ (resp. $K^\dagger_I$) is a $K$-affinoid algebra (resp. $K$-dagger algebra) in the sense of Berkovich, and the supremum norm (which is power multiplicative but not necessarily multiplicative) on $K_{I}$ is defined by $|f|_{I}:=\max_{\rho\in I}\{|f|_\rho\}$.

For any dagger algebra $R$, its Berkovich spectrum coincides with that of its completion $R'$ as topological spaces. Thus to distinguish dagger spaces from affinoid spaces, we use the symbol $X^\dagger$ to denote the topological space $X$ equipped with the overconvergent structure sheaf.

For any affinoid algebra $R$ and $x\in \Sp(R)$, the natural homomorphism $\chi_x: R\to \HH(x)$ induces a homomorphism on their reductions $\widetilde{\chi_x}:\wt{R}\to \wt{\HH(x)}$. Then we obtain the map
 \begin{align*}
    \mathrm{red}:\Sp(R)& \longrightarrow\Spec(\widetilde{R})\\
    x & \longmapsto \ker \wt{\chi_x}
 \end{align*}
and call it the reduction map of $\Sp(R)$. We call the unique minimum closed subset of $\Sp(R)$ on which every element of $R$ takes its maximum (which exists by \cite{Ber} Corollary 2.4.5) the Shilov boundary of $\Sp(R)$.

\begin{rmk}[\cite{Ber} Section 2.4]\label{reduction}
Let $R$ be an affinoid algebra, then the reduction map has following properties.
\begin{enumerate}
    \item For any morphism of affinoid spaces 
    $f:\Sp(R)\to\Sp(S)$, there is a natural morphisms of their reductions $\wt{f}:\Spec(\wt{R})\to\Spec(\wt{S})$ and we have the following commutative diagram 
    \begin{equation*}
        \begin{tikzcd}
            \Sp(R) \arrow{r}{f} \arrow[swap]{d}{\mathrm{red}} & \Sp(S) \arrow{d}{\mathrm{red}} \\%
            \Spec(\wt{R}) \arrow{r}{\wt{f}}& \Spec(\wt{S}).
        \end{tikzcd} 
    \end{equation*}
    \item The reduction map of a strict affinoid algebra is surjective and anti-continuous, that is, the preimage of an open (resp. closed) subset of $\Spec (\wt{R})$ in $\Sp(R)$ is closed (resp. open).
    \item Let $R$ be a strict $K$-affinoid algebra. Then the map $\mathrm{red}$ induces a bijection between the Shilov boundary of $\Sp(R)$ and the set of generic points of $\Spec(\wt{R})$. Moreover, if $|R|_\sup=|K|$ (e.g. $K$ is algebraically closed), for such a point there is an isomorphism $k(\wt{x})\cong \wt{\HH(x)}$. 
\end{enumerate}
\end{rmk}

Let $X$ be an affinoid space, let $I=[\alpha,\beta]\subset \RR^n_{>0}$ be a closed polysegment, let 
$$\mathrm{Vert}(I):=\{\rho\in I: \rho_i\text{ equals to either }\alpha_i \text{ or } \beta_i,\text{ for } 1\leq i\leq n \}$$
be the set of vertices of $I$ and let $\mathcal{S}(X)$ be the  Shilov boundary of $X$.
Then we have the following description of the Shilov boundary of $X\times A_K^n(I)$.
\begin{lem}\label{shilov boundary of relative polyannuli}
    Let $X=\Sp(R)$ be an affinoid space and $I\subset \RR^n_{>0}$ be a closed polysegment. Then the set
    $$\mathcal{S}:=\{|\cdot|_{(x,\rho)}\in X\times A_K^n(I):x\in\mathcal{S}(X) \text{ and } \rho\in \mathrm{Vert}(I)\}$$
    is the Shilov boundary of $X\times A_K^n(I)$.
\end{lem}
\begin{proof}
    Let $f\in \Gamma(\ran{I},\OO_{\ran{I}})$ and write $f=\sum_{i\in\ZZ^n}f_i t^i$ with $f_i\in R$. For any seminorm $|\cdot|\in\ran{I}$, 
    $$ |f|\leq \max_{i\in\ZZ^n}|f_i t^i|=|f_{i_0}| |t^{i_0}|\leq |f_{i_0}|_x \rho^{i_0}\leq\max_{j\in\ZZ^n}|f_j|_x|\rho^j=|f|_{x,\rho}, $$
    for some $x\in \mathcal{S}(X), \rho\in\mathrm{Vert}(I)$, where $i_0$ is an index such that $|f_{i_0}t^{i_0}|$ attains the maximum. Thus, for any element, it attains the maximum at some point in $\mathcal{S}$. Conversely, we easily see that $\mathcal{S}$ is the minimal set satisfying this property. Thus $\mathcal{S}$ is the Shilov boundary of $\ran{I}$.
\end{proof}

For an $l$-tuple $r=(r_1,\dots,r_l)\in\RR^l_{>0}$, we say $r$ is $K$-free if the image of $\log r_1,\dots,\log r_l$ in $\RR/\log\sqrt{|K^\times|}$ generates an $l$-dimensional $\QQ$-subspace. In this case, the ring $K_{[r,r]}$ is a complete nonarchimedean field and we denote it simply by $K_r$.

\begin{lem}
    Let $f:X\to Y$ be a finite surjective morphism of equidimensional affinoid spaces. Then $f^{-1}(\cS(Y))=\cS(X)$.
\end{lem}
\begin{proof}
    We start with the case where $X,Y$ are both strict affinoid and set $X=\Sp(R)$ and $Y=\Sp(S)$. They have the same dimension. Then $\wt{f}$ is finite by \cite{BGR} Theorem 6.3.5/1 and is surjective by the commutative diagram in Remark \ref{reduction}/1 and the surjectivity of the reduction map in Remark \ref{reduction}/2. Also, $\Spec(\wt{R})$ and $\Spec (\wt{S})$ are equidimensional and of the same dimension by \cite{poi} Lemme 4.9. In this case, the preimage of the set of generic points of $\Spec(\wt{S})$ is exactly the set of generic points of $\Spec(\wt{R})$. So we obtain the desired conclusion by Remark \ref{reduction}/3. When $X,Y$ are not necessarily strict affinoid, we can find a $K$-free tuple $r$ such that $X\otimes K_r$ and $Y\otimes K_r$ are both strict $K_r$-affinoid. Then we have the following commutative diagram: 
    \begin{equation*}
        \begin{tikzcd}
            X\otimes K_r \arrow{r}{f\otimes K_r} \arrow[swap]{d}{\pi_1} & Y\otimes K_r \arrow{d}{\pi_2} \\%
            X \arrow{r}{f}& Y
        \end{tikzcd} 
    \end{equation*}
    where $\pi_1,\pi_2$ denote the canonical projections.
    Then the conclusion follows from Lemma \ref{shilov boundary of relative polyannuli} and the strict affinoid case.
\end{proof}

\subsection{$p$-adic exponents and generalized $p$-adic Fuchs theorem}

Now we recall the definition of abstract exponents, $p$-adic Louville numbers and their properties. We will not include proofs of propositions in this subsection: for further details, one can check \cite{Ked1} Chapter 13 and \cite{Ked2} Section 3 for $1$-dimensional situations, \cite{Gac} and \cite{w1} Section 1 for high dimensional situations.

\begin{defn}
    For $a\in\ZP$, set 
    $\la a \ra_m:=\min\{n\in\ZZ_{>0}:a+n\text{ or } a-n\in p^m\ZP\}$. We say that $a\in\ZP$ is a $p$-adic Liouville number if $a\notin \ZZ$ and 
    $$ \liminf_{m\to \infty}\frac{\la a \ra_m}{m}<\infty. $$
    If $a$ is not $p$-adic Liouville, we say that it is a $p$-adic non-Liouville number.
\end{defn}
\begin{defn}[D\'efinitions in p.194 of \cite{Gac}] 
    Let $A=\{A_1,\dots,A_m\}$ be a multisubset of $\ZZ_p^n$. With $A_k=(A_k^1,\dots,A_k^n)$, we say that $A$ has $p$-adic non-Liouville differences (resp. non-integer differences) in the $r$-th direction if for any $1\leq l<k\leq m$, $A^r_l-A^r_k$ is a $p$-adic non-Liouville number (resp. not an integer), and we say that $A$ has $p$-adic non-Liouville differences or $A$ is $(\NLD)$ (resp. non-integer differences or $A$ is $(\NID)$) if it has $p$-adic non-Liouville  differences (resp. non-integer differences) in every direction.
    \end{defn}

    \begin{defn}
        [cf.  \cite{Ked2}, Definition 3.4.4, \cite{w1}, Definition 1.28] Let $A,\AA_1,\dots,\AA_k$ be multisubsets of $\ZZ_p^n$ such that $A=\bigcup_{i=1}^k \AA_i$ as multisets. We say that $\AA_1,\dots,\AA_k$ form a Liouville partition of $A$ in the $r$-th direction if $\AA_1^r,\dots,\AA_k^r$ is a Liouville partition of $A^r$, namely, for any $1\leq l<m\leq k$ and $a_l\in \AA_l^r, a_m\in\AA_m^r$, $a_l-a_m$ is a $p$-adic non-Liouville number which is not an integer. We say that $\AA_1,\dots,\AA_k$ form a Liouville partition of $A$ if it can be obtained   inductively on $k$ as follows:
    \begin{enumerate}
        
        \item When $k=1$, $\AA_1$ is a Liouville partition of $A$ if $\AA_1=A$ as multisets.
        
        \item For general $k$, $\AA_1,\dots,\AA_k$ is a Liouville partition of $A$ if there exists a partition 
        $$ \{1,\dots,k\}=\bigcup_{i=1}^l I_i $$
        as sets for some $l\geq 2$ with each $I_i$ nonempty such that $\bigcup_{j\in I_1}\AA_j,\dots,\bigcup_{j\in I_l}\AA_j$ is a Liouville partition in the $r$-th direction of $A$ for some $1\leq r\leq n$ and that $\AA_j\ (j\in I_i)$ is a Liouville partition of $\bigcup_{j\in I_i} \AA_j$, which is defined by the induction hypothesis.
        \end{enumerate}
        
    \end{defn}
    
    \begin{defn}[D\'efinitions in p.189 of \cite{Gac}]
        \label{equivalence}
    For two multisubsets $A=(A_1,\dots,A_m)$ and $B=(B_1,\dots,B_m)$ of $\mathbb{Z}_p^n$, we say that $A$ is weakly equivalent to $B$ and denote it by $A\weq B$, if there exists a constant $c>0$ and a sequence of permutations $\sigma_h$ $(h\in \ZZ_{>0})$ of $\{1,2,\dots,m \}$ such that, for all $1\leq i \leq n$ and $1\leq j\leq m$, 
    $$ \la A^i_{\sigma_h(j)}-B^i_j \ra_h\leq ch. $$
    We say that $A$ is (strongly) equivalent to $B$ if there exists a permutation $\sigma$ of $\{1,2,\dots,m\}$ such that for all $1\leq i\leq n$ and $1\leq j\leq m$,
    $$A_{\sigma(j)}^i-B^i_j\in \ZZ.$$
    \end{defn}

    \begin{prop}[\cite{w1} Proposition 1.25, cf. \cite{Ked2}, Proposition 3.4.5]
        Let $A$ be a finite multisubset of $\ZZ_p^n$ and let $\AA_1,\dots,\AA_k$ be a Liouville partition of $A$ in the $r$-th direction.
        \begin{enumerate}
        
        \item Let $\BB_1,\dots,\BB_k$ be multisubsets of $\ZZ_p^n$ such that $\BB_i^r$ is weakly equivalent to $\AA_j^r$ for $1\leq j\leq k$. Then $\BB_1,\dots,\BB_k$ form a Liouville partition in the $r$-th direction of $B=\bigcup_{j=1}^k \BB_k$.
        
        \item Suppose that $B$ is a multisubset of $\ZZ_p^n$ weakly equivalent to $A$. Then $B$ admits a Liouville partition $\BB_1,\dots,\BB_k$ in the $r$-th direction such that $\BB_j$ is weakly equivalent to $\AA_j$ for $1\leq j\leq k$.
        
        \end{enumerate}
        \label{weak_equivalence}
        \end{prop}

\begin{defn}
    Let $I\subset\RR^n_{>0}$ be a polysegment.  A $\nabla$-module over $A^n_K(I)$ is a coherent locally free module $P$ over $A_K^n(I)$ with an integrable connection
$$\nabla:P\to P\otimes\Omega^1_{A^n_K(I)/K}.$$
\end{defn}
\begin{defn}\label{absoluteRobba}
    A complete nonarchimedean differential field of order $m$ is a complete nonarchimedean field $F$ of characteristic $0$ equipped with $m$ commuting bounded nonzero derivations $\partial_1,\dots,\partial_m$. A differential module over a complete nonarchimedean differential field $F$ is a finite dimensional normed $F$-vector space $V$ equipped with commuting bounded actions of $\partial_1,\dots,\partial_m$ satisfying the Leibniz rule. 
    Let $|D_i|_V$ denote the operator norm of $D_i$ on $V$. We use $|D_i|_{\sp,V}$ to denote the spectral radius of $D_i$ on $V$, i.e. $\lim_{k \to \infty}|D_i^k|_V^{\frac{1}{k}}$. If $V\neq 0$, define the intrinsic radius of convergence of $V$ on the $i$-th direction for $1\leq i\leq m$ by
$$  IR_{\partial{i}}(V)
=\frac{|\partial_{i}|_{\mathrm{sp},F}}{|D_i|_{\mathrm{sp},V}},  $$
which is a number in $(0,1]$. Define the intrinsic radius of convergence $IR(V)$ of $V$ as the minimal value of intrinsic radii of convergence among all directions.
\end{defn}
\begin{defn}[cf. \cite{KX}, Definition 1.5.2] Let $I\subset \RR_{>0}^n$ be a polysegment and let $P$ be a $\nabla$-module over $A^n_K(I)$. Take $\rho\in I$, let $F_\rho$ be the completion of $K(t_1,\dots,t_n)$ with respect to the $\rho$-Gauss norm, and put $V_\rho=P\otimes_{K_I}F_\rho$ which is a differential module over $F_\rho$. The intrinsic radius of $P$ at $\rho$ is defined as $IR(V_\rho)$.
    We say that $P$ satisfies the Robba condition if $IR(V_\rho)=1$ for all $\rho\in I$.
    \label{Robba}
    \end{defn}
    \begin{convention}
        For $\zeta=(\zeta_1,\dots,\zeta_n)\in \mupi^n$, an $n$-tuple of variables $t=(t_1,\dots,t_n)$ and an $n$-tuple of $m\times m$ diagonal matrices $$A=(A^1,\dots,A^n)=(\mathrm{diag}(a_{11},\dots,a_{1m}),\dots,\mathrm{diag}(a_{n1},\dots,a_{nm})),$$ we use the following conventions:

        \begin{enumerate}
         \item $\zeta t:=(\zeta_1 t_1,\dots,\zeta_n t_n)$.
        
        \item $\zeta^A:=\zeta_1^{A^1}\cdots\zeta_n^{A^n}$, with $\zeta_i^{A^i}:=\mathrm{diag}(\zeta_i^{ a_{i1}},\dots,\zeta_i^{ a_{im}})$.
        \end{enumerate}
    \end{convention}

    For a $\nabla$-module $P$ over $A^n_K(I)$ satisfying the Robba condition, we define the action of $\zeta\in \mupi^n$ on $P\otimes_KK(\mupi)$ by $$\zeta^*(x)=\sum_{\alpha\in \ZZ^n_{\geq 0}}(\zeta-1)^\alpha\binom{tD}{\alpha} (x),$$
    where $(\zeta-1)^\alpha=(\zeta_1-1)^{\alpha_1}\cdots(\zeta_n-1)^{\alpha_n}$ and $\binom{tD}{\alpha}=\binom{t_1D_1}{\alpha_1}\dots\binom{t_nD_n}{\alpha_n}  $ with $D_i=\nabla(\partial_{t_i})$.  This series converges because of the Robba condition. Moreover, the action of $\mupi^n$ is a group action (check \cite{w1}, Section 2 for details).

The following argument is essential for the definition of $p$-adic exponent.
\begin{prop}[\cite{w1}, Theorem 1.18]
    Let $I\subset \RR_{>0}^n$ be a polysegment, and let $P$ be a coherent locally free module over $A_K^n(I)$. Then for any $\rho$ in the interior of $I$, there exist a subpolysegment $J\subset I$ containing $\rho$ in its interior such that $P|_{A_K^n(J)}$ is free.
    \label{QS}
\end{prop}

The definition of $p$-adic exponent and its properties are as follows:

\begin{defn}[\cite{w1}, Definitions 3.1, 3.13, cf. \cite{Ked1} Definition 13.5.1]
        \label{exp}
            Let $P$ be a coherent free $\nabla$-module over $A^n_K(I)$ of rank $r$ with basis $e_1,\dots,e_r$. An exponent of $P$ is an $n$-tuple of $r\times r$ diagonal matrices of $A=(A^1,\dots,A^n)$ with entries in $\ZP$ for which there exists a sequence $\{S_{k,A}\}_{k=1}^\infty$ in $\mathrm{Mat}_{r\times r}(K_I)$ satisfying the following conditions.
        \begin{enumerate}
        \item  If we put 
        $(v_{k,A,1},\dots,v_{k,A,r})=(e_1,\dots,e_r)S_{k,A},$
        then for all $\zeta\in \mupk^n$
        $$\zeta^*(v_{k,A,1},\dots,v_{k,A,r})=(v_{k,A,1},\dots,v_{k,A,r})\zeta^A.$$
        
        \item There exists $l>0$ such that $|S_{k,A}|_{I}\leq p^{lk}$ for all $k$.
        
        \item We have $|\det(S_{k,A})|_{\rho}\geq 1$ for all $k$ and all $\rho\in I$.
        \end{enumerate}
        For an open polysegment $I\subset \RR_{>0}^n$ and a coherent locally free $\nabla$-module $P$ over $A^n_K(I)$, an exponent of $P$ is an exponent of $P|_{A_K^n(J)}$ for some $J\subset I$ such that $P|_{A_K^n(J)}$ is free (such $J$ always exists by Proposition \ref{QS}).
\end{defn}
It is an easy observation that if $A$ is an exponent of $P$, then any multisubset $B=(B_1,\dots,B_r)$ of $\ZZ_p^n$ that are equivalent to $A$ is also an exponent of $P$.
\begin{prop}[cf. \cite{w1} Theorem 3.2, Theorem 3.3]\label{weak equivalence}
Let $P$ be a $\nabla$-module over $A^n_K(I)$ satisfying the Robba condition. Then an exponent of $P$ always exists and any two exponents of $P$ are weakly equivalent.
\end{prop}

\begin{lem}[cf. \cite{w1}, Lemma 3.4 ]\label{exact_sequence}
    Let $0\to P_1\to P\to P_2\to 0$ be an exact sequence of $\nabla$-modules over $A^n_K(I)$. Then $P$ satisfies the Robba condition if and only if both $P_1$ and $P_2$ satisfy the Robba condition. Moreover, if $A_i$ is an exponent of $P_i$ for $(i=1,2)$, then the multiset union $A_1\cup A_2$ is an exponent of $P$.
\end{lem}
\begin{proof}
    For each $\rho\in I$, the intrinsic radius of convergence of $P$ at $\rho$ is equal to $1$ by \cite{Ked1} Lemma 6.2.8. Then $P$ admits $A_1\cup A_2$ as an exponent by \cite{w1} Lemma 3.4.
\end{proof}

Generalized $p$-adic Fuchs theorem is the following:

\begin{thm}[\cite{w1} Corollary 3.15]\label{decom_all}
        Let $P$ be a coherent locally free $\nabla$-module over $A^n_K(I)$ for some open polysegment $I$ satisfying the Robba condition, admitting an exponent $A$ with Liouville partition $\AA_1,\dots,\AA_k$. Then there exists a unique decomposition $P=P_{1}\oplus\dots\oplus P_{k}$, with $\AA_i$ being an exponent of $P_i$ for $1\leq i\leq k$.
\end{thm}
\begin{rmk}
    Keep notations as in Theorem \ref{decom_all}. The statement of Theorem \ref{decom_all} is slightly different from \cite{w1} Corollary 3.15. This is because by Proposition \ref{weak equivalence} and Proposition \ref{wkexp} in the appendix of this paper, a multisubset $B$ of $\ZZ_p^n$ is weakly equivalent to an exponent $A$ of $P$ if and only if $B$ is also an exponent of $P$, but the author was unaware of the ``only if'' part at the time the paper \cite{w1} was written. Due to this reason, if we assume $P$ is of rank $r$, we refer an exponent $A$ of $P$ in following three ways interchangeably:
    \begin{enumerate}
        \item As a multisubset $(A_1,\dots, A_r)$ of $\ZZ_p^n$ to keep the consistency of notations used in \cite{Ked2} and \cite{w1}.
        \item As an $n$-tuple of $r\times r$ diagonal matrices $(A^1,\dots,A^n)$ for simplicity of computation.
        \item As an element of $\ZZ_p^{r\times n}$ or an element of $\ZZ_p^{r\times n}/\weq$ for simplicity of notations.
    \end{enumerate}
\end{rmk}

\section{$p$-adic exponents for relative $\nabla$-modules over relative polyannuli}
In this section, we firstly prove two key claims of this paper.  The first one is a Quillen-Suslin type theorem, saying that any coherent locally free module over polyannuli of small positive width relative to polydiscs are free. It generalizes \cite{w1} Theorem 1.18. The second one asserts that any coherent locally free module over relative polyannuli becomes free after shrinking the polyannuli and the base. Then we introduce the notion of Robba condition and $p$-adic exponents for relative $\nabla$-module over relative polyannuli. We also prove basic properties of them.

\begin{convention}
    From now on, we fix a positive integer $m$ to denote the dimension of an equidimensional rigid space or dagger space (which will serve as the base of our relative polyannuli) and use $s_1,\dots,s_m$ to denote the natural coordinates of the $m$-dimensional closed polydisc $\BBB_{\gamma}^m$ of radius $\gamma\in\RR^m_{>0}$.
\end{convention}

\subsection{A Quillen-Suslin type theorem over $\mathbb{B}^m_{\gamma}\times A^n_K(I)$ }
This subsection is a generalization of Section 1.2 and 1.3 of \cite{w1}, and the strategy of the proof is basically the same. Here we will provide only a brief outline of the proof process to avoid unnecessary repetition. Readers are encouraged to refer to \cite{w1} for details.

\begin{defn}[\cite{Ked1}, Definition 8.1.5]
    Let $R$ be a $K$-Banach algebra with norm $|\cdot|$, and let $\alpha,\beta\in\RR_{>0}$ with $\alpha<\beta$. We define the ring $R[[_0 \alpha/t,t/\beta\rangle$ (the ring of Laurent series which converge and take bounded values on the one dimensional semi-open annulus $\alpha<|t|\leq \beta$) as follows:
    $$ R[[_0 \alpha/t,t/\beta\rangle := \left\{ \sum_{i\in\ZZ}c_it^i: c_i\in R,\ \sup_i \{|c_i|\alpha^i\}<\infty,\ \lim_{i\to\infty}|c_i|\beta^i=0  \right\}. $$ 
    For $f=\sum_{i\in\ZZ}f_it^i\in R[[_0 \alpha/t,t/\beta\rangle$ and $\rho\in[\alpha,\beta]$, the $\rho$-Gauss norm of $f$ is defined as $\max_{i\in\ZZ}\{|f_i|\rho^i\}$.
    \end{defn}

    \begin{defn}[\cite{w1}, Definition 1.7]
        An element $g=\sum_{i\in\ZZ} g_i t^i$ in $ R[[_0 \rho/t,t/\rho\rangle$ is called $t$-bidistinguished of (upper and lower) degree $(\lowd,\upd)$ if the following conditions are satisfied:
        \begin{enumerate}
            \item  Both $g_\lowd$ and $g_\upd$ are units in $R$.
            \item $|g_\lowd|\rho^\lowd=|g_\upd|\rho^\upd=|g|$, $|g_\lowd|\rho^\lowd>|g_\nu|\rho^\nu$ for all $\nu<\lowd$  and $|g_\upd|\rho^\upd>|g_\mu|\rho^\mu$ for all $\mu>\upd$.
        \end{enumerate}
    \end{defn}

    \begin{prop}[cf. \cite{w1}, Proposition 1.8]
        \label{Factorization1}
        Assume the norm on $R$ is multiplicative. Let $f=\sum_{i\in\ZZ}f_it^i\in  R[[_0 \rho/t,t/\rho\rangle$ be a $t$-bidistinguished element of degree $(\lowd,\upd)$. Then there exists a unique factorization $f=f_\upd t^\upd gh$ with 
        $$ g\in  R[[_0 \rho/t,t/\rho\rangle\cap R[[t]],\ \ h\in R[[_0 \rho/t,t/\rho\rangle \cap R[[t^{-1}]] $$
        and $|h|=|h_0|=1$, $|g-1|<1$, where $h_0\in R$ is the constant term of $h$.
    \end{prop}
\begin{proof}
    The proof is the same as proof of \cite{w1}, Proposition 1.8.
\end{proof}
From now on in this subsection, unless otherwise specified, let \break$\gamma=(\gamma_1,\dots,\gamma_m)\in\sqrt{|K^\times|}^m$ and $\rho=(\rho_1,\dots,\rho_n)\in \sqrt{|K^\times|}^n$. Let $R=K\langle\gamma^{-1}s\rangle\langle\hat{\rho}/\hat{t},\hat{t}/\hat{\rho}\rangle$, where $\hat{\rho}=(\rho_1,\dots,\rho_{n-1})$ and $\hat{t}=(t_1,\dots,t_{n-1})$. Then $R$ is complete with respect to $(\gamma,\hat{\rho})$-Gauss norm which is multiplicative (denoted by $|\cdot|$). Also notice that $K\langle\gamma^{-1}s\rangle\langle\rho/t,t/\rho\rangle\subset R[[_0 \rho_n/t_n,t_n/\rho_n\rangle$.
\begin{lem}[cf. \cite{w1} Lemma 1.9]
    \label{multiply_small_x}
    Let $f\in K\langle\gamma^{-1}s\rangle\langle\rho/t,t/\rho\rangle$ be $t_n$-bidistintuished of degree $(\lowd,\upd)$, and $x\in K\langle\gamma^{-1}s\rangle\langle\rho/t,t/\rho\rangle$ be an element with $|x|<1$. Then $(1+x)f$ is also $t_n$-bidistinguished of degree $(\lowd,\upd)$.
\end{lem}
\begin{proof}
    The proof is the same as proof of \cite{w1}, Proposition 1.9.
\end{proof}
Now we can prove the following factorization result in $K_{[\rho,\rho],n}$.

\begin{prop}\label{weierstrass}
For a $t_n$-bidistinguished element $f\in K\langle\gamma^{-1}s\rangle\langle\rho/t,t/\rho\rangle$ of degree $(\lowd,\upd)$, there exists a monic polynomial $P\in R[t_n]$ and a unit $u\in K\langle\gamma^{-1}s\rangle\langle\rho/t,t/\rho\rangle^\times$ such that $f=Pu$.
\end{prop}
\begin{proof}
    The proof is the same as proof of \cite{w1}, Proposition 1.10.
\end{proof} 

\begin{lem}[cf. \cite{w1} Lemma 1.11]
    \label{change_to_bid}
For any non-zero element $f\in K\langle\gamma^{-1}s\rangle\langle\rho/t,t/\rho\rangle$, there exists an isometric isomorphism $\sigma\in \Aut(K\langle\gamma^{-1}s\rangle\langle\rho/t,t/\rho\rangle)$ such that $\sigma(f)$ is $t_n$-bidistinguished.
\end{lem}
\begin{proof}
For simplicity of the proof, we may change the notation to replace entries of $s$ by $t_1,\dots, t_m$ and entries of $t$ by $t_{m+1},\dots,t_{m+n}$ and denote the $(m+n)$-tuple of indeterminants $(s,t)=(t_1,\dots,t_m,t_{m+1},\dots,t_{m+n})$ by $\underline{t}$. By adding an underline to a letter we just want to emphasize that it denotes an $(m+n)$-tuple. In the same way we denote $(\gamma_1,\dots,\gamma_m,\rho_1,\dots,\rho_n)$ by $\underline{\rho}$. For an index $\underline{i}=(i_1,\dots,i_m,i_{m+1},\dots,i_{m+n})$ in the proof, the symbol $\underline{i}\in\ZZ_{\geq 0}^m\times\ZZ^n$ means that $i_1,\dots,i_m\in\ZZ_{\geq 0}$ and $i_{m+1},\dots,i_{m+n}\in\ZZ$.
Let $f=\sum_{\underline{i}\in \ZZ_{\geq 0}^m\times\ZZ^n} f_{\underline{i}}\underline{t}^{\underline{i}}$ and $\underline{u}=(u_1,\dots,u_{m+n}), \underline{v}=(v_1,\dots,v_{m+n})$ be the largest and smallest index $\underline{\mu}$ with respect to the lexicographic order in $\ZZ_{\geq 0}^m\times\ZZ^n$ such that $|f_{\underline{\mu}}|\underline{\rho}^{\underline{\mu}}=|f|$. Let $L$ be a positive integer with $L\geq \max_{1\leq k\leq m+n}|\mu_k|$ for any $\underline{\mu}=(\mu_1,\dots,\mu_{m+n})$ satisfying $|f_{\underline{\mu}}|\underline{\rho}^{\underline{\mu}}=|f|$. Take $w\in K$ such that $|w|\rho_n^l=1$ for some $l\in\ZZ_{>0}$. Let $\sigma_j$ be the homomorphism defined by
\begin{align*}
\sigma_j:  K\langle\gamma^{-1}s\rangle\langle\rho/t,t/\rho\rangle & \longrightarrow  K\langle\gamma^{-1}s\rangle\langle\rho/t,t/\rho\rangle\ \ \text{ }\\
t_i & \longmapsto t_i(wt_n^l)^{j^{m+n-k}}\ \ \ \ \ \text{ for } 1\leq k\leq m+n-1,\\
t_n & \longmapsto t_n.
\end{align*}
When $j>3L$, one can argue as in proof of \cite{w1} Lemma 1.11 to prove that $\sigma_j(f)$ is $t_n$-bidistinguished of degree $(\lowd,\upd)$ where $\lowd=\sum_{i=1}^{m+n-1}lj^{m+n-i}v_i+v_{m+n}$ and $\upd=\sum_{i=1}^{m+n-1}lj^{m+n-i}u_i+u_{m+n}$. 
\end{proof}
\begin{rmk}\label{uniform_of_j}
    The proof of Lemma \ref{change_to_bid} implies that, for finitely many elements $f_1,\dots,f_m\in K\langle\gamma^{-1}s\rangle\langle\rho/t,t/\rho\rangle$, we can choose a uniform $j$ such that $\sigma_j(f_i)$ are all $t_n$-bidistinguished. 
    \end{rmk}
    \begin{defn}
    An $r$-tuple $f=(f_1,\dots,f_r)$ of elements of a ring $R$ is called unimodular if these elements generate the unit ideal. Two $r$-tuples $f$ and $g$ are called equivalent and denoted by $f\sim g$ if there is a matrix $M\in\GL_r(R)$ such that $Mf^T=g^T$.
    \end{defn}
    
    \begin{prop}[cf. \cite{w1}, Proposition 1.14]\label{unimodular}
    Let $f=(f_1,\dots,f_r)$ be a unimodular tuple in $K\langle\gamma^{-1}s\rangle\langle\rho/t,t/\rho\rangle$. Then $f\sim e_1=(1,0,\dots,0)$.
    \end{prop}
    \begin{proof}
    We prove it by induction on $n$. The case $n=0$ follows from \cite{kedlaya2004full} Proposition 6.4. Rest of the proof is the same as proof of \cite{w1} Proposition 1.14.
    \end{proof}

    \begin{prop}[cf. \cite{w1}, Proposition 1.15]\label{f.f_resolution}
        Every finite module over $K\langle\gamma^{-1}s\rangle\langle\rho/t,t/\rho\rangle$ has a finite free resolution.
    \end{prop}
    \begin{proof}
        We proceed by induction on $n$. The case $n=0$ follows from \cite{kedlaya2004full} Proposition 6.5. Rest of the proof is the same as proof of \cite{w1} Proposition 1.15.
    \end{proof}

    \begin{thm}[cf. \cite{w1} Corollary 1.17]
        For $\gamma\in\RR^m_{>0}$, every coherent locally free module over $\mathbb{B}^m_{\gamma}\times A^n_K([\rho,\rho])$ is free. 
    \end{thm}
    \begin{proof}
        Let $\rho=(\rho_1,\dots,\rho_n)\in \RR^n_{>0}$. We first reduce to the case $\rho\in\sqrt{|K|^\times}^n$. Let $V$ be the $\QQ$-linear subspace of $\RR/\log\sqrt{|K^\times|}$ generated by $\log\rho_1,\dots,\log\rho_n$. By properly changing the indices, we may assume that $\log\rho_1,\dots,\log\rho_l$ is a basis of $V$. Put $\rho'=(\rho_1,\dots,\rho_l)$ and $\rho''=(\rho_{l+1},\dots,\rho_n)$. Then $K'=K_{[\rho',\rho'],l}$ is a field and the norm on $K'$ extends that on $K$. Thus $\rho''\in \sqrt{|K'^\times|}^{n-l}$ and $K\langle\gamma^{-1}s\rangle\langle\rho/t,t/\rho\rangle=K'\langle\gamma^{-1}s\rangle\langle\rho''/t'',t''/\rho''\rangle$, where $t''=(t_{l+1},\dots,t_{n})$. 

        Now let $P$ be a finite locally free module over $R$. By Proposition \ref{f.f_resolution}, $P$ admits a finite free resolution and hence is stable by \cite{lang2002algebra}, XXI, Theorem 2.1. On the other hand, Proposition \ref{unimodular} implies that $P$ has the unimodular column extension property (see p.849, p.846 of \cite{lang2002algebra} for definition). Then by \cite{lang2002algebra}, XXI, Theorem 3.6, we conclude that $P$ is free.
    \end{proof}

    \begin{thm}[cf. \cite{w1} Theorem 1.18]\label{generalized QS}
        Suppose $I\subset \RR_{>0}^n$ is a closed polysegment and $\gamma\in \sqrt{|K^\times|}^m$. Let $P$ be a coherent locally free module over $\BBB_{\gamma}^m\times A^n_K(I)$, then for any $\rho$ in the interior of $I$, there exists $J\subset I$ containing $\rho$ in its interior such that $P|_{\BBB^m_{\gamma}\times A^n_K(J)}$ is free. 
    \end{thm}
\begin{proof}
    The proof is the same as \cite{w1} Theorem 1.18.
\end{proof}

\subsection{Relative $\nabla$-modules and their $p$-adic exponents}
In this subsection we introduce $p$-adic exponents for relative $\nabla$-modules over relative polyannuli satisfying the Robba condition and prove some properties for it.
\begin{defn}
    For a polysegment $I\subset \RR_{>0}^n$ and a $K$-rigid space or dagger space $X$, the relative polyannulus with radius $I$ over $X$ is the product $X\times A^n_K(I)$. If $X=\Sp(R)$ is an affinoid space (resp. affinoid dagger space), the ring of global sections of $X\times A_K^n(I)$ is
    $R_I=R\cotimes_K K_I$ (resp. $R^\dagger_I= R \dagotimes_K K^\dagger_I$).
    For $(x,\rho)\in X\times I$ and every $f=\sum_{i\in\ZZ^n}f_it^i\in R_I$, $(x,\rho)$ defines a seminorm over $R_I$ by setting $|f|_{(x,\rho)}=\max_{i\in\ZZ^n}
    \{|f_i|_x\rho^i\}$. The function $(x,\rho)\mapsto |f|_{(x,\rho)}$ is continuous on $X\times I$. When $I$ is closed, the supremum norm associated to $x$ on $R_I$ is defined by 
    $|f|_{(x,I)}:=\max_{\rho\in I}\{|f|_{(x,\rho)}\}$. The function $x\mapsto |f|_{(x,I)}$ is also continuous on $X$.
\end{defn}
\begin{convention}
    From now on, we fix a positive integer $n$ to denote the relative demension of a polyannulus $X\times A_K^n(I)$ relative to $X$.
\end{convention}

Firstly we prove the following two lemmas which help us to reduce most situations to the case of free modules.

        \begin{lem}\label{xtonbd}
            Suppose that $X=\Sp(R)$ is an affinoid space (resp. dagger space) and $I\subset\RR_{>0}^n$ be a closed polysegment. Let $\mathfrak{a}$ be an ideal of $R_I$ and $x\in X$. If $\mathfrak{a}\HH(x)_I$ is the unit ideal of $\HH(x)_I$, then there exists an affinoid (resp. affinoid dagger) neighborhood $U=\Sp(A)$ of $x$ such that $\mathfrak{a}\otimes_{R_I}A_I$ is the unit ideal of $A_I$.
            \end{lem}
            
            \begin{proof}
                We give a proof for affinoid spaces here, and the proof for dagger spaces is the same.
                Let $g_1,\dots,g_s\in R_I$ be a set of generators of $\mathfrak{a}$, then there exists $h_1,\dots,h_s\in \HH(x)_{I}$ such that 
            $$h_1g_1+\cdots+h_s g_s=1.$$
            Since $\mathrm{Frac}(R/\ker|\cdot|_x)[t,t^{-1}]$ is dense in $\HH(x)_{I}$, one can choose $h'_1,\dots,h'_s\in R[t,t^{-1}]$ and $f'\in R$ with $|f'|_x\neq 0$ such that 
            $$\frac{h'_1}{f'}g_1+\cdots+\frac{h'_s}{f'}g_s=1+\phi,$$
            where $\phi\in R_I[f'^{-1}]$ and $|\phi|_{(x,I)}<1$. Since $|\cdot|_{(x,I)}$ is a continuous function of $x$, one can choose an affinoid neighborhood $U=\Sp(A)$ of $x$ such that $|\phi|_{(y,I)}<1$ and $|f'|_y\neq 0$ for all $y\in U$, so that $1+\phi$ is invertible in $A_I$. Then $g_1,\dots,g_s$ generates the unit ideal over $A_{I}$.
            \end{proof}

            \begin{lem}\label{localQS}
                Suppose $X=\Sp(R)$ is an affinoid space (resp. affinoid dagger space), and $I\subset \RR^n_{>0}$ is a closed polysegment. Let $P$ be a coherent locally free module over $\ran{I}$. Then for every $x\in X$ and $\rho$ in the interior of $I$, there exists an affinoid (resp. affinoid dagger) neighborhood $U$ of $x$ and closed polysegment $J\subset I$ containing $\rho$ in its interior such that $P|_{U\times A_K^n(J)}$ is free.
            \end{lem}
            \begin{proof}
                We give a proof for affinoid spaces here, and the proof for affinoid dagger spaces is the same.
            By Proposition \ref{QS}, we can choose a closed polysegment $J\subset I$ containing $\rho$ in its interior such that $P|_{x\times A_K^n(J)}$ is free. So by restricting to $\ran{J}$, it is harmless to assume that $J=I$. Now choose a basis $e_1,\dots,e_r$ of $P|_{x\times A_K^n(I)}$ and write $m_1,\dots,m_k$ to be the set of generators of $P$ ($k\geq r$). Then we can find matrices $G=(g_{ij})\in \mathrm{Mat}_{k\times r}(\HH(x)_{I})$ and $H=(h_{ij})\in \mathrm{Mat}_{r\times k}(\HH(x)_{I})$ such that 
            \begin{align*}
                (e_1,\dots,e_r)&=(\overline{m}_1,\dots,\overline{m}_k)G\ \  \text{ and}\\
                (\overline{m}_1,\dots,\overline{m}_k)&= (e_1,\dots,e_r)H,
            \end{align*}
            where $\overline{m}_i$ is the image of $m_i$ in $P|_{x\times A_K^n(I)}$ for $1\leq i\leq k$.

            Then we have $HG=I_r$. Since $\mathrm{Frac}(R/\ker|\cdot|_x)[t,t^{-1}]$ is dense in $\HH(x)_{I}$, one can choose $G'=(\frac{g'_{ij}}{f})$ with $g'_{ij}\in R[t,t^{-1}]$, $f\in R$, $|f|_x\neq 0$ such that
            \begin{align*}
                |G-\overline{G}'|_{(x,I)}&<\min\left\{\frac{1}{|H|_{(x,I)}},1 \right\},
            \end{align*}
            where $\overline{G}'$ is the image of $G'\in \Mat_{k\times r}(R[f^{-1}][t,t^{-1}])$ in $\Mat_{k\times r}(\HH(x)_I)$.
            Write
            \begin{align*}
                I_r=HG &= H(G-\overline{G}')+H\overline{G}'.
            \end{align*}
            Since $|H(G-\overline{G}')|_{(x,I)}<1$, $H\overline{G}'$ is an invertible matrix with entries in $\HH(x)_I $. Set
            $$ (e'_1,\dots,e'_r)=(m_1,\dots,m_k) G'. $$
            Then we have $e'_1,\dots,e'_r\in \frac{1}{f}P$. If we denote the image of $e_i'$ in $P|_{x\times A_K^n(I)}$ by $\overline{e}'_i$ for $1\leq i\leq r$, then 
            $$(\overline{e}'_1,\dots,\overline{e}'_r)=(\overline{m}_1,\dots,\overline{m}_k)\overline{G}'=(e_1,\dots,e_r)H\overline{G}'$$ and so $\overline{e}'_1,\dots,\overline{e}'_r $ form a basis of $P|_{x\times A_K^n(I)}$. By choosing an affinoid neighborhood $\Sp(B)\subset X$ of $x$ such that $f$ is invertible in $B$, and then substituting $X$ by $\Sp(B)$, it is harmless to assume that $f$ is invertible in $R$. Let $\varphi:R_I^{\oplus r}\to P$ be the morphism mapping $l_i$ to $e'_i$ for $1\leq i\leq r$, where $l_1,\dots, l_r$ denote the natural basis of $R_I^{\oplus r}$. Now we have the exact sequence
            \begin{equation}\tag{*}
                0\to\ker\varphi\to R_I^{\oplus r}\overset{\varphi}{\to}P\to\coker\varphi\to 0. 
            \end{equation}
            Tensoring with $\HH(x)_I$, we obtain the following exact sequence:
            $$ \HH(x)_I^{\oplus r}\overset{\varphi\otimes 1}{\lra}P\otimes_{R_I}\HH(x)_I\to\coker(\varphi\otimes 1)\to 0. $$
            Since $P|_{x\times A_K^n(I)}=P\otimes_{R_I}\HH(x)_I$ is free and $e'_1,\dots,e'_r$ form a basis, we have $\coker(\varphi\otimes 1)=0$.
            Put $\mathfrak{a}:=\{ g\in R_I:g\cdot\coker\varphi=0 \}$.
            \begin{claim}
                let $\mathfrak{p}$ be a prime ideal of $R_I$ containing $\mathfrak{a}$, then $\coker\varphi\otimes_{R_I}\mathrm{Frac}(R_I/\mathfrak{p})\neq 0$.
            \end{claim}
                 We prove the claim. Firstly, the condition implies that $\coker\varphi\otimes_{R_I}(R_I)_\mathfrak{p}\neq 0$, because if is not the case, then there exists $g\notin\mathfrak{p}$ such that $g\cdot\coker\varphi=0$, which contradicts the assumption that $\mathfrak{a}\subset\mathfrak{p}$. Hence by Nakayama's lemma 
            $$\coker\varphi\otimes_{R_I}\mathrm{Frac}(R_I/\mathfrak{p})=\coker\varphi\otimes_{R_I}[(R_I)_\mathfrak{p}/\mathfrak{p}(R_I)_\mathfrak{p}]\neq 0.$$
            The claim is proved.
            Let $\pi:\ran{I}\to X$ be the projection.
            \begin{claim}
                $\mathfrak{a}\nsubseteq \ker|\cdot|_z$ for all $z\in\pi^{-1}(x)$.
            \end{claim}
            We prove the claim. Note that $\ker|\cdot|_z$ is a prime ideal of $R_I$, so if $\mathfrak{a}\subset \ker|\cdot|_z$ for some $z\in\pi^{-1}(x)$, then $\coker\varphi\otimes_{R_I}\mathrm{Frac}(R_I/\ker|\cdot|_z)\neq 0$ by the claim above. Thus $\coker\varphi\otimes_{R_I}\HH(z)\neq 0$. However, since 
            $$\coker\varphi\otimes_{R_I}\HH(z)=(\coker\varphi\otimes_{R_I}\HH(x)_I)\otimes_{\HH(x)_I}\HH(z),$$
            it implies that 
            $$\coker(\varphi\otimes 1)=\coker\varphi\otimes_{R_I}\HH(x)_I\neq 0,$$
            which is a contradiction. So the claim is proved.

            The claim implies that $\mathfrak{a}\HH(x)_I\nsubseteq\ker|\cdot|_z$ for all $z\in \Sp(\HH(x)_I)$, that is, $\mathfrak{a}\HH(x)_I=\HH(x)_I$. So by Lemma \ref{xtonbd}, there exists an affinoid neighborhood $U=\Sp(A)$ of $x$ such that $\mathfrak{a}\otimes_{R_I}A_I=A_I$ and so $\coker\varphi\otimes_{R_I}A_I=0$.
            
            Then, tensoring the exact sequence (*) with $A_I$, we have the following:
            $$ 0\to\ker\varphi\otimes_{R_I} A_I\to A_I^{\oplus r}\to P\otimes_{R_I}A_I\to 0. $$
            Since $P$ is finite projective of rank $r$, this exact sequence splits and so $\ker\varphi\otimes_{R_I}A_I$ is also projective but of rank $0$, implying that it is a zero module. Thus $P|_{U\times A_K^n(I)}$ is free.
            \end{proof}

\begin{defn}
    Suppose $X$ is a $K$-rigid or dagger space, and $I\subset\RR_{>0}^n$ is a polysegment. A $\nabla$-module over $X\times A_K^n(I)$ relative to $X$ is a coherent locally free module $P$ over $X\times A_K^n(I)$ with an integrable connection
    $$\nabla:P\to P\otimes_{\OO_{X\times A_K^n(I)}}\Omega^1_{X\times A_K^n(I)/X}.$$
\end{defn}
Let $X$ be an affinoid dagger space and $I\subset \RR_{>0}^n$ be a polysegment. For a $\nabla$-module $P$ over $\ran{I}$ relative to $X$, there exists a fringe space $X'$ of $X$, a closed polysegment $J$ containing $I$ in its interior and a $\nabla$-module $P'$ over $X'\times A_K^n(J)$ relative to $X'$ such that $P$ is the restriction of $P'$ to $X\times A_K^n(I)$. In this case, we say that $(P',X',J)$ is a presentation of $P$.

\begin{defn}[cf. \cite{Kedlaya2022monodromy}, Definition 2.3.1]
    Suppose $X$ is a rigid space and $I\subset \RR_{>0}^n$ is a polysegment. Let $P$ be a  $\nabla$-module over $X\times A_K^n(I)$ relative to $X$. We say that $P$ satisfies the Robba condition if $P|_{x\times A_K^n(I)}$ satisfies the Robba condition for all $x \in X$ in the sense of Definition \ref{Robba}. 
    For $X$ a dagger space, we say a $\nabla$-module $P$ over $X\times A_K^n(I)$ relative to $X$ satisfies the Robba condition if for every $x\in X$, there exists an affinoid dagger neighborhood $U$ of $x$ and a presentation $(P',U',J)$ of $P|_{U\times A_K^n(I)}$ such that $P'$ satisfies the Robba condition as a $\nabla$-module over $U'\times A_K^n(J)$ relative to $U'$.
\end{defn}
    Note that the intrinsic radius of convergence does not change after extension of the constant field, so neither does the Robba condition. Let $R$ be a $K$-affinoid algebra or a $K$-dagger algebra.
    There is a natural group action of $\mupi^n$ on $R_I\widehat{\otimes}_K K(\mupi)$: For $\zeta\in\mupi^n$ and $f\in R_I\widehat{\otimes}_K K(\mupi)$, the action is defined by
$$ \zeta^*(f(t))=f(\zeta t). $$
    Let $X=\Sp(R)$. For a $\nabla$-module $P$ over $\ran{I}$ relative to $X$ satisfying the Robba condition, we define the action of $\zeta\in \mupi^n$ on $P\otimes_KK(\mupi)$ by the following convergent series:
    $$\zeta^*(x)=\sum_{\alpha\in \ZZ^n_{\geq 0}}(\zeta-1)^\alpha\binom{tD}{\alpha} (x).$$

    Now we introduce the notion on exponents of relative $\nabla$-modules over relative polyannuli satisfying the Robba condition.

    \begin{defn}
        Suppose that $X$ is a rigid space or a dagger space and $I\subset \RR_{>0}^n$ is a closed polysegment. Let $P$ be a $\nabla$-module over $X\times A_K^n(I)$ relative to $X$ of rank $r$ satisfying the Robba condition. We denote the image of an exponent of $P|_{x\times A_K^n(I)}$ in the set of weak equivalence class $\ZZ_p^{r\times n}/\weq$ by $\exp_P(x)$ or simply by $\exp(x)$ if $P$ is clear from the context. An element $A$ in $\ZZ_p^{r\times n}/\weq$ is called an exponent of $P$ if it coincides with $\exp(x)$ for some $x\in X$. We say the exponent of $P$ is well-defined, if its exponent is uniquely determined as an element of $\ZZ_p^{r\times n}/\weq$, independently of the choice of $x$.
    \end{defn}

    In the rest of this subsection, we prove that the definition of exponent of a $\nabla$-module over $X\times A_K^n(I)$ relative to $X$ is well-defined when $X$ is connected. The idea of the proof is partly inspired by that of Kedlaya (cf. \cite{Kedlaya2022monodromy} Section 2.3).

    \begin{lem}[cf. \cite{Kedlaya2022monodromy} Lemma 2.3.5]\label{contraction}
        Let $X=\Sp(R)$ be an affinoid space, $I\subset\RR_{>0}^n$ be a closed polysegment and $P$ be a free $\nabla$-module over $X\times A_K^n(I)$ satisfying the Robba condition. For $T\subset X$, suppose there exists a linear map $\lambda:R\to V$ of $K$-Banach spaces with $1\notin \ker\lambda$ such that $\lambda$ is a contraction for every $x\in T$. Then there exists $A\in \ZZ_p^{r\times n}$ and a sequence of matrices $\{S_{k,A}\}_{k=0}^\infty\subset \Mat_{r\times r}(R_I)$ which defines $A$ as an exponent of $P|_{x\times A_K^n(I)}$ for all $x\in T$.
        \end{lem}
        \begin{proof}
            Choose a basis $e_1,\dots,e_r$ of $P$ and use $E(\zeta)\in\mathrm{Mat}_{r\times r}(R_I)$ to denote the representation matrix of the action of $\zeta^*$ on this basis. For each $n$-tuple of $r\times r$ diagonal matrices of $A=(A^1,\dots,A^n)$ with entries in $\ZP$, define a sequence of matrices $\{S_{k,A}\}_{k=0}^\infty$ in $\mathrm{Mat}_{r\times r}(R_I)$ as
            $$S_{k,A}=p^{-nk}\sum_{\zeta\in \mupk^n}E(\zeta)\zeta^{-A}.$$
            As was shown in the proof of \cite{w1} Theorem 3.2, it has the following properties:
            \begin{enumerate}
                \item It always satisfies condition 1 and 2 of Definition \ref{exp}, regarded as matrices in $\mathrm{Mat}_{r\times r}(\HH(x)_I)$ for all $x\in X$,
                \item $S_{0,A}=I$,
                \item $\det S_{k,A}= \sum\limits_{{b_{ij}=0,}{1\leq i,j\leq m}}^{p-1}\det(S_{k+1,A+p^k(\diag(b_{11},\dots,b_{1m}),\dots,\diag(b_{n1},\dots,b_{nm}))})$.
            \end{enumerate}
             It remains to show that we can choose $A$ such that for every $x\in T$, the sequence $\{S_{k,A}\}_{k=0}^\infty$ satisfies the condition 3 of Definition \ref{exp} up to a nonzero constant. Denote the constant term of $\det S_{k,A}$ by $(\det S_{k,A})_0$ which is an element in $R$. By the third property above, for any multisubset $A$ of $\ZZ_p^n$, there exists $B\equiv A\ (\mathrm{mod}\  p^k)$ such that $|\lambda((\det S_{k+1,B})_0)|_V\geq|\lambda((\det S_{k,A})_0)|_V
             $. Since $S_{0,A}=I$, we can choose suitable $A$ by induction on $k$ such that $|\lambda((\det S_{k,A})_0)|_V\geq |\lambda(1)|_V$ for any $k$. Since $\lambda$ is a contraction, we have
             $$|\det S_{k,A}|_{(x,\rho)}\geq |(\det S_{k,A})_0|_x\geq|\lambda((\det S_{k,A})_0)|_V\geq|\lambda(1)|_V$$ 
             for all $x\in T$ and $\rho\in I$.
        \end{proof}

        A non-empty connected $K$-analytic space $X$ is called a virtual disc (resp. a virtual annulus) if $X_{\calg{K}}$ is isomorphic to a disc (resp. an annulus whose orientations are preserved by $\Gal(\overline{K}/K)$) over $\calg{K}$ (\cite{duc} 3.6.32 and 3.6.35). In particular, it implies that any virtual disc (resp. virtual annulus) over an algebraically closed complete nonarchimedean field is a disc (resp. an annulus).
\begin{defn}[\cite{poineau2015convergence} Definition 2.2]
    Let $X$ be a $K$-analytic curve. A subset $S$ of $X$ is called a weak triangulation of $X$ if 
    \begin{enumerate}
        \item $S$ is locally finite and only contains points of type $2$ or $3$;
        \item every connected component of $X-S$ is a virtual open disc or a virtual open annulus.
    \end{enumerate}
\end{defn}
By \cite{duc} Th\'eor\`eme 5.1.14, every quasi-smooth $K$-analytic curve admits a weak triangulation. The skeleton of a virtual annulus $C$ is the set of points that have no neighborhood isomorphic to a virtual open disc and we denote it by $\Gamma_C$. The union of $S$ and the skeleton of the connected components of $X-S$ that are virtual annuli forms a locally finite graph, and we call it the skeleton of the weak triangulation $S$ and denote it by $\Gamma_S$. 

\begin{prop}\label{exp of curve}
    Suppose $X$ is a connected quasi-smooth $K$-analytic curve and $I\subset \RR^n_{>0}$ is a polysegment. Let $P$ be a $\nabla$-module over $\ran{I}$ relative to $X$ satisfying the Robba condition. Then exponent of $P$ is well-defined.
\end{prop}
\begin{proof}
    We can work locally to assume that $X=\Sp(R)$ is affinoid and $I$ is closed. Applying Lemma \ref{localQS}, we can also reduce to the case where $P$ is free. Firstly, we assume that $K$ is algebraically closed. Let $S$ be a weak triangulation of $X$. For $s\in S$, let $E_s$ be the connected component of $(X-S)\cup\{s\}$ containing $s$. Then every connected component of $E_s-\{s\}$ is an open annulus or an open disc. Moreover, if we denote one of these connected components by $C$, then $s$ is the unique boundary of $C$ if it is a disc and one of the boundaries of $C$ if it is an annulus.
    \begin{claim}
        For every point $x\in C$, $\exp(x)=\exp(s)$.
    \end{claim}
    We prove the claim as follows: if $C$ is an open disc, then $s$ is the unique boudary of $C$ and so for every $x\in C$, the restriction map $\lambda: R\to \HH(x)$ is a contraction for both $s$ and $x$. Thus by Lemma \ref{contraction}, $\exp(x)=\exp(s)$.
    
    If $C$ is an open annulus, then there exists an open interval $J\subset\RR_{>0}$ such that $C\simeq A_K^1(J)$. We define the following map 
    $$ \lambda: R\to\Gamma(C,\OO_X)\simeq K_J \to K, $$
    where the first arrow is the restriction map and the second arrow is the map $\sum_{i\in\ZZ}a_i t^i\mapsto a_0$. This is a contraction for every point in $\Gamma_C$. Thus by Lemma \ref{contraction}, there exists $A\in\ZZ_p^{r\times n}$ and a sequence of matrices $\{S_{k,A}\}_{k\geq 0}$ in $ \Mat_{r\times r}(R_I)$ such that $A=\exp(x)$ for every $x\in\Gamma_C$. Then we can find a sequence $\{x_l\}_{l\geq 1}\subset \Gamma_C$ which converges to $s$. We prove that $\{S_{k,A}\}_{k\geq 0}$ also defines $A$ as an exponent of $P|_{s\times A_K^n(I)}$. By the same argument as in proof of \cite{w1} Theorem 3.2, we see that 1 and 2 of Definition \ref{exp} are satisfied. Moreover, for every $k\geq 0$ and $\rho\in I$,
    $$ |\det S_{k,A}|_{(s,\rho)}=\lim_{l\to\infty} |\det S_{k,A}|_{(x_l,\rho)}\geq 1$$
    and so $A=\exp(s)$. For $y\in C-\Gamma_C$, then there exists $x\in\Gamma_C$ such that $y$ is contained in an open disc with $x$ its unique boundary. By previous argument for the case $C$ is a disc, $\exp(y)=\exp(x)=\exp(s)$. The claim is proved.

    Since $X$ is connected, it is arcwise-connected by \cite{Ber} Theorem 3.2.1. Thus for $s\neq s'\in S$, there is a homeomorphic embedding $\theta_{s,s'}:[0,1]\to X$. Let $s_1,\dots,s_k\in S$ be points contained in $\theta_{s,s'}(]0,1[)$ and order them in the way that $s_i=\theta_{s,s'}(t_i)$ for $1\leq i\leq k$ with $0<t_1<\cdots<t_k<1$. Then $\theta_{s,s'}((t_i,t_{i+1}))$ must be contained in some connected component $C$ of $X-S$. Then $s_i,s_{i+1}$ are two boundaries of $C$. By the claim above, $\exp(s_i)=\exp(s_{i+1})$ and so $\exp(s)=\exp(s')$, from which we conclude the desired argument.

    For general $K$, let $\pi:X_{\calg{K}}\to X$ be the canonical projection. For $x,y\in X$, choose $x'\in\pi^{-1}(x)$ and $y'\in\pi^{-1}(y)$ both of which belong to the same connected component. Since exponent does not change after extension of constant field, an exponent of $P$ defined at $x$ (resp. $y$) is an exponent of $P\cotimes_K \overline{K}$ defined at $x'$ (resp. $y'$). Then use the argument above for algebraically closed fields to conclude.
\end{proof}

\begin{lem}\label{exp pullback}
    Suppose $f:X\to Y$ is a morphism of rigid spaces, $I\subset \RR_{>0}^n$ is a polysegment and denote the induced morphism between relative polyannuli by $f_I$. Let $P$ be a $\nabla$-module over $Y\times A_K^n(I)$ satisfying the Robba condition. Then the pull-back $f^*_IP$ of $P$ to $\ran{I}$ also satisfies the Robba condition. Moreover, if the exponent of $P$ is well defined and equal to $A\in\ZZ_p^{r\times n}/\weq$, then the exponent of $f^*_IP$ is also well-defined and equal to $A$.
\end{lem}
\begin{proof}
    Since the Robba condition is verified point-wise, the first claim holds because the Robba condition is kept after field extensions. The second claim follows from the fact that, for $y\in Y$, if $\{S_{k,A}\}_{k\geq 0}\subset \Mat_{r\times r}(\HH(y)_I)$ is the sequence of matrices defining $A$ as an exponent of $P|_{y\times A_K^n(I)}$, then it also defines $A$ as an exponent of $f^*_IP|_{x\times A_K^n(I)}$ for all $x\in f^{-1}(y)$. 
\end{proof}

\begin{thm}
    Suppose $X$ is a connected $K$-rigid space and $I\subset \RR^n_{>0}$ is a polysegment. Let $P$ be a $\nabla$-module over $\ran{I}$ relative to $X$ satisfying the Robba condition. Then exponent of $P$ is well-defined.
\end{thm}
\begin{proof}
    For $x,x'$ in $X$, we may enlarge $K$ to assume that $x,x'$ are both Tate points. By \cite{achinger2023geometric} Proposition A.2, there exists a sequence of morphisms $f_i:C_i\to X\ (i=0,\dots,l)$ over $K$ where $C_i$ is a smooth connected affinoid $K$-analytic curve and such that $x\in \Im(f_0)$, $x'\in \Im(f_l)$ and $\Im(f_{i-1})\cap\Im(f_{i})\neq \emptyset$ for $1\leq i\leq n$. 
    To prove the claim by induction, we may assume that $l=1$. Choose $x''\in \Im(f_{0})\cap\Im(f_{1})$, let $y\in f^{-1}_0(x)$ and $y''\in f^{-1}_0(x'')$. Then, by Lemma \ref{exp pullback}, $\exp_P(x)=\exp_{f_{0I}^*P}(y)$ and $\exp_P(x'')=\exp_{f_{0I}^*P}(y'')$. Since by Lemma \ref{exp of curve}, $\exp_{f_{0I}^*P}(y)=\exp_{f_{0I}^*P}(y'')$, we have $\exp_P(x)=\exp_P(x'')$. By the same reason, we have $\exp_P(x')=\exp_P(x'')$ which implies $\exp_P(x)=\exp_P(x')$, as desired.
\end{proof}

Denote the subcategory of  $\nabla$-modules over $\ran{I}$ relative to $X$ satisfying the Robba condition by $\Robb{X}$.

For $(P_1,\nabla_1),(P_2,\nabla_2)\in\Robb{X}$, we can define the  connection $\nabla$ on the tensor product $P_1\otimes P_2$ by
$$\nabla(\partial)(x_1\otimes x_2)=\nabla_1(\partial)(x_1)\otimes x_2+x_1\otimes\nabla_2(\partial)(x_2)\quad (\partial\in\Der(\ran{I}/X)), $$
and we can check that $(P_1\otimes P_2,\nabla)$ is an object in $\Robb{X}$ by applying \cite{Ked1} Lemma 6.2.8(c) to each $\nabla(\partial_{t_i})$ for every $(x,\rho)\in X\times I$. 

For $(P,\nabla)\in\Robb{X}$, we can define the connection $\nabla^\vee $ on the dual $P^\vee$ by
$$ \nabla^\vee(\partial)(\varphi)(x)=\partial(\varphi(x))-\varphi(\nabla(\partial)(x))\quad (\partial\in\Der(\ran{I}/X)), $$
and we can check that $(P^\vee,\nabla^\vee)$ is an object in $\Robb{X}$ by applying \cite{Ked1} Lemma 6.2.8(c) to each $\nabla(\partial_{t_i})$ for every $(x,\rho)\in X\times I$. 

The following lemma is a generalization of \cite{w1} Lemma 3.4.
\begin{lem}\label{relative_exact_sequence}
    Assume $X$ is a connected rigid space or a dagger space and $I\subset \RR_{>0}^n$ is a polysegment. Let $0\to P_1\to P\to P_2\to 0$ be an exact sequence of $\nabla$-modules over $X\times A_K^n(I)$ relative to $X$, then $P$ is an object in $\Robb{X}$ if and only if both $P_1$ and $P_2$ are objects in $\Robb{X}$. Moreover, if $P_1,P_2$ are objects in $\Robb{X}$ with exponent $A_1,A_2$, respectively, then the multiset union $A_1\cup A_2$ is an exponent of $P$. We also have the followings:
    \begin{enumerate}
        \item $P_1\otimes P_2$ is an object in $\Robb{X}$ and admits the multiset $A_1+A_2$ as an exponent.
        \item $P_1^\vee$ is an object in $\Robb{X}$ and admits the multiset $-A_1$ as an exponent.
    \end{enumerate}
\end{lem}

\begin{proof}
Choose a point $x\in X$. $P_2$ is locally free and hence flat, so the following sequence we obtain by tensoring with $\HH(x)_I$ over $R_I$ is exact:
$$0\longrightarrow P_1|_{x\times A^n_K(I)}\longrightarrow P|_{x\times A^n_K(I)}\longrightarrow P_2|_{x\times A^n_K(I)}\longrightarrow 0.$$
Then the argument follows from Lemma \ref{exact_sequence}. By restricting $P_1\otimes P_2$ and $P_1^\vee$ to $x\times A_K^n(I)$, the claims 1 and 2 follow from \cite{w1} Lemma 3.4.
\end{proof}

\section{Pushforward of (relative) $\nabla$-modules}
In this section, we prove the invariance of Robba condition, $p$-adic exponents and properties of decomposition of (relative) $\nabla$-modules after pushforward by finite \'etale morphisms. This technique is useful because in Proposition \ref{finite etale}, we will prove that any smooth affinoid space or a dagger space over algebraically closed field is locally a finite \'etale cover of closed unit polydisc, which can help us to reduce the base to a rather simpler one.

\subsection{Pushforward of relative $\nabla$-modules by finite \'etale morphisms}
To study the invariance of $p$-adic exponent after pushforward by a finite \'etale morphism, firstly we prove following lemmas.
\begin{lem}\label{invariance of spectral radius}
    Let $\pi:F\to F'$ be a finite extension of complete nonarchimedean fields, and $V$ be a finite dimensional vector space over $F'$. For an additive map $T\in\End(V)=\End(\pi_* V)$, we have 
    $$|T|_{\sp,V}=|T|_{\sp,\pi_* V}.$$
\end{lem}
\begin{proof}
    It suffices to show that $|T^n|_V=|T^n|_{\pi_*V}$ for every $n\geq 0$. This is true because the norm of a linear operator only depends on the norm on $V$.
\end{proof}

\begin{lem}\label{invariance of exp}
Suppose $L/K$ is a finite extension of complete nonarchimedean fields of degree $d$ and $I\subset\RR^n_{>0}$ is a polysegment. Let $\pi:A_L^n(I)\to A_K^n(I)$ be the canonical map. If $P$ is a $\nabla$-module over $A_L^n(I)$ of rank $r$ satisfying the Robba condition, then $\pi_*P$ is a $\nabla$-module over $A_K^n(I)$ of rank $rd$ satisfying the Robba condition. Moreover, if $A$ is an exponent of $P$, then $\pi_*A:=\underbrace{A\cup\cdots\cup A}_{d\text{ times}}$ is an exponent of $\pi_*P$.
\end{lem}
\begin{proof}
    Take $\rho=(\rho_1,\dots,\rho_n)\in I$. Then the preimage of $|\cdot|_\rho\in A_K^n(I)$ is exactly $|\cdot|_\rho\in A_L^n(I)$.  Denote the complete residue field of $A_L^n(I)$ (resp. $A_K^n(I)$) at $|\cdot|_\rho$ by $F_{L,\rho}$ (resp. $F_{K,\rho}$). Then we have $\pi_*(P\otimes_{L_I}F_{L,\rho})=\pi_*P\otimes_{K_I}F_{K,\rho}$. By Lemma \ref{invariance of spectral radius}, we have
    $$|\nabla(\partial_{t_i})|_{\sp,\pi_*P\otimes_{K_I}F_{K,\rho}}=|\nabla(\partial_{t_i})|_{\sp,P\otimes_{L_I}F_{L,\rho}}.$$
    Since $|\partial_{t_i}|_{\sp,F_{K,\rho}}=|\partial_{t_i}|_{\sp,F_{L,\rho}}=p^{-1/(p-1)}\rho_i^{-1}$ for $1\leq i\leq n$ by \cite{Ked1} Definition 9.4.1, $\pi_*P$ satisfies the Robba condition.

    By Proposition \ref{QS}, we can shrink $I$ to reduce to the case where $P$ is free to prove the latter claim. Since $A$ is an exponent of $P$, we can find a basis $e_1,\dots,e_r$ of $P$ and a sequence of matrices $\{S_{k,A}\}_{k=0}^\infty$ such that conditions 1, 2 and 3 of Definition \ref{exp} are satisfied. Since the exponent does not change after field extensions, to prove the assertion for exponent of $\pi_*P$, we may enlarge $K$. Choose a finite Galois extension $F/K$. Then we have $\pi_*P\cotimes_K F=\pi_{F*}(P\cotimes_K F)$, where $\pi_F$ is the canonical map $\coprod_{i=1}^d A^n_{F_{\sigma_i}}(I)\to A^n_{F}(I)$ with $\sigma_1,\dots,\sigma_d\in\mathrm{Hom}_K(L,F)$ which is given by the following ring homomorphism: for $f=\sum_{i\in\ZZ^n}f_it^i\in F_I$, 
    $$ (\pi_F^\#(f))_j=\sigma_j(f)=\sum_{i\in\ZZ^n}\sigma_j(f_i)t^i, \text{ for } 1\leq j\leq d. $$
    It is obtained from $\pi\cotimes 1$ via isomorphisms $A_L^n(I)\times_K F\cong \coprod_{i=1}^d A_{F_{\sigma_i}}(I)$ and $A_K^n(I)\times _K F=A_F^n(I)$.
    Let $\{\1_j\}_{j=1}^d$ be the canonical basis of $\prod_{i=1}^d F_{\sigma_i,I}$ as free $F_I$-module. Here we notice that the action $\circ_{\sigma_i}$ of $f\in F_I$ on $f'\in F_{\sigma_i,I}$ is by $f\circ_{\sigma_i}f'=\sigma_i(f)f'$.
    Then $\{\1_j\cdot e_i\cotimes 1\}_{\substack{1\leq i\leq r\\1\leq j\leq d}}$ form a basis of $\pi_*P\cotimes_K F$ as $\nabla$-module over $A^n_F(I)$.
    Now we prove that $\{S_{k,\pi_*A}\}_{k=0}^\infty$ where 
    $$S_{k,\pi_*A}=
\begin{pmatrix*}
    \sigma_1^{-1}(S_{k,A})&\cdots& O\\
    \vdots &\ddots &\vdots \\
    O&\cdots&\sigma_d^{-1}(S_{k,A})
\end{pmatrix*}$$
is the desired sequence of matrices defining exponent of $\pi_*P$. Conditions 2 and 3 of Definition \ref{exp} are easy to veriry since each $\sigma_i$ is an isometry on $F_I$. Now we verify the condition 1. For any $k\geq 0$, choose $\zeta=(\zeta_1,\dots,\zeta_d)\in\prod_{j=1}^d\mupk^n$, then
\begin{align*}
    &\ \ \ \ \zeta^*_j[(\1_j\cdot e_1\cotimes 1,\dots,\1_j\cdot e_r\cotimes 1)\circ_{\sigma_j}\sigma^{-1}_j(S_{k,A})]\\
    &=  \zeta^*_j[( e_1\cotimes 1,\dots, e_r\cotimes 1)(\1_j\circ_{\sigma_j}\sigma^{-1}_j(S_{k,A}))]=\zeta^*_j[(e_1\cotimes 1,\dots, e_r\cotimes 1)S_{k,A}]\\
    &=(e_1\cotimes 1,\dots, e_r\cotimes 1)S_{k,A}\zeta_j^A=[(\1_j\cdot e_1\cotimes 1,\dots,\1_j\cdot e_r\cotimes 1)\circ_{\sigma_j}\sigma^{-1}_j(S_{k,A})]\zeta_j^A
\end{align*}
and so the condition 1 is verified.
\end{proof}
\begin{rmk}\label{property of finite etale}
    For a finite flat morphism $f: X\to Y$ of rigid spaces, we list some properties of it:
    \begin{enumerate}
        \item The function $y\mapsto \deg_y f$ is a locally constant function.
        \item If $X$ is nonempty and $Y$ is connected, then $f$ is surjective. This is because $f$ being finite implies that $f$ is closed and $f$ being flat implies that $f$ is open.
    \end{enumerate}
\end{rmk}

Now we verify the invariance of $p$-adic exponent via pushforward by finite \'etale morphisms.
\begin{prop}\label{exponent_of_pushforward}
    Suppose $I\subset \RR_{>0}^n$ is a polysegment. Let $f:X\to Y$ be a finite \'etale morphism of degree $d$ between connected rigid spaces and let $P$ be an object in $\Robb{X}$ of rank $r$ with exponent $A$. Use $f_I:X\times A_K^n(I)\to Y\times A_K^n(I)$ to denote the induced finite \'etale morphism between relative polyannuli. Then $f_{I*}P$ is an object in $\Robb{Y}$ of rank $rd$ admitting $f_{I*}A:=\underbrace{A\cup\cdots\cup A}_{d\text{ times}}$ as an exponent.
\end{prop}
\begin{proof}
Since $f$ is finite \'etale of degree $d$, for every $y\in Y$, $f^{-1}(y)$ is a nonempty finite subset of $X$ and for each $x\in f^{-1}(y)$, $\HH(x)/\HH(y)$ is a finite extension. Moreover, $\sum_{x\in f^{-1}(y)}[\HH(x):\HH(y)]=d$. Since $P|_{x\times A_K^n(I)}$ satisfies the Robba condition, $f_{I*}P|_{y\times A_K^n(I)}$ also satisfies the Robba condition by Lemma \ref{invariance of exp}. $A$ is an exponent of $P$ and hence an exponent of $P|_{x\times A_K^n(I)}$, so $f_{I*}A$ is an exponent of $f_{I*}P|_{y\times A_K^n(I)}$ by Lemma \ref{invariance of exp}. Rank of $f_{I*}P$ can be verified by Lemma \ref{invariance of exp} and Remark \ref{property of finite etale}.
\end{proof}

Keep the notations as in Proposition \ref{exponent_of_pushforward}. It is an easy observation that if $\AA_1\cup\cdots\cup\AA_k$ is a Liouville partition of $A$, then $f_{I*}\AA_1\cup\cdots\cup f_{I*}\AA_k$ is a Liouville partition of $f_{I*}A$.

\begin{lem}\label{zero module}
    Suppose $I\subset\RR_{>0}^n$ is a polysegment. Let $X$ be a reduced rigid or dagger space and $P$ be a coherent locally free module over $\ran{I}$. Then for every $v\in P$, $v=0$ if and only if $v|_{x\times A_K^n(I)}=0$ for every $x\in X$.
\end{lem}
\begin{proof}
    We only have to show that the condition is sufficient. We can easily reduce to the case where $X=\Sp(R)$ is affinoid.
    Since $P$ is locally free, we can find a coherent free module $Q$ on $\ran{I}$ containing $P$. Let $e_1,\dots,e_s$ be a basis of $Q$. Then regard $v$ as an element in $Q$, we can write $v=\sum_{i=1}^s f_ie_i$ with $f_i\in R_I$. We further write $f_i=\sum_{j\in\ZZ^n}f_{ij}t^j$. Because $v|_{x\times A_K^n(I)}=0$, we have $f_i|_{x\times A_K^n(I)}=0$, i.e. $f_{ij}\in\ker|\cdot|_x$ for every $x\in X$. Since $X$ is reduced, this implies that $f_i=0$ for $1\leq i\leq s$ and so $v=0$.
\end{proof}

\begin{lem}[cf. \cite{kedlaya2017corrigendum} and \cite{w1} Lemma 3.9]\label{uniqueness}
     Suppose $X$ is a connected reduced affinoid or affinoid dagger space and $I\subset \RR_{>0}^n$ is a polysegment. Let $P$ be an object in $\Robb{X}$ with exponent $A$ having Liouville partition $A=\AA_1\cup\cdots\cup\AA_k$. The decomposition of $P=P_1\oplus\cdots\oplus P_k$ such that $P_i$ admits $A_i$ as an exponent, if exists, is unique.
\end{lem}
\begin{proof}
    We may reduce to the case $k=2$ and it suffices to prove the assertion for Liouville partition on every direction $r$ for $1\leq r\leq n$.
    Suppose that $P$ decomposes as $P'_1\oplus P'_2$ and $P''_1\oplus P''_2$ with $P'_i,\ P''_i$ having exponents $\AA_i$  (here we implicitly use Proposition \ref{wkexp}). Let $f$ be the composition of the inclusion $P'_1\to P$ with the projection $P\to P''_2$. We prove that $f$ is the zero map. Let $f_x$ be the restriction of $f$ on $x\times A_K^n(I)$. Then we have exact sequences:
    $$0\to\ker (f_x)\to P'_1|_{x\times A_K^n(I)}\to \Im(f_x)\to 0,$$ 
    $$0\to \Im(f_x)\to P''_2|_{x\times A_K^n(I)}\to \coker(f_x)\to 0.$$
    Choose exponents $C_1, C_2, C_3$ for $\ker(f_x)$, $\Im(f_x)$ and $\coker(f_x)$, respectively. Then, by Lemma \ref{exact_sequence}, $\AA_1$ is weakly equivalent to $C_1\cup C_2$ and $\AA_2$ is weakly equivalent to $C_2\cup C_3$. Since $\AA_1\cup \AA_2$ is a Liouville partition in the $r$-th direction, $(C_1\cup C_2)\cup(C_2\cup C_3)$ is a Liouville partition in the $r$-th direction by Proposition \ref{weak_equivalence}. Thus $C_2=\emptyset$, and  $f_x=0$ for every $x\in X$. This implies that for every $v\in P_1$, we have $f(v)|_{x\times A_K^n(I)}=0$. By Lemma \ref{zero module}, $f(v)=0$ and so $f$ is the zero map. This shows that $P'_1\hookrightarrow P''_1$. By symmetry, we have $P''_1\hookrightarrow P'_1$ and so $P'_1=P''_1$. Replacing the index, we also have $P'_2=P''_2$.
\end{proof}
Properties of pushforward of decomposition of relative $\nabla$-module by finite \'etale morphisms is given as follows:
\begin{prop}\label{relative pushforward of decomposition}
    Suppose $f:X\to Y$ is a finite \'etale morphism of connected reduced rigid or dagger spaces of degree $d$ and $I\subset \RR_{>0}^n$ is a polysegment. Let $P$ be an object in $\Robb{X}$, with an exponent $A$ having Liouville partition $\AA_1\cup\cdots\cup\AA_k$ and let $f_I:X\times A_K^n(I)\to Y\times A_K^n(I)$ be the induced finite \'etale morphism on the relative polyannuli. If the pushforward $f_{I*}P$ of $P$ to $Y\times A_K^n(I)$ (which is an object in $\Robb{Y}$ by Lemma \ref{exponent_of_pushforward}) admits a decomposition (which is unique by Lemma \ref{uniqueness}) $P_1\oplus\cdots\oplus P_k$ with respect to the Liouville partition $f_{I*}A=f_{I*}\AA_1\cup\cdots\cup f_{I*}\AA_k$ with $P_i$ having exponent $f_{I*}\AA_i$. Then this is also a decomposition of $P$ in the category $\Robb{X}$ where $P_i$ admits $\AA_i$ as an exponent.
\end{prop}
\begin{proof}
    Firstly we deal with the case where $X=\Sp(R)$, $Y=\Sp(S)$ are affinoid and $I$ is closed. Without loss of generality we may assume $k=2$.
    For each $y\in Y$, by Proposition \ref{exponent_of_pushforward}, $f_{I*}P|_{y\times A_K^n(I)}$ is a $\nabla$-module over $y\times A_K^n(I)$ satisfying the Robba condition with exponent $f_{I*}A$ having Liouville partition $f_{I*}\AA_1\cup f_{I*}\AA_2$. So by Theorem \ref{decom_all}, there is a unique decomposition of $f_{I*}P|_{y\times A_K^n(I)}=P_{y,1}\oplus P_{y,2}$ as objects in $\Robb{y}$ with $f_{I*}\AA_i$ being an exponent of $P_{y,i}$ for $i=1,2$. Moreover, for each $x\in X$, again by Theorem \ref{decom_all}, there exists a unique decomposition of $P|_{x\times A_K^n(I)}= P_{x,1}\oplus P_{x,2}$ as objects in $\Robb{x}$ with $\AA_i$ being an exponent of $P_{x,i}$ for $i=1,2$.

    Since $f$ is finite \'etale, by Remark \ref{property of finite etale}, $f$ is surjective and $f^{-1}(y)$ is a finite discrete set for every $y\in Y$. Set $P_{f^{-1}(y),i}:=\bigoplus_{x\in f^{-1}(y)} P_{x,i}$ for $i=1,2$. This is a $\nabla$-module over $f^{-1}(y)\times A_K^n(I)$. By observing that $P_{f^{-1}(y),i}|_{x\times A_K^n(I)}=P_{x,i}$ for each $x\in f^{-1}(y)$, we obtain that it satisfies the Robba condition and $\AA_i$ is an exponent of $P_{f^{-1}(y),i}$. Thus $f_{I*}(P_{f^{-1}(y),i})$ is an object in $\Robb{y}$ having exponent $f_{I*}\AA_i$ for $i=1,2$, by Proposition \ref{exponent_of_pushforward}. 
    Since $f_{I*}P=P_1\oplus P_2$ as objects in $\Robb{Y}$ with $P_i$ having exponent $f_{I*}\AA_i$ for $i=1,2$, for every $y\in Y$, we have 
    $$ (f_{I*}P)|_{y\times A_K^n(I)}=P_1|_{y\times A_K^n(I)}\oplus P_2|_{y\times A_K^n(I)}. $$
    Moreover, since
    \begin{align*}
        (f_{I*}P)|_{y\times A_K^n(I)} &= f_{I*}(P|_{f^{-1}(y)\times A_K^n(I)})\\
        &=\bigoplus_{x\in f^{-1}(y)} f_{I*}(P|_{x\times A_K^n(I)})\\
        &= \bigoplus_{x\in f^{-1}(y)} f_{I*}(P_{x,1}\oplus P_{x,2})\\
        &= f_{I*}\left(\bigoplus_{x\in f^{-1}(y)} P_{x,1}\right)\oplus f_{I*}\left(\bigoplus_{x\in f^{-1}(y)} P_{x,2}\right)\\
        &=f_{I*}(P_{f^{-1}(y),1})\oplus f_{I*}(P_{f^{-1}(y),2}).
    \end{align*}
    By uniqueness assertion in Theorem \ref{decom_all} applied to $(f_{I*}P)|_{y\times A_K^n(I)}$, we have $P_{y,i}=f_{I*}(P_{f^{-1}(y),i})=P_i|_{y\times A_K^n(I)}$. 
    
    Now we prove that $P_i$ is also an object in $\Robb{X}$ for $i=1,2$. Firstly we prove that $P_i$ is a coherent module over $\ran{I}$. For $v\in P_1$ and $g\in R_I$, regarding $v$ as an element in $P$, we see that $gv\in P$. Now regard $gv$ as an element in $f_{I*}P$. Then we can write $gv=v_1+v_2$ with $v_i\in P_i$ for $i=1,2$ in a unique way. To prove that $gv\in P_1$, it suffices to show that $v_2=0$. Consider $v\otimes 1$ as an element in $P_1|_{y\times A_K^n(I)}$. Then $v\otimes 1\in P_{f^{-1}(y),1}$ and so $gv\otimes 1\in P_{f^{-1}(y),1}$. Since we can also write $gv\otimes 1=v_1\otimes 1+v_2\otimes 1$ with $v_i\otimes 1\in P_i|_{y\times A_K^n(I)}=P_{f^{-1}(y),i}$ for $i=1,2$. By uniqueness of expression of $gv\otimes 1$, we have $v_2\otimes 1=0$ in $ P_{f^{-1}(y),2}\subset P|_{f^{-1}(y)\times A_K^n(I)}$. This implies that $v_2\otimes 1=0$ in $P_{x,2}\subset P|_{x\times A_K^n(I)}$ for every $x\in f^{-1}(y)$ and so $v_2\otimes 1=0$ in $P|_{x\times A_K^n(I)} $ for every $x\in X$.
    By Lemma \ref{zero module}, $v_2=0$.

    By the same argument we can prove that $P_2$ is also a coherent module over $\ran{I}$. We have $\Omega^1_{\ran{I}/X}=\Omega^1_{Y\times A_K^n(I)/Y}\otimes_{\OO_Y\times A_K^n(I)} \OO_{\ran{I}}$ and this implies that     
    \begin{align*}
        P_i\otimes_{\OO_{\ran{I}}}\Omega^1_{\ran{I}/X}&=P_i\otimes_{\OO_{\ran{I}}}\Omega^1_{Y\times A_K^n(I)/Y}\otimes_{\OO_{Y\times A_K^n(I)}}\OO_{\ran{I}}\\
        &= P_i\otimes_{\OO_{Y\times A_K^n(I)}}\Omega^1_{Y\times A_K^n(I)/Y}.
    \end{align*}
    Thus the relative $\nabla$-module structure of $P_i$ on $\ran{I}$ is naturally given by that on $Y\times A_K^n(I)$ via 
    $$ \nabla:P_i\to P_i\otimes_{\OO_{Y\times A_K^n(I)}}\Omega^1_{Y\times A_K^n/Y}= P_i\otimes_{\OO_{\ran{I}}}\Omega^1_{\ran{I}/X}. $$
    So we have proven that $P_i$ is also an object in $\Robb{X}$ for $i=1,2$.

    When $X,Y$ are not necessarily affinoid and $I$ is not necessarily closed, find an affinoid covering $\{V_\lambda\}_{\lambda\in\Lambda}$ of $Y$ and a covering $\{I_\delta\}_{\delta\in\Delta}$ of $I$ by closed polysegments. Put $U_\lambda=f^{-1}(V_\lambda)$, then $\{U_\lambda\}_{\lambda\in \Lambda}$ is an affinoid covering of $X$. By applying previous argument to $f_I|_{U_\lambda\times A_K^n(I_\delta)}:U_\lambda\times A_K^n(I_\delta)\to V_\lambda\times A_K^n(I_\delta)$, we have the decomposition $P|_{U_\lambda\times A_K^n(I_\delta)}=\bigoplus_{i=1}^k P_{(\lambda,\delta),i}$ of $\nabla$-module over $U_\lambda\times A_K^n(I_\delta)$ relative to $U_\lambda$, with $\AA_i$ an exponent of $P_{(\lambda,\delta),i}$. By uniqueness assertion in Lemma \ref{uniqueness}, we have the canonical equality
    $$ P_{(\lambda,\delta),i}|_{(U_\lambda\cap U_{\lambda'})\times (A_K^n(I_\delta)\cap A_K^n(I_{\delta'}))}=P_{(\lambda',\delta'),i}|_{(U_\lambda\cap U_{\lambda'})\times (A_K^n(I_\delta)\cap A_K^n(I_{\delta'}))}, $$
    for $(\lambda,\delta),(\lambda',\delta')\in \Lambda\times \Delta$. Then since we have the exact sequence of sheaves
    $$ 0\to P\to \prod_{(\lambda,\delta)\in\Lambda\times \Delta}P|_{U_\lambda\times A_K^n(I_\delta)}\rightrightarrows \prod_{(\lambda,\delta),(\lambda',\delta')\in\Lambda\times\Delta}P|_{(U_\lambda\cap U_{\lambda'})\times (A_K^n(I_\delta)\cap A_K^n(I_{\delta'}))},  $$
    if we define 
    $$P_i=\ker\left( \prod_{(\lambda,\delta)\in\Lambda\times \Delta}P_{(\lambda,\delta),i}\rightrightarrows \prod_{(\lambda,\delta),(\lambda',\delta')\in\Lambda\times\Delta}P_{(\lambda,\delta),i}|_{(U_\lambda\cap U_{\lambda'})\times (A_K^n(I_\delta)\cap A_K^n(I_{\delta'}))} \right),$$
    we obtain the desired decomposition $P=P_1\oplus\cdots\oplus P_k$. The uniqueness follows from Lemma \ref{uniqueness}.
\end{proof}

\subsection{Pushforward of absolute $\nabla$-modules by finite \'etale morphisms}
Firstly, we introduce the notion of absolute $\nabla$-modules.
\begin{defn}
    Suppose $X$ is a smooth $K$-rigid or $K$-dagger space, and $I\subset\RR_{>0}^n$ is a polysegment. A(n absolute) $\nabla$-module over $X\times A_K^n(I)$ is a coherent module $P$ over $X\times A_K^n(I)$ endowed with an integrable connection
    $$\nabla:P\to P\otimes_{\OO_{X\times A_K^n(I)}}\Omega^1_{X\times A_K^n(I)/K}.$$
\end{defn}
The category of absolute $\nabla$-modules over $X\times A_K^n(I)$ is an abelian category. Also, an absolute $\nabla$-module over $\ran{I}$ is known to be automatically locally free.
If we denote the category of $\nabla$-modules over $\ran{I}$ satisfying the Robba condition by $\Rob{X}$, we have the following:
\begin{lem}\label{abcat}
    $\Rob{X}$ is a rigid Abelian tensor category.
\end{lem}
\begin{proof}
    Tensor product and dual in $\Rob{X}$ can be constructed in the same way as those in $\Robb{X}$.
    It remains to prove the Robba condition. For the kernel and the cokernel of a morphism $f:P\to Q$ in $\Rob{X}$, we have following exact sequences:
    \begin{align*}
        0\lra\ker f\lra P\to \Im f\lra 0\\
        0\lra \Im f\lra Q\lra \coker f\lra 0.
    \end{align*}
    Then the Robba condition for the kernel and the cokernel  follows from Lemma \ref{relative_exact_sequence}.
\end{proof}

\begin{lem}[cf. \cite{kedlaya2017corrigendum},\cite{w1} Lemma 3.9]\label{nabla uniqueness}
    Suppose $X$ is a smooth connected affinoid or dagger space and $I\subset \RR_{>0}^n$ is a polysegment. Let $P$ be an object in $\Rob{X}$ with exponent $A$ having Liouville partition $A=\AA_1\cup\cdots\cup\AA_k$. The decomposition of $P=P_1\oplus\cdots\oplus P_k$ such that $P_i$ admits $A_i$ as an exponent, if exists, is unique.
\end{lem}
\begin{proof}
   This follows immediately from Lemma \ref{uniqueness}.
\end{proof}

With a similar argument as in Proposition \ref{relative pushforward of decomposition}, properties of pushforward of decomposition of $\nabla$-module by finite \'etale morphisms is given as follows:
\begin{cor}\label{pushforward of decomposition}
    Suppose $f:X\to Y$ is a finite \'etale morphism of connected smooth rigid or dagger spaces of degree $d$ and $I\subset \RR_{>0}^n$ is a polysegment. Let $P$ be an object in $\Rob{X}$, with an exponent $A$ having Liouville partition $\AA_1\cup\cdots\cup\AA_k$ and let $f_I:X\times A_K^n(I)\to Y\times A_K^n(I)$ be the induced finite \'etale morphism on the relative polyannuli. Suppose that the pushforward $f_{I*}P$ of $P$ to $Y\times A_K^n(I)$ (which is an object in $\Rob{Y}$ by Lemma \ref{exponent_of_pushforward}) admits a decomposition (which is unique by Lemma \ref{nabla uniqueness}) $P_1\oplus\cdots\oplus P_k$ with respect to the Liouville partition $f_{I*}A=f_{I*}\AA_1\cup\cdots\cup f_{I*}\AA_k$ with $P_i$ having exponent $f_{I*}\AA_i$. Then this is also a decomposition of $P$ in the category $\Rob{X}$ where $P_i$ admits $\AA_i$ as an exponent.
\end{cor}
\begin{proof}
    As in the proof of Proposition \ref{relative pushforward of decomposition}, we may assume $k=2$ and firstly deal with the case where $X=\Sp(R)$, $Y=\Sp(S)$ are affinoid and $I$ is closed.
    We can argue in the same way as in the proof of Proposition \ref{relative pushforward of decomposition} to conclude that both $P_1,P_2$ are also modules over $\ran{I}$. Since $f_I$ is finite \'etale, we have $\Omega^1_{\ran{I}/K}=\Omega^1_{Y\times A_K^n(I)/K}\otimes_{\OO_{Y\times A_K^n(I)}}\OO_{\ran{I}}$ which implies that
    \begin{align*}
        P_i\otimes_{\OO_{\ran{I}}}\Omega^1_{\ran{I}/K}&=P_i\otimes_{\OO_{\ran{I}}}\Omega^1_{Y\times A_K^n(I)/K}\otimes_{\OO_{Y\times A_K^n(I)}}\OO_{\ran{I}}\\
        &= P_i\otimes_{\OO_{Y\times A_K^n(I)}}\Omega^1_{Y\times A_K^n(I)/K}.
    \end{align*}
    Thus the $\nabla$-module structure of $P_i$ on $\ran{I}$ is naturally given by that on $Y\times A_K^n(I)$ via 
    $$ \nabla:P_i\to P_i\otimes_{\OO_{Y\times A_K^n}}\Omega^1_{Y\times A_K^n/K}= P_i\otimes_{\OO_{\ran{I}}}\Omega^1_{\ran{I}/K}. $$
    So we prove that $P_i$ is also an object in $\Rob{X}$ for $i=1,2$.

    When $X,Y$ are not necessarily affinoid, by the same gluing argument as in last paragraph in proof of Proposition \ref{relative pushforward of decomposition}, we can obtain a global decomposition.
\end{proof}

\section{An application}
As an application of Proposition \ref{pushforward of decomposition}, we prove a generalization of \cite{shiho2010logarithmic} Proposition 2.4 (called 
``generization'') in this section. Firstly we introduce notions of log-$\nabla$-modules and $\Sigma$-unipotence defined in \cite{shiho2010logarithmic}. Next we explain the relation between Shiho's ``generization'' proposition and $p$-adic Fuchs theorem for log-$\nabla$-modules over relative polyannuli. Finally we prove $p$-adic Fuchs theorem for log-$\nabla$-modules over relative polyannuli when its exponent has non-Liouville differences.

\subsection{Log-$\nabla$-modules and $\Sigma$-unipotence}
We introduce the notions defined in \cite{shiho2010logarithmic}.
Firstly, let us recall the notion of log-$\nabla$-modules.

    \begin{defn}[{\cite{shiho2010logarithmic}}, Definition 1.2]\label{lognabdef}
        Let $X$ be a rigid spaces over $K$ 
        and let $x_1, ..., x_l$ be elements in $\Gamma(X,\OO_X)$. 
        Then a 
        log-$\nabla$-module on $X$ with respect to $x_1, ..., x_l$ 
        is a locally free $\OO_X$-module 
        $E$ endowed with an integrable $K$-linear log connection 
        $\nabla: E \longrightarrow E \otimes_{\OO_X} \omega^1_{X}$, where $\omega^1_X$ denotes the sheaf of continuous log differentials with respect to $\dlog x_i\ (1\leq i\leq l)$, which is defined by
        $$\omega^1_{X} := (\Omega^1_{X/K} \oplus \bigoplus_{i=1}^l \OO_X \cdot 
        \mathrm{dlog} x_1)/N,$$
        where $N$ is the sheaf locally generated by $(dx_i,0)-(0,x_i\mathrm{dlog} x_i)$ $(1 \leq i \leq l)$.
    \end{defn}

Let $X, x_1,...,x_l,(E,\nabla)$ be as in Definition \ref{lognabdef} and put $D_i:=\{x_i=0\} \subseteq X$ for $1\leq i\leq l$, 
$M_i:={\rm im}(\Omega^1_{X} \oplus \bigoplus_{j\not=i}\OO_X \dlog x_j \lra 
\omega^1_{X})$. Then the composite map 
$$ 
E \overset{\nabla}{\lra} E \otimes_{\OO_X} \omega^1_{X} \lra 
E \otimes_{\OO_X} (\omega^1_{X}/M_i) \cong 
E \otimes_{\OO_X} \OO_{D_i} \dlog x_i \cong E \otimes_{\OO_X} \OO_{D_i}
$$ 
naturally induces an element $\res_i$ in 
$\End_{\OO_{D_i}}(E \otimes_{\OO_X} \OO_{D_i})$. By \cite{shiho2011cut} Proposition-Definition 1.24, if $X$ is smooth affinoid and zero loci of $x_i$ are smooth and meet transversally, then there exists a polynomial with coefficients in $K$ annihilating $\res_i$. We call the roots of the minimal polynomial $P_i(\lambda)\in K[\lambda]$ annihilating $\res_i$ the exponent of $(E,\nabla)$ along $D_i$ for $1\leq i\leq r$. For the case $X$ is not necessarily affinoid, exponent of $(E,\nabla)$ is defined as the union of the exponents of $(E,\nabla)|_U$ where $U$ is an affinoid subdomain of $X$.

For a smooth 
affinoid space $X$ 
endowed with $x_1, ...,x_l \in \Gamma(X,\OO_X)$ 
such that the zero loci of $x_i$ are smooth and 
meet transversally, we define the 
notion of a log-$\nabla$-module $(E,\nabla)$ 
on $X \times A_K^n[0,0]$ with respect to 
$x_1, ..., x_l$ as a 
log-$\nabla$-module $(E,\nabla)$ 
on $X$ with respect to 
$x_1, ..., x_l$ endowed with commuting $\OO_X$-linear endomorphisms $\partial_i := t_i 
\frac{\partial}{\partial t_i}$ of $(E,\nabla)$ 
$(1 \leq i \leq n)$. 

Next we recall the notion of $\Sigma$-unipotent log-$\nabla$-modules. For an aligned polysegment $I \subseteq \RR_{>0}^n$ or $I=[0,0]$, and 
$\xi := (\xi_1, ..., \xi_n) \in K^n$, we define the $\nabla$-module 
$(M_{\xi},\nabla_{M_{\xi}})$ on $A^n_K(I)$
as the $\nabla$-module 
$(\OO_{A_K^n(I)},d + \sum_{i=1}^n \xi_i \frac{dt_i}{t_i})$. 

\begin{defn}[\cite{shiho2010logarithmic} Definition 1.3]\label{unipdef}
    Let $X$ be a smooth rigid space endowed with 
    $x_1, \allowbreak ..., x_l \in \Gamma(X,\OO_X)$ whose zero loci are smooth and 
    meet transversally. Let $I \subseteq \RR^n_{> 0}$ 
    be an aligned polysegment or $I=[0,0]$, and fix $\Sigma := \Sigma'\times\prod_{i=1}^n\Sigma_i \subseteq 
    \overline{K}^{l}\times\ZZ_p^n$. 
    A log-$\nabla$-module $E$ on $X\times A_K^n(I)$ is called $\Sigma$-unipotent if it admits a filtration 
    $$ 0 = E_0 \subset E_1 \subset \cdots \subset E_k=E $$ 
    by sub log-$\nabla$-modules whose successive quotients have the from $\pi_1^*F \otimes \pi_2^*M_{\xi}$ where $F$ is a log-$\nabla$-module on $X$ with respect to $x_1,\dots, x_l$ with exponent belonging to $\Sigma'$, $\xi\in\prod_{i=1}^n\Sigma_i$, and 
    $\pi_1: X \times A_K^n(I) \lra X$, $\pi_2:X \times A_K^n(I) \lra A_K^n(I)$ 
    are the projections. We denote the category of $\Sigma$-unipotent log-$\nabla$-modules over $X\times A_K^n(I)$ with respect to $x_1,\dots,x_l$ by 
    $ \mathrm{ULNM}_{X\times A_K^n(I),\Sigma} $.
\end{defn}

\begin{prop}[\cite{shiho2010logarithmic} Corollary 1.15] \label{1.15}
    Let $X,\Sigma$ be as in Definition \ref{unipdef}, and $I\subset\RR^n_{>0}$ be an open polysegment. 
    Then the functor 
    $\mathcal{U}_I: \allowbreak \mathrm{ULNM}\allowbreak {}_{X\times A^n_K[0,0],\Sigma} \allowbreak \lra 
    \mathrm{ULNM}_{X\times A^n_K(I),\Sigma}$
    sending an object $E$ in $\mathrm{ULNM}_{X\times A^n_K[0,0],\Sigma}$ to 
    $\pi^*E$ $($here we regard $E$ as a log-$\nabla$-module over $X$ with respect to $x_1,\dots,x_l$ and $\pi:X \times A^n_K(I) \lra X$ denotes the projection$)$ 
    endowed with the connection 
    $$ {v} \mapsto \pi^*\nabla({v}) + \sum_{i=1}^n\pi^*(\partial_i)({{v}}) \frac{dt_i}{t_i} $$
    is an equivelence of categories.
    \end{prop}

    \begin{defn}[\cite{shiho2010logarithmic} Definition 1.8]
        A set $\Sigma\subset \overline{K}$ is called $(\NID)$ (resp. $(\NLD)$) if, for any $\alpha,\beta\in \Sigma$, $\alpha-\beta$ is not a non-zero integer (resp. $\alpha-\beta$ is non-Liouville). A set $\Sigma:=\prod_{i=1}^r\Sigma_i\subset\overline{K}^r$ is called $(\NID)$ (resp. $(\NLD)$) if so is $\Sigma_i$ for all $1\leq i\leq r$.
    \end{defn}

We have the following characterization of $\Sigma$-unipotence for $\nabla$-modules on polyannuli under a suitable condition on $\Sigma$:

\begin{lem}[cf. \cite{shiho2011cut} Proposition 1.20]\label{1.20}
    Let $L/K$ be an extension of complete nonarchimedean fields. Let $I\subset \RR^n_{>0}$ be an open polysegment and let $E$ be a $\nabla$-module on $A_L^n(I)$. Then for a subset $\Sigma\subset \ZZ_p^n$ which is $(\NLD)$, the following are equivalent:
    \begin{enumerate}
        \item $E$ is $\Sigma$-unipotent.
        \item $E$ satisfies the Robba condition and there exists an exponent $A$ of $E$ such that $A\subseteq \Sigma$.
    \end{enumerate}
\end{lem}

\begin{proof}
    (2)$\Longrightarrow$(1) follows from \cite{Gac} Th\'eor\`eme on page 216 when $E$ is free and \cite{w1} Corollary 3.20 for general $E$. In rank one case, (1)$\Longrightarrow$(2) follows from \cite{w1} Example 3.16 and the general case follows from the rank one case and Lemma \ref{exact_sequence}.
\end{proof}
\begin{defn}
    For a log-$\nabla$-module $P$ over $\ran{I}$ with respect to $x_1,\dots,x_r\in\Gamma(X,\OO_X)$, we say $P$ satisfies the Robba condition if $P$ satisfies the Robba condition when regarding as a $\nabla$-module over $\ran{I}$ relative to $X$. In this case, we say $A$ is an exponent of $P$ if $A$ is an exponent of $P$ when $P$ is regarded as an object in $\Robb{X}$.
\end{defn}
\begin{lem}[cf. \cite{kedlaya2017corrigendum} and \cite{w1} Lemma 3.9]\label{log uniqueness}
    Suppose $X$ is a connected smooth affinoid space with $x_1, ..., x_l\in\Gamma(X,\OO_X)$ and $I\subset \RR_{>0}^n$ is a closed polysegment. Let $P$ be a log-$\nabla$-module on $\ran{I}$, with exponent $A$ having Liouville partition $A=\AA_1\cup\cdots\cup\AA_k$. The decomposition of $P=P_1\oplus\cdots\oplus P_k$ such that $P_i$ is a log-$\nabla$-module on $\ran{I}$ admitting $\AA_i$ as an exponent, if exists, is unique.
\end{lem}
\begin{proof}
    The proof is the same as that of Lemma \ref{uniqueness}.
\end{proof}

As a generalization of Lemma \ref{1.20}, we have the following:
\begin{prop}\label{generalize 1.20}
    Suppose $I\subset \RR^n_{>0}$ be an open polysegment and $X$ is a smooth affinoid space with one-point Shilov boundary $\{\eta\}$ endowed with $x_1,\dots,x_l\in \Gamma(X,\OO_X)$ whose zero loci are smooth and meet transversally. Let $P$ be a log-$\nabla$-module over $X\times A_K^n(I)$ with respect to $x_1,\dots,x_l$. Then for a subset $\Sigma=\Sigma'\times\prod_{i=1}^n\Sigma_i\subset \overline{K}^l\times \ZZ_p^n$ which is $(\NID)$ and $(\NLD)$, the following are equivalent:
    \begin{enumerate}
        \item $P$ is $\Sigma$-unipotent.
        \item $P$ satisfies the Robba condition and there exists an exponent $A$ of $P$ such that $A\subseteq \Sigma$.
    \end{enumerate}
\end{prop}
\begin{proof}
    (2)$\Longrightarrow$(1): In this case $P|_{\eta\times A_K^n(I)}$ satisfies the Robba condition admitting $A$ as an exponent. By Lemma \ref{1.20}, $P|_{\eta\times A_K^n(I)}$ is $\Sigma$-unipotent and so $P$ is $\Sigma$-unipotent by \cite{shiho2010logarithmic} Proposition 2.4. (1)$\Longrightarrow$(2): Rank $1$ case follows from applying \cite{w1} Example 3.16 on $P|_{x\times A_K^n(I)}$ for every $x\in X$ and the general case follows from the rank $1$ case and Lemma \ref{relative_exact_sequence}.
\end{proof}

A direct consequence of  \cite{shiho2010logarithmic} Proposition 2.4 to our setting is the following:
\begin{prop}[cf. \cite{shiho2010logarithmic} Proposition 2.4]\label{decomposition for one-point}
    Suppose $I\subset \RR^n_{> 0}$ is an open polysegment and $X$ is a smooth affinoid space with one-point Shilov boundary $\{\eta\}$ endowed with $x_1,\dots,x_l\in \Gamma(X,\OO_X)$ whose zero loci are smooth and meet transversally. Let $P$ be a log-$\nabla$-module over $X\times A_K^n(I)$ with respect to $x_1,\dots,x_l$ of rank $r$ satisfying the Robba condition with exponent $A$. If $A$ has non-Liouville differences and the set $\Sigma'\subset \overline{K}^l$ of exponent of $P$ along $D_i$ is $(\NID)$ and $(\NLD)$, then there exists a unique direct sum decomposition $P=\oplus_{\lambda\in(\ZZ_p/\ZZ)^n}P_\lambda$, where each $P_\lambda$ has exponent identically equal to $\lambda$.
\end{prop}
\begin{proof}
    We proceed by induction on $n$. Let $\AA_1,\dots,\AA_k$ be the partition of $A$ into $\ZZ$-cosets in the $n$-th direction and let $A^n=\AA_1^n\cup\cdots\cup\AA_k^n $ denote the corresponding partition of the factors in the $n$-th direction. Regard $X\times A_K^n(I)$ as $(X\times A_K^{n-1}(I'))\times A_K^1(I_n)$ where $I=I'\times I_n$. Take $a_i\in \AA_i^n$ for each $1\leq i\leq k$ and set $\Sigma_n:=\{a_1,\dots,a_k\}$. Then $\Sigma=\Sigma'\times\Sigma_n$ is $(\NID)$ and $(\NLD)$. Let $P_\eta$ be the restriction of $P$ on $(\eta\times A^{n-1}_K(I'))\times A_K^1(I_n)$. Since $P_\eta$ satisfies the Robba condtion with exponent $A$, $P_\eta$ is $\Sigma$-unipotent by Lemma \ref{1.20}. Moreover, the supremum norm on $\Gamma(\eta\times A_K^{n-1}(I'),\OO_{\eta\times A_K^{n-1}(I')})$ restricts to the supremum norm on $\Gamma(X\times A_K^{n-1}(I'),\OO_{X\times A_K^{n-1}(I')})$. Thus, with the same proof as that of \cite{shiho2010logarithmic} Proposition 2.4 (1), we can show that $P$ is $\Sigma$-unipotent. By Proposition \ref{1.15}, there exists $P_0\in\mathrm{ULNM}_{(X\times A_K^{n-1}(I'))\times A_K[0,0],\Sigma}$ such that $\mathcal{U}_{I_n}(P_0)=P$. As a log-$\nabla$-module on $X\times A_K^{n-1}(I')$, $P_0$ has an exponent $A'$ which is the multisubset of $\ZZ_p^{n-1}$ obtained from $A$ by removing its factors in the $n$-th direction. Since $A$ has non-Liouville differences, so does $A'$ and by induction hypothesis, we have
    \begin{equation*}
        P_0=\bigoplus_{\lambda'=(\lambda'_1,\dots,\lambda'_{n-1})\in(\ZZ_p/\ZZ)^{n-1}}P_{0,\lambda'}\tag{*},
    \end{equation*}
    where $P_{0,\lambda'}$ has exponent identically equal to $\lambda'$. Since the endomorphism $t_n\partial_{t_n}$ preserves the above decomposition, (*) is a decomposition in $\mathrm{ULNM}_{(X\times A_K^{n-1}(I'))\times A_K[0,0],\Sigma}$. By \cite{baldassarri1993formal} Corollary 1.8.3, we have the following decomposition of $P_{0,\lambda'}$ for each $\lambda'$:
    $$ P_{0,\lambda'}=\bigoplus_{\lambda_n\in\Sigma_n} P_{0,\lambda',\lambda_n}, $$
    where $P_{0,\lambda',\lambda_n}=\ker(t_n\partial_{t_n}-\lambda_n)^r$. This implies that $P_{0,\lambda',\lambda_n}$ is $\Sigma'\times \{\lambda_n\}$-unipotent. Thus,
    $$ P=\mathcal{U}_{I_n}(P_0)=\bigoplus_{(\lambda',\lambda_n)\in\ZZ_p^n}\mathcal{U}_{I_n}(P_{0,\lambda',\lambda_n}). $$
    Now we prove that this decomposition corresponds to the correct exponent. For $1\leq r\leq n-1$, the exponent of $\mathcal{U}_{I_n}(P_{0,\lambda',\lambda_n})$ in the $r$-th direction is $\lambda'_r$ since it is so for $P_{0,\lambda',\lambda_n}$. For thr $n$-th direction, choose a point $x'$ of $X\times A_K^{n-1}(I')$, then $\bigoplus_{(\lambda',\lambda_n)\in\ZZ_p^n}\mathcal{U}_{I_n}(P_{0,\lambda',\lambda_n})|_{x'\times A_K^1(I_n)}$ is a decomposition of $P|_{ x'\times A_K^1(I_n)}$. Since $\mathcal{U}_{I_n}(P_{0,\lambda',\lambda_n})|_{x'\times A_K^1(I_n)}$ is also $\Sigma'\times\{\lambda_n\}$-unipotent, by Lemma \ref{1.20}, $\mathcal{U}_{I_n}(P_{0,\lambda',\lambda_n})|_{x'\times A_K^1(I_n)}$ and hence $\mathcal{U}_{I_n}(P_{0,\lambda',\lambda_n})$ has the correct exponent in the $n$-th direction. Uniqueness follows from Lemma \ref{log uniqueness}.
\end{proof}

\subsection{Local existence of finite \'etale cover to the unit polydisc}
In this subsection, we prove that any smooth rigid space or a smooth dagger space is locally a finite \'etale cover of the unit polydisc.
    
    \begin{prop}[cf. \cite{achinger2017wild} Proposition 6.6.1]\label{finite etale}
    Suppose that $K$ is algebraically closed and $R$ is a reduced $K$-affinoid algebra. Let $g:K\la s_1,\dots, s_m\ra \to R$ be an \'etale morphism. Then there exists a finite \'etale morphism $f:K\la s_1,\dots, s_m \ra \to R$.
    \end{prop}
    
    \begin{proof}
    This proof is essentially that of \cite{achinger2017wild} Proposition 6.6.1.
    Let $x_1,\dots,x_m$ be the image of $s_1,\dots,s_m$ in $R^\circ$. Since $K$ is algebraically closed, by \cite{BGR} Theorem 6.4.3/1 and Corollary 6.4.3/6, $R^\circ$ is topologically of finite type over $\OO_K$. Thus we can choose $x_{m+1},\dots,x_{r}\in R^\circ$ such that $x_1,\dots,x_r$ form a set of topological generators of $R^\circ$ over $\OO_K$. This gives rise to a presentation 
    $$ R^\circ\cong \OO_K\langle s_1,\dots s_m,s_{m+1},\dots,s_r \rangle/I. $$
    Let $\Omega_g:=\widehat{\Omega}_{R^\circ/\OO_K\langle s_1,\dots,s_m \rangle}$ which is a finite $R^\circ$ module. Since $g$ is \'etale, $\Omega_g$ is killed by some power of $p$, say $p^k$. By \cite{achinger2015} Proposition 5.4 for $\wt{g^*}:\kappa_{K}[s_1,\dots,s_m]\to \wt{R}$ and $N=p^{k+1}$, there exist elements $\overline{y_1},\dots,\overline{y_m}$ in the subring of $\wt{R}$ generated by $N$-th power of elements such that the map
    \begin{align*}
    \overline{f^*}:\kappa_{K}[s_1,\dots,s_m]& \longrightarrow \wt{R}\\
    s_i & \mapsto \overline{x_i}+\overline{y_i}
    \end{align*}
    is finite. Let $y_1,\dots,y_m$ be any lift of $\overline{y_1},\dots,\overline{y_m}$ to $R$ which are polynomials with $\ZZ$-coefficients in $p^{k+1}$-powers of elements in $R^\circ$. Then the map 
    \begin{align*}
    f:K\langle s_1,\dots,s_m \rangle & \longrightarrow R \\
    s_i & \mapsto x_i+y_i
    \end{align*}
    is also finite since its reduction $\overline{f}$ is finite, by \cite{BGR} Theorem 6.3.5/1. Since $f$ is a finite morphism between equidimensional smooth rigid spaces of the same dimension, flatness of $f$ follows from the ``miracle flatness'' theorem (\cite{stacksproject} Lemma 00R4).
    
    Define $\Omega_f$ in the same way as $\Omega_g$ for the map $f$ and consider following presentations:
    \begin{align*}
        (R^\circ)^m\overset{dg}{\longrightarrow}\Omega_{R^\circ/\OO_K}\longrightarrow\Omega_g\longrightarrow 0,\\
        (R^\circ)^m\overset{df}{\longrightarrow}\Omega_{R^\circ/\OO_K}\longrightarrow\Omega_f\longrightarrow 0.
    \end{align*}
    After tensoring with $\OO_K/p^{k+1}$, we have the following exact sequences:
    \begin{align*}
        (R^\circ/p^{k+1})^m\overset{dg}{\longrightarrow}\Omega_{R^\circ/p^{k+1}\big{/}\OO_K/p^{k+1}}\longrightarrow\Omega_g/p^{k+1}\Omega_g\longrightarrow 0,\\
        (R^\circ/p^{k+1})^m\overset{df}{\longrightarrow}\Omega_{R^\circ/p^{k+1}\big{/}\OO_K/p^{k+1}}\longrightarrow\Omega_f/p^{k+1}\Omega_f\longrightarrow 0.
    \end{align*}
    By the assumption on each $y_i$, we have $dy_i\in p^{k+1}\Omega_{R^\circ/\OO_K}$ and so $dg,df$ in the above diagram are the same. Hence $\Omega_f/p^{k+1}\Omega_f=\Omega_g/p^{k+1}\Omega_g$. Since $p^k\Omega_g=0$, we have
    $$ p^k(\Omega_f/p^{k+1}\Omega_f)=p^k(\Omega_g/p^{k+1}\Omega_g)=0. $$
    This implies that $p^k\Omega_f=p^{k+1}\Omega_f$. Since $f$ is finite, $f^\circ$ is also finite by \cite{BGR} Corollary 6.4.1/5 and so $p^k\Omega_f$ is a finite $R^\circ$-module. By \cite{BGR} Lemma 1.2.4/6, we have $p^k\Omega_f=0$ and so $f$ is \'etale.
    \end{proof}

    \begin{rmk}
        Proposition \ref{finite etale} also holds when $K$ is a discrete valuation field by \cite{achinger2017wild} Proposition 6.6.1. But we are not able to prove the claim when $K$ is neither discrete nor algebraically closed. Because in these cases, $R^\circ$ may not be topologically of finite type over $\OO_K$, which leads the above proof to fail.
        \end{rmk}
    
\begin{cor}\label{existence_of_finite_etale_morphism}
    Let $X$ be a smooth rigid space or a dagger space over an algebraically closed field $K$. Then, for every point $x\in X$ there exists an affinoid neighborhood $U$ of $x$ and a finite \'etale morphism $U\to \BBB^m$.
\end{cor}
\begin{proof}
    We firstly deal with the case where $X$ is a rigid space. By \cite{ayoub2015motifs} Corollaire 1.1.51, for any $x\in X$, there exists an affinoid neighborhood $U$ of $x$ and an \'etale morphism $U\to \BBB^m$. Then the existence of finite \'etale morphism $U\to \BBB^m$ follows from Proposition \ref{finite etale}. When $X$ is a dagger space, we may reduce to the case where $X$ is affinoid dagger. Since the completion functor from the category of dagger spaces to the category of rigid spaces is faithful by \cite{Gro} Theorem 2.19, the conclusion follows from \cite{vezzani2018monsky} Proposition 2.15 and the rigid case.
\end{proof}

\begin{rmk}\label{extension of finite etale morphisms}
    In the proof of \cite{vezzani2018monsky} Proposition 2.15, the author actually proved a stronger conclusion that any finite \'etale morphism between completion of dagger spaces can be extended to their fringe spaces.
\end{rmk}

\subsection{Galois descent}
Let $R$ be a $K$-affinoid or dagger algebra. In this subsection, we firstly prove that the $K$-subalgebra of $R\cotimes_K \overline{K}$ fixed by the Galois group of $K$ is exactly $R$, thanks to the existence of suitable Schauder basis for $R$ as a $K$-vector space. After that, we prove a Galois descent argument for (generalized) $p$-adic Fuchs theorem.
\begin{defn}
    Let $I$ be an index set of at most countable cardinality and $V$ be a normed $K$ vector space. A system $\{v_i\}_{i\in I}$ is called a topological generating system of $V$ if every $v\in V$ can be written as a convergent series $v=\sum_{i\in I} c_iv_i$ for $c_i\in K$. It is called a Schauder basis of $V$ if coefficients $c_i$ are uniquely determined by $v$.
    \end{defn}
    
    \begin{defn}
    We say a normed $K$-vector space $V$ is of countable type, if $V$ contains a dense $K$-linear subspace at most countable dimension.
    \end{defn}
    
    In fact, every $K$-affinoid or $K$-dagger algebra is a normed $K$-vector space of countable type. For example, suppose $R$ is a $K$-affinoid algebra. Then there is a surjective map $\pi:K\langle s \rangle\twoheadrightarrow R$. If we denote $\pi(s_i)$ by $f_i$ (this is called an affinoid generating system of $R$ over $K$) for $1\leq i\leq m$, then $K[f_1,\dots,f_m]$ is a dense $K$-subspace of $R$ of at most countable dimension. The same argument holds for $K$-dagger algebras.
    
    \begin{lem}\label{fixed elements}
    Suppose that $R$ is a $K$-affinoid algebra (resp. a $K$-dagger algebra) and $G=\Gal(\overline{K}/K)$ is the absolute Galois group of $K$. Let $\wh{R}=R\widehat{\otimes}_K \overline{K}$ (resp. $\wh{R}=R\dagotimes_K \overline{K}$). Then $G$ acts on $\wh{R}$ and the subring $\wh{R}^G$ of $\wh{R}$ fixed by $G$ is exactly $R$.
    \end{lem}
    \begin{proof}
    Choose an affinoid generating system $f_1,\dots,f_r\in R$. Then $K[f_1,\dots,f_r]$ is dense in $R$ and so $\calg{K}[f_1,\dots,f_r]$ is dense in $\wh{R}$. Now choose an $\alpha$-cartesian $K$-basis (see \cite{BGR} Definition 2.6.1/3 for definition) $\{w_i\}_{i\in I}$ of $K[f_1,\dots,f_r]$ (existence of such a basis is guaranteed by \cite{BGR} Proposition 2.3.3/4 and \cite{BGR} Theorem 2.6.2/4). Since by \cite{Ked1} Lemma 1.3.11, $\calg{K}\otimes_K R\to \wh{R}$ is injective, the upper horizontal arrow $i$ of the following commutative diagram is injective:
    \begin{equation*}
        \begin{tikzcd}
            \calg{K}\otimes_K K[f_1,\dots,f_r] \arrow{r}{i} \arrow[swap]{d}{} & \calg{K}[f_1,\dots,f_r] \arrow{d}{} \\%
            \calg{K}\otimes_K R \arrow{r}{}& \wh{R}.
        \end{tikzcd} 
    \end{equation*}
    Thus $i$ is bijective and $\{w_i\}_{i\in I}$ also form an $\alpha$-cartesian $\calg{K}$-basis of $\calg{K}[f_1,\dots,f_r]$. Then $\{w_i\}_{i\in I}$ is a $\calg{K}$-Schauder basis of $\wh{R}$ by \cite{BGR} Proposition 2.7.2/3. Hence any $x\in \wh{R}$ can be uniquely written as a convergent power series $x=\sum_{i\in I}c_iw_i$ with $c_i\in\calg{K}$ and $\sigma\in G $ act on $x$ by 
    $$\sigma(x):=\sum_{i\in I}\sigma(c_i)w_i.$$
    So $x$ is fixed by $G$ if and only if $c_i\in \calg{K}$ is fixed by $G$, which implies that $c_i\in K$ for all $i\in I$ by \cite{robba} Theorem 1.1.
    \end{proof}
Now we prove a Galois descent argument to relax the ``algebraically closed'' condition put on $K$ when we prove the (generalized) $p$-adic Fuchs theorem.

\begin{prop}\label{Galois descent}
Suppose $I\subset \RR_{>0}^n$ is a polysegment and $X$ is a connected $K$-rigid space (resp. dagger space). Let $P$ be one of the following:
\begin{enumerate}
    \item an object in $\Robb{X}$.
    \item an object in $\Rob{X}$.
    \item a log-$\nabla$-module over $\ran{I}$ with respect to some $x_1,\dots,x_l\in\Gamma(X,\OO_X)$ satisfying the Robba condition, where $X$ is smooth, zero loci of $x_1,\dots,x_l$ are smooth and they meet transversally.
\end{enumerate}
Let $A$ be an exponent of $P$ with Liouville partition $\AA_1\cup\cdots\cup\AA_k$. If the direct sum decomposition of $\wh{P}:=P\cotimes_K \overline{K}$ (resp. $\wh{P}:=P\otimes^\dagger_K \overline{K}$) into $Q_1\oplus\cdots\oplus Q_k$ such that $\AA_i$ is an  exponent of $Q_i$ for $1\leq i\leq k$ exists, 
then there exists a unique direct sum decomposition $P=P_1\oplus\dots\oplus P_k$ such that $\AA_i$ is an exponent of $P_i$ for $1\leq i\leq k$.
\end{prop}
\begin{proof}
We prove the assertion for rigid spaces, and the proof for dagger spaces is the same.
Firstly we work with the case where $X=\Sp(R)$ is affinoid and $I$ is closed. In this case we fix a basis $e_1,\dots,e_r$ of $P$.
Take any $\sigma\in\Gal(\overline{K}/K)$. Then $\sigma$ acts on $\wh{P}$ and it commutes with derivations. So $\sigma(Q_i)$ is also a (log-)$\nabla$-module. It satisfies the Robba condition because $\sigma:Q_i\to\sigma(Q_i)$ is an isometric isomorphism. To verify that $\AA_i$ is an exponent of $\sigma(Q_i)$, we can assume it is free by Lemma 
\ref{localQS}. Let $e'_1,\dots,e'_l$ be a basis of $Q_i$ and $\{S_{k,\AA_i}\}_{k=0}^\infty$ in $\Mat_{l\times l}(\wh{R}_I)$ be the sequence of matrices defining the exponent $\AA_i$. For each $k\geq 0$, choose $\zeta\in\mupk^n$, then we have
\begin{align*}
    \zeta^*[(\sigma(e'_1),\dots,\sigma(e'_l))\sigma(S_{k,\AA_i})]
    &= \left(\sigma(\sigma^{-1}(\zeta))\right)^*[(\sigma(e'_1),\dots,\sigma(e'_l))\sigma(S_{k,\AA_i})]\\
    &=\sigma\left(\sigma^{-1}(\zeta)^*[(e'_1,\dots,e'_l)S_{k,\AA_i}]\right)\\
    &=\sigma\left([(e'_1,\dots,e'_l)S_{k,\AA_i}](\sigma^{-1}(\zeta))^{\AA_i}\right)\\
    &=\left(\sigma(e'_1),\dots,\sigma(e'_l)\right)\sigma(S_{k,\AA_i})\zeta^{\AA_i}
\end{align*}
and so $\{\sigma(S_{k,\AA_i})\}_{k=0}^\infty$ defines $\AA_i$ as an exponent of $\sigma(Q_i)$. Note that the uniqueness of the decomposition of $\wh{P}$ with respect to the Liouville partition of an exponent follows from Lemma \ref{uniqueness} for the case of relative $\nabla$-modules, from Lemma \ref{nabla uniqueness} for the case of absolute $\nabla$-modules and from Lemma \ref{log uniqueness} for log-$\nabla$-modules. Thus, since we have
$$ \wh{P}=\sigma(\wh{P})=\sigma(Q_1)\oplus\cdots\oplus\sigma(Q_k), $$
we obtain $\sigma(Q_i)=Q_i$ for all $1\leq i\leq k$.

Let $q_i:\wh{P}\to Q_i$ be the natural projection. 
\begin{claim}
    $\sigma\circ q_i=q_i\circ\sigma$.
\end{claim}
We prove the claim. Choose $v\in \wh{P}$, $v$ can be uniquely written as $v_1+\cdots+v_k$ with $v_i\in Q_i$. Then
\begin{align*}
    q_i(\sigma(v))=q_i\left(\sigma\left(\sum_{i=1}^kv_i\right)\right)=q_i\left(\sum_{i=1}^k\sigma(v_i)\right)=\sigma(v_i)=\sigma(q_i(v)).
\end{align*}
The third equality holds because by the above argument $\sigma(v_i)\in Q_i$. So the claim is proved.
Now set $P_i:=q_i(P)$. For $v\in P_i$, choose $w\in P$ such that $v=q_i(w)$, then 
$$\sigma(v)=\sigma(q_i(w))=q_i(\sigma(w))=q_i(w)=v.$$
Thus $P_i$ is fixed by $\Gal(\overline{K}/K)$ and so $P_i\subseteq Q_i^{\Gal(\overline{K}/K)}$.
Conversely, if an element $v\in Q_i$ is fixed by $\Gal(\overline{K}/K)$, as an element of $\wh{P}$, write $v=\sum_{i=1}^r l_ie_i$ with $l_i\in \wh{R}_I$. Then each $l_i$ should be fixed by $\Gal(\overline{K}/K)$ and so $l_i\in R_I$ by Lemma \ref{fixed elements}. Consequently, $v\in P\cap Q_i$ and in particular, $v=q_i(v)\in P_i$. This implies that $Q_i^{\Gal(\overline{K}/K)}\subseteq P_i$. Hence $P_i=Q_i^{\Gal(\overline{K}/K)}$ and so $P=\wh{P}^{\Gal(\overline{K}/K)}=\bigoplus_{i=1}^k Q_i^{\Gal(\overline{K}/K)}=\bigoplus_{i=1}^k P_i$. Moreover, if we put $\wh{P}_i=P_i\cotimes_K \overline{K}$, $0=\wh{P}/\wh{P}=\bigoplus_{i=1}^k Q_i/\wh{P}_i$ and so we have $\wh{P}_i=Q_i$. Since exponent does not change after extension of constant field, $\AA_i$ is an exponent of $P_i$. Uniqueness follows from Lemma \ref{uniqueness} for the case of relative $\nabla$-modules, from Lemma \ref{nabla uniqueness} for the case of $\nabla$-modules and from Lemma \ref{log uniqueness} for the case of log-$\nabla$-modules. When $X$ is not necessarily affinoid and $I$ is not necessarily closed, we firstly obtain the decomposition locally. Then, by the same gluing argument as in the last paragraph of the proof of Proposition \ref{relative pushforward of decomposition}, we glue them to obtain a global decomposition.
\end{proof}

\subsection{The $p$-adic Fuchs theorem}
As a generalization of Proposition \ref{decomposition for one-point}, we have the following theorem, in which we no longer assume that $X$ has one-point Shilov boundary.
\begin{thm}[cf. \cite{shiho2010logarithmic} Proposition 2.4]
    Suppose $I\subset \RR^n_{> 0}$ is an open polysegment.
    Let $X$ be a connected smooth rigid space. Endow $X$ with $x_1,\dots,x_l\in \Gamma(X,\OO_X)$ whose zero loci are smooth and meet transversally. Let $P$ be a log-$\nabla$-module over $X\times A_K^n(I)$ with respect to $x_1,\dots, x_l$ satisfying the Robba condition with exponent $A$. If $A$ has non-Liouville differences, then there exists a unique direct sum decomposition of $P=\oplus_{\lambda\in(\ZZ_p/\ZZ)^n}P_\lambda$ where each $P_\lambda$ has exponent identically equal to $\lambda$.
\end{thm}
\begin{proof}
Firstly we consider the case where $K$ is algebraically closed.
Let $D_i:=\{x_i=0\}$ for $1\leq i\leq l$ and $D=\bigcup_{i=1}^l D_i$. Set $X_0=X-D$, then $P|_{X_0\times A^n_K(I)}$ is a $\nabla$-module over $X_0\times A^n_K(I)$. By Corollary \ref{existence_of_finite_etale_morphism}, for each $x\in X_0$, there exists an affinoid neighborhood $U$ of $x$ with a finite \'etale morphism $f:U\to \BBB^m$. This induces a finite \'etale morphism $f_I:U\times A^n_K(I)\to \BBB^m\times A^n_K(I)$. Since $\BBB^m$ admits a one-point Shilov boundary, by Proposition \ref{decomposition for one-point}, $f_{I*}(P|_{U\times A^n_K(I)})$ admits a decomposition $\bigoplus_{\lambda\in (\ZZ_p/\ZZ)^n}P_{U,\lambda}$, where $P_{U,\lambda}$ has exponent identically equal to $\lambda$. By Proposition \ref{pushforward of decomposition}, this is also a decomposition of $P|_{U\times A^n_K(I)}$. By the same gluing argument as the last paragraph in the proof of Proposition \ref{pushforward of decomposition}, we obtain the decomposition
$$P|_{X_0\times A^n_K(I)}=\bigoplus_{\lambda\in (\ZZ_p/\ZZ)^n}P_{X_0,\lambda},$$
where $P_{X_0,\lambda}$ has exponent identically equal to $\lambda$. By \cite{shiho2010logarithmic} Proposition 1.18, $P_{X_0,\lambda}$ extends to a sub log-$\nabla$-module $P_\lambda$ of $P$. By \cite{shiho2010logarithmic} Proposition 1.11 and comparing rank, $P/\bigoplus_{\lambda\in (\ZZ_p/\ZZ)^n}P_{\lambda}$ is locally free of rank $0$, and hence a zero module. Thus $P=\bigoplus_{\lambda\in (\ZZ_p/\ZZ)^n}P_{\lambda}$ is the decomposition of $P$. Since exponent of $P_\lambda$ can be verified by restricting to $x\times A_K^n(I)$ for some $x\in X_0\subset X$, it confirms that each $P_\lambda$ has the correct exponent. The uniqueness follows from Lemma \ref{log uniqueness}.

For the case where $K$ is not algebraically closed, we firstly base change to its completed algebraic closure. Then we do obtain the decomposition by the argument above and applying Proposition \ref{Galois descent}.
\end{proof}
\begin{rmk}
    In a recent talk by Christol, he announced that the $p$-adic Fuchs theorem for free $\nabla$-modules over $X\times A_K^1(I)$ relative to $X$ is true when Shilov boundary of $X$ is the singleton. He suggested briefly only the hint of proof for zero monodromy case. This result is stronger than ours in the sense that it does not require the base to carry a differential structure. However, we can work with a more general base (whose Shilov boudary need not be the singleton), polyannulus, and deal with locally free (not necessarily free) $\nabla$-modules. We expect that, if we assume this result of Christol, it could be generalized by the ``pushforward'' technique developed in this paper to the case where the base is smooth and its Shilov boundary is not necessarily the singleton.
\end{rmk}

\section{Two generalized $p$-adic Fuchs theorems for modules with (relative) connections}

In the begining of this section, we prove four lemmas and recall the $p$-adic Birkhoff theorem by Christol, developed in \cite{christol1979decomposition} and generalized in \cite{christollecturenotes}.  We will replace some of his notations to maintain consistency of the notations used in this paper. Throughout this section, we assume that $m=1$, which denotes the dimension of the base.

\begin{lem}\label{X=0}
    Let $A\in\mathrm{Mat}_{r\times r}(K)$ and $n\in\ZZ-\{0\}$. Assume that the set of eigenvalues of $A$ is $(\NID)$, then the matrix equation $AX+nX=XA$ has the unique solution $X=O$.
\end{lem}

\begin{proof}
    Let $\lambda_1,\dots,\lambda_r\in \overline{K}$ be the eigenvalues of $A$ and view the map $X\mapsto XA-AX-nX$ as a linear transformation on the space  $\mathrm{Mat}_{r\times r}(K)$. By \cite{Ked1} Lemma 7.3.5, eigenvalues of this map are the differences $\lambda_i-\lambda_j-n$ with $1\leq i,j\leq r$. Since the set of eigenvalues of $A$ is $(\NLD)$, this map is invertible and so $X=O$.
\end{proof}
\begin{lem}\label{C=0}
    Suppose $R$ is a $K$-algebra and $\lambda_1,\dots,\lambda_k\in K$ are distinct numbers. Let $A,B,C\in\Mat_{r\times r}(R)$ be of the following form 
    \begin{equation*}
        A=
        \begin{pmatrix*}
          \lambda_1 I_{n_1\times n_1}& & \\
          &\ddots & \\
          & & \lambda_k I_{n_k\times n_k}  
        \end{pmatrix*}, C=
        \begin{pmatrix*}
            C_{1,n_1\times n_1}& & \\
            &\ddots&\\
            & & C_{k,n_k\times n_k}\\
        \end{pmatrix*}.
    \end{equation*}
    If $AB+C=BA$, then $C=O$ and $B$ has the form 
    $$\begin{pmatrix*}
        B_{1,n_1\times n_1}& & \\
        &\ddots&\\
        & & B_{k,n_k\times n_k}\\
    \end{pmatrix*}.$$
\end{lem}
\begin{proof}
    Write 
    $$B=\begin{pmatrix*}
       B_{11}&\cdots&B_{1k}\\
       \vdots & & \vdots\\
       B_{k1}&\cdots&B_{kk}
    \end{pmatrix*},
    $$
    where $B_{ij}$ has size $n_i\times n_j$. By considering the equality $AB+C=BA$ blockwise, we obtain
    $\lambda_i B_{ii}+C_i=\lambda_i B_{ii}$ and $\lambda_i B_{ij}=\lambda_j B_{ij}$ for $i\neq j$. Thus $C_i=O$ for $1\leq i\leq k$ and $B_{ij}=O$ for $i\neq j$.
\end{proof}

\begin{defn}
    Define the ring of analytic elements $K[[s/\beta]]_{\an}$ on open disc $|s|<\beta$ as the completion with respect to $\beta$-Gauss norm of the subring of $K(s)$ consisting of rational functions with no poles in the disc $|s|<\beta$. Define the ring of bounded elements $K[[s/\beta]]_0$ on the open disc $|s|<\beta$ as follows:
    $$ K[[s/\beta]]_0=\left\{\sum_{i=0}^\infty c_is^i\in K[[s]]:\sup_{i}\{|c_i|\beta^i\}<\infty \right\}. $$
\end{defn}
It worth noticing that when $\beta\notin \sqrt{|K^\times|}$, $K[[s/\beta]]_{\an}=K[[s/\beta]]_0=K\la s/\beta\ra$ and that for any $0<\beta'<\beta$, $K\la s/\beta\ra\subseteq K[[s/\beta]]_{\an}\subseteq K[[s/\beta]]_0\subset K\la s/\beta'\ra$.

\begin{lem}\label{intersection}
    Let $I=[\alpha,\beta]\subset \RR_{>0}^n$ be a polysegment, $\rho\in\RR_{>0}$ and $\rho'\in I$. As subalgebras of $K(s,t)_{(\rho,\rho')}$,
    $$ K[[s/\rho]]_\an\la \rho'/t,t/\rho'\ra\cap K(s)_\rho\la \alpha/t,t/\beta\ra=K[[s/\rho]]_\an\la\alpha/t,t/\beta\ra. $$
\end{lem}
\begin{proof}
    It is obvious that the right hand side is contained in the left hand side. Conversely, let $f$ be an element of the left hand side and write $f=\sum_{i\in\ZZ^n}f_it^i$ where $f_i\in K(s)_\rho$. Let $R$ denote the subring of $K(s)$ consisting of rational funtions with no poles in the disc $|s|<\rho$. Since $R[t^{-1},t]$ is dense in $K[[s/\rho]]_\an\la \rho'/t,t/\rho'\ra$, there exists a sequence $\{g_k\}_{k=1}^\infty\subset R[t^{-1},t]$ such that $\lim_{k\to \infty}|f-g_k|_{(\rho,\rho')}=0$. Write $g_k=\sum_{i\in\ZZ^n}g_{ki}t^i$ with $g_{ki}\in R$, then for every $l\in\ZZ^n$, we have
    $$ |f_l-g_{kl}|_\rho\rho'^l\leq  \max_{i\in\ZZ^n}|f_i-g_{ki}|_\rho\rho'^i=|f-g_k|_{(\rho,\rho')}\to 0\text{ , as } k\to\infty.$$
    This implies $|f_l-g_{kl}|_\rho\to 0$ as $k\to\infty$ and so $f_l\in K[[s/\rho]]_\an$.
\end{proof}

\begin{lem}\label{intersection is a field}
    Let $F/K$ be an extension of complete nonarchimedean fields and $\rho\in\RR_{>0}$. Then as subalgebras of $F(t)_\rho$,
    $$F\la \rho/s \ra \cap F[[s/\rho]]_{\an}=F.$$
\end{lem}
\begin{proof}
    It is obvious that $F\subseteq F\la \rho/s \ra \cap F[[s/\rho]]_{\an}$. Let $\mathfrak{a}\subset F[[s/\rho]]_{\an}$ be the ideal generated by $s$. This is a maximum ideal of $F[[s/\rho]]_{\an}$ and $F[[s/\rho]]_{\an}/\mathfrak{a}=F$. Thus
    $$ (F\la \rho/s \ra\cap F[[s/\rho]]_{\an})/(\mathfrak{a}\cap F\la \rho/s \ra\cap F[[s/\rho]]_{\an})\hookrightarrow F .$$
    Since $\mathfrak{a}\cap F\la \rho/s \ra \cap F[[s/\rho]]_{\an}\subseteq \mathfrak{a}\cap F\la \rho/s \ra=(0)$, $F\la \rho/s \ra \cap F[[s/\rho]]_{\an}\subseteq F$.
\end{proof}

The following theorem is essentially \cite{christollecturenotes} Th\'eor\`eme 19.18.
\begin{thm}[cf. \cite{christollecturenotes}, Th\'eor\`eme 19.18]\label{decomposition of matrix}
   Let $\rho\in\RR_{>0}$ and $\Omega/K$ an extension of complete nonarchimedean fields. For any matrix $H\in \GL_r(\Omega(s)_\rho)$ there exists a decomposition $H=s^NLM$ satisfying following properties:
   \begin{enumerate}    
    \item $N$ is a diagonal matrix with entries in $\ZZ$.
    \item $L\in \GL_r(\Omega\la\rho/s \ra)$ and $L(\infty)=I$ where $I$ denotes the identity matrix and $L(\infty)$ denotes the constant term of $L$.
    \item $M\in\GL_r(\Omega[[ s/\rho ]]_\an)$.
   \end{enumerate}
\end{thm}
\begin{proof}
    By the proof of \cite{christollecturenotes} Th\'eor\`eme 19.18, we can choose the following sequences of matrices:
    \begin{enumerate}
        \item[\textbullet] $\{N_k\}_{k=1}^\infty$ which is a sequence of $r\times r$ diagonal matrices with entries in $\ZZ$ and converges to some $N$.
        \item[\textbullet] $\{M_k\}_{k=1}^\infty\subset \GL_r(\Omega(s)\cap\Omega[[s/\rho]]_0)$ which converges to some invertible matrix $M$.
        \item[\textbullet] $\{L_k\}_{k=1}^\infty\subset\GL_r(\Omega(s)\cap\Omega\la \rho/s\ra)$ which converges to some invertible matrix $L$ and satisfies $L_k(\infty)=I$.
        \item[\textbullet] $\{H_k\}_{k=1}^\infty\subset \GL_r(\Omega(s))$ which converges to $H$ and satisfies $H_k=s^{N_k}L_kM_k$.
    \end{enumerate}
    It is easy to see that $N,L,M$ satisfy the desired properties in the statement of the theorem. Note that the difference from the proof of \cite{christollecturenotes} Th\'eor\`eme 19.18 to the proof here is that we do not have the assumption $H\in \GL_r(\Omega_I)$ for some interval $I$ containing $\rho$. As a consequence, the conclusion for $M$ is weaker than that in the statement of \cite{christollecturenotes} Th\'eor\`eme 19.18.
\end{proof}
From now on we prove two generalized $p$-adic Fuchs type theorem based on the theorem above.

\subsection{Relative $\Sigma$-semi-constant $\nabla$-modules}
For a polysegment $I\subset \RR_{>0}^n$ and $\xi=(\xi_1,\dots,\xi_n)\in K^n$, we define the $\nabla$-module $(M_\xi,\nabla_{M_\xi})$ on $A_K^n(I)$ as the $\nabla$-module $(K_I,d+\sum_{i=1}^n\xi_i \frac{dt_i}{t_i})$. Now we introduce the notion of (semi-)$\Sigma$-constancy for relative $\nabla$-modules:

\begin{defn}\label{sigma semiconstancy}
    Let $X$ be a connected smooth rigid space or dagger space, $I\subset \RR_{>0}^n$ be a polysegment and $P$ is a $\nabla$-module over $\ran{I}$ relative to $X$. For $\Sigma=\prod_{i=1}^n\Sigma_i\subset {K}^n$,
     \begin{enumerate}
        \item we say $P$ is $\Sigma$-constant if it has the form $\pi_1^*(P_0)\otimes\pi_2^*(M_\xi,\nabla_{M_\xi})$ for some coherent locally free module $P_0$ over $X$ and $\xi\in\Sigma$, where $\pi_1:\ran{I}\to X,\pi_2:\ran{I}\to A_K^n(I)$ denote the projections.
        \item we  say $P$ is $\Sigma$-semi-constant if it is a direct sum of $\Sigma$-constant relative $\nabla$-modules.
     \end{enumerate}
\end{defn}

\begin{lem}\label{characterization of sigma semiconstancy}
    Suppose $X, I,\Sigma$ are as in Definition \ref{sigma semiconstancy}. Let $P$ be a $\nabla$-module over $\ran{I}$ relative to $X$ of rank $r$. Then, the followings are equivalent:
    \begin{enumerate}
        \item $P$ is $\Sigma$-semi-constant.
        \item There exists an affinoid (resp. affinoid dagger) covering $\{U_\lambda\}_{\lambda\in\Lambda}$ of $X$ such that $P|_{U_\lambda\times A_K^n(I)}$ is free and such that there exists a basis of $P|_{U_\lambda\times A_K^n(I)}$ on which $\tpa$ acts via a diagonal matrix $G_{t_i}$ (independent of choice of $\lambda$) with entries in $\Sigma_i $.
    \end{enumerate}
\end{lem}
\begin{proof}
    1 $\Longrightarrow$ 2: If $P$ is $\Sigma$-semi-constant, then there exists $\xi_1,\dots,\xi_k\in{K}^n$ and coherent locally free modules $P_{01},\dots,P_{0k}$ over $X$ such that
    \begin{equation*}
         P=\bigoplus_{i=1}^k \pi_1^*(P_{0i})\otimes \pi_2^*(M_{\xi_i}). \tag{*}
    \end{equation*}
    Then we can find an affinoid covering $\{U_\lambda\}_{\lambda\in\Lambda}$ of $X$ such that $P_{0i}|_{U_\lambda}$ is free for $1\leq i\leq k$ and $\lambda \in \Lambda$. Since it suffices to prove 2 for every direct summand in (*), we can reduce to the case where $P$ is $\Sigma$-constant, namely, $k=1$. Set $e_1,\dots, e_r$ be a basis of $P_{01}|_{U_\lambda}$ and set $v$ be a basis of $M_{\xi_1}$. Then $\pi^*(e_1)\otimes\pi_2^*(v),\dots,\pi^*(e_r)\otimes\pi_2^*(v)$ is the desired basis of $P|_{U_\lambda\times A_K^n(I)}$.

    2 $\Longrightarrow$ 1: 
    Set $\xi=(\xi_1,\dots,\xi_n)$. If such a covering and basis exist, let $e_{\lambda,1},\dots,e_{\lambda,r}$ be the basis of $P|_{U_\lambda\times A_K^n(I)}$ on which $\tpa$ acts via $G_{t_i}$ for $1\leq i\leq n$. We choose $\xi$-eigenvectors $e'_{\xi,\lambda,1},\dots,e'_{\xi,\lambda,d}$, namely $\tpa(e'_{\xi,\lambda,i})=\xi_i e'_{\xi,\lambda,i}$ for $1\leq i\leq d$, from $e_{\lambda,1},\dots,e_{\lambda,r}$.
    Let $U_\lambda,U_{\lambda'}$ be two affinoid subdomains such that $U_\lambda\cap U_{\lambda'}\neq\emptyset$ and use $\overline{e'_{\xi,\lambda,i}}, \overline{e'_{\xi,\lambda',i}}$ to denote the restrictions of $e'_{\xi,\lambda,i},e'_{\xi,\lambda',i}$ on $(U_\lambda\cap U_{\lambda'})\times A_K^n(I)$. Let $H\in\GL_d(\Gamma(U_\lambda\cap U_{\lambda'},\OO_X)_I)$ be the change-of-basis matrix such that 
    $$(\overline{e'_{\xi,\lambda',1}},\dots,\overline{e'_{\xi,\lambda',d}})=(\overline{e'_{\xi,\lambda,1}},\dots,\overline{e'_{\xi,\lambda,d}})H,  $$
    by applying $\tpa$ on both sides, we have $H\xi_i=\xi_iH+\tpa(H)$. Thus we conclude that $\tpa(H)=0$ for $1\leq i\leq n$, namely, $H\in\Gamma(U_\lambda\cap U_{\lambda'},\OO_X)$. If we set $P_{0,\xi,\lambda}=\mathrm{span}_{\Gamma(U_\lambda,\OO_X)}(e_{\xi,\lambda,1},\dots,e_{\xi,\lambda,d})$, we have
    $$ P_{0,\xi,\lambda}|_{U_\lambda\cap U_{\lambda'}}\cong P_{0,\xi,\lambda'}|_{U_\lambda\cap U_{\lambda'}}, $$
    via $H$. Thus if we set
    $$ P_{0,\xi}:=\ker(\prod_{\lambda\in\Lambda}P_{0,\xi,\lambda}\rightrightarrows \prod_{\lambda,\lambda'\in\Lambda}P_{0,\xi,\lambda}|_{U_\lambda\cap U_{\lambda'}}), $$
    we immediately see that $P=\bigoplus_{\xi\in{K}^n}\pi^*(P_{0,\xi})\otimes \pi_2^*(M_\xi)$. 
\end{proof}
\begin{rmk}
    Keep the notations as in Definition \ref{sigma semiconstancy}. Assume that $P$ is $\Sigma$-semi-constant and is an object in $\Robb{X}$. According to Lemma \ref{characterization of sigma semiconstancy}, we can find an affinoid (or an affinoid dagger) covering $\{U_\lambda\}_{\lambda\in\Lambda}$ such that $P|_{U_\lambda\times A_K^n(I)}$ is free admitting a basis on which $\tpa$ acts via a diagonal matrix $G_{t_i}=\diag(g_{i1},\dots,g_{ir})$ with $g_{ij}\in \Sigma_i$. In this case, the $n$-tuple of matrices $G=(G_{t_1},\dots,G_{t_n})$ defines an exponent $G$ of $P$ and for simplicity, we also say that $P$ is $G$-semi-constant. 
\end{rmk}

\begin{lem}\label{generization}
    Suppose $\rho\in\RR_{>0}$ and $I\subset \RR_{>0}^n$ is an open polysegment containing some $K$-free element $\rho'$. Let $\eta$ denote the Shilov boundary of $\BBB_{\rho}^1$ and $\Sigma=\prod_{i=1}^n \Sigma_i\subset {K}^n$ be a set which satisfies $(\NID)$. For a $\nabla$-module $P$ over $\BBB_{\rho}^1\times A_K^n(I)$ relative to $\BBB_{\rho}^1$, if $P|_{\eta\times A_K^n(I)}$ is $\Sigma$-semi-constant, then for any $\gamma<\rho$, $P|_{\BBB_{\gamma}^1\times A_K^n(I)}$ is free and $\Sigma$-semi-constant.
\end{lem}

\begin{proof}
    Firstly we assume that $P$ is free and $I$ is closed. Then there exists a basis $e_1,\dots,e_r$ of $P|_{\eta\times A_K^n(I)}$ such that 
    $$ \tpa(e_1,\dots,e_r)=(e_1,\dots,e_r)G_{t_i}, $$
    with $G_{t_i}$ being a diagonal matrix with eigenvalues in $\Sigma_i$.
    Let $e'_1,\dots,e'_r$ be a basis of $P$ and suppose that $\tpa$ acts on this basis via a matrix $G'_{t_i}\in\Mat_{r\times r}(K\la \rho^{-1}s\ra \la \alpha/t,t/\beta\ra)$ for $1\leq i\leq n$. Let $H\in\GL_r(K(s)_\rho\la \alpha/t,t/\beta\ra)$ be the change-of-basis matrix such that $(e_1,\dots,e_r)H=(e'_1,\dots,e'_r) $.  By a simple computation we have 
    \begin{equation}
        G'_{t_i}=H^{-1}G_{t_i}H+H^{-1}\tpa(H).\tag{*}
    \end{equation}
    Regard $H$ as an invertible matrix in $\GL_r(F_{\underline{\rho}})=\GL_r(F_{\rho'}(s)_\rho)$ where $\underline{\rho}=(\rho,\rho')$ and $F_{\rho'}=K(t)_{\rho'}=K\la \rho'/t,t/\rho'\ra$.
    By Theorem \ref{decomposition of matrix}, we have a factorization $H=s^NLM$ where $N=\diag(n_1,\dots,n_r)$ with $n_j\in\ZZ$ for $1\leq j\leq r$, $L\in\GL_r(F_{\rho'}\la \rho/s\ra)$ with $L(\infty)=I$ and $M=\GL_r(F_{\rho'}[[ s/\rho]]_\an)$.
    So (*) becomes
    \begin{align*}
        G'_{t_i} &= M^{-1}L^{-1}s^{-N}G_{t_i}s^NLM+M^{-1}L^{-1}s^{-N}\tpa(s^NLM)\\
        &=M^{-1}L^{-1}G_{t_i}LM+M^{-1}L^{-1}(\tpa(L)M+L\tpa(M))\\
        &=M^{-1}L^{-1}G_{t_i}LM+M^{-1}L^{-1}\tpa(L)M+M^{-1}\tpa(M),
    \end{align*}
    because $G_{t_i}$ is a diagonal matrix for $1\leq i\leq n$. This implies
    \begin{equation}
       MG'_{t_i}M^{-1}-\tpa(M)M^{-1}=L^{-1}G_{t_i}L+L^{-1}\tpa(L). \tag{**}
    \end{equation} 
    Regarding $G'_{t_i}$ as a matrix in $\Mat_{r\times r}(F_{\rho'}\la s/\rho\ra)$, left hand side of (**) is a matrix in $\Mat_{r\times r}(F_{\rho'}[[ s/\rho]]_\an)$. Meanwhile, right hand side of (**) is a matrix in $\Mat_{r\times r}(F_{\rho'}\la\rho/s \ra)$, and so both sides of (**) should be matrices in $\Mat_{r\times r}(F_{\rho'})=\Mat_{r\times r}(K\la \rho'/t,t/\rho'\ra)$ by Lemma \ref{intersection is a field}. Write $L^{-1}G_{t_i}L+L^{-1}\tpa(L)=A_i$ and then $G_{t_i}L+\tpa(L)=LA_i$. By the property of $L$, we can write $L=\sum_{l=-\infty}^0L_l s^l$ with $L_l\in \Mat_{r\times r}(F_{\rho'})$ for $l \leq 0$ and $L_0=I$. By considering coefficients of $s^l$, we have 
    $ G_{t_i}L_l+\tpa(L_l)=L_l A_i. $
    By putting $l=0$, we can easily see that $A_i=G_{t_i}$. This implies that for $l\leq 0$,
    $$ G_{t_i}L_l+\tpa(L_l)=L_l G_{t_i}. $$
    Use $\wh{t}_i$ to denote the $(n-1)$-tuple $(t_1,\dots,t_{i-1},t_{i+1},t_n)$. Now write $L_l=\sum_{k_i=-\infty}^\infty L_{lk_i}t^{k_i}$ with $L_{lk_i}\in \Mat_{r\times r}(K\la\wh{\rho}'_i/\wh{t}_i,\wh{t}_i/\wh{\rho}'_i \ra)$. Then we have 
    $$ \sum_{k_i=-\infty}^\infty(G_{t_i}L_{lk_i}+k_iL_{lk_i})t_i^{k_i}=\sum_{k_i=-\infty}^\infty L_{lk_i}G_{t_i} t_i^{k_i}, $$
    which implies that $G_{t_i}L_{lk_i}+k_iL_{lk_i}=L_{lk_i}G_{t_i}$. Since the differences of eigenvalues of $G_{t_i}$ is not an integer, we have $L_{lk_i}=0$ if $k_i\neq 0$. This implies that $L_l=L_{l0}$ and $L\in GL_r(K\la \rho/s \ra\la \wh{\rho}'_i/\wh{t}_i,\wh{t}_i/\wh{\rho}'_i\ra)$. By applying the above procedure for every $1\leq i\leq n$, we conclude that $L\in\GL_r(K\la\rho/s \ra)$. Set $(e''_1,\dots, e''_r)=(e'_1,\dots,e'_r)M^{-1}$. Since we can write $M=L^{-1}s^{-N}H$, by Lemma \ref{intersection} we have 
    \begin{align*}
        M &\in\GL_r(K[[ s/\rho]]_\an\la\rho'/t,t/\rho' \ra)\cap \GL_r(K(s)_\rho\la\alpha/t,t/\beta \ra)\\
    &=\GL_r(K[[ s/\rho]]_\an\la\alpha/t,t/\beta \ra)\subset \GL_r(K\la s/\gamma\ra\la\alpha/t,t/\beta \ra).
    \end{align*}
    Thus $(e''_1,\dots,e''_r)$ defines a basis of $P|_{\BBB_{\gamma}^1\times A_K^n(I)}$. Moreover, since we proved that $L^{-1}G_{t_i}L+L^{-1}\tpa(L)=G_{t_i}$ and $L\in\GL_r(K\la \rho/s \ra)$, we have $L^{-1}G_{t_i}L=G_{t_i}$ and so
    \begin{align*}
        \tpa(e''_1,\dots,e''_r)&=\tpa(e_1,\dots,e_r)s^NL= (e_1,\dots,e_r)G_{t_i}s^NL\\
        &= (e_1,\dots,e_r)s^NL(L^{-1}G_{t_i} L)=(e''_1,\dots,e''_r)G_{t_i}.
    \end{align*}
    This shows that $P|_{\BBB_{\gamma}^1\times A_K^n(I)}$ is also $\Sigma$-semi-constant. 

    Now we prove the case when $P$ is locally free and $I$ is open. For a polysegment $J\subset I$ and $\xi=(\xi_1,\dots,\xi_n)\in{K}^n$, we define $V_{J,\xi}\subset P|_{\BBB^1_\gamma\times A_K^n(J)}$ by $V_{J,\xi}:=\bigcap_{i=1}^n\ker(\tpa-\xi_i)$.
    For an $n$-tuple of $r\times r$ diagonal matrices $G=(G_1,\dots,G_n)$ with $G_i=\diag(g_{i1},\dots,g_{ir})$ and $g_{ij}\in {K}$ for $1\leq i\leq n,1\leq j\leq r$, set 
    $$S(G):=\{\xi\in{K}^n:\xi=(g_{1j},\dots, g_{nj}) \text{ for some }j=1,\dots, r\}$$ 
    and define $V_{J,G}\subset P|_{\BBB^1_\gamma\times A_K^n(J)}$ by $V_{J,G}=\bigoplus_{\xi\in S(G)}V_{J,\xi}$. For $\delta\in I$, there exists an open subpolysegment $I_\delta\subset I$ containing $\delta$ such that $P|_{\BBB^1_\rho\times A_K^n(I_\delta)}$ is free. Then, by above conclusion for the case $P$ is free, for any closed subpolysegment $J\subset I_\delta$, $P|_{\BBB^1_\gamma\times A_K^n(J)}$ admits a basis $e_1,\dots,e_r$ on which each $\tpa$ acts via a diagonal matrix $G_{t_i}$ with entries in $\Sigma_i$. Set $G=(G_{t_1},\dots,G_{t_n})$ and set $V_J=V_{J,G}$. For each $x\in V_J$, it has a unique expression $x=\sum_{\xi\in S(G)} v_\xi$ with $v_\xi\in V_{J,\xi}$. Choose $e'_1,\dots,e'_{r'}$ from $e_1,\dots,e_r$ which form a basis of $V_{J,\xi}$, write $v_\xi=f_1 e'_1+\cdots +f_{r'}e'_{r'}$ and put $f=(f_1,\dots,f_{r'})^T$, we have 
    $$0=(\tpa-\xi_i)((e'_1,\dots,e'_{r'})f)=(e'_1,\dots,e'_{r'})\tpa(f)$$
    for any $1\leq i\leq n$, and this forces $f_1,\dots,f_{r'}$ to be elements in $K\la \gamma^{-1}s\ra$. So we see that $V_J$ is the free $K\la \gamma^{-1}s\ra$-module of rank $r$ generated by $e_1,\dots,e_r$. Then, since $V_{I_\delta}=\varprojlim_{J\subset I_\delta} V_J$, $V_{I_\delta}$ is also a free $K\la \gamma^{-1}s\ra$-module of rank $r$. For $I_\delta$ chosen above and an open subpolysegment $I'\subset I_\delta$, the restriction map $V_{I_\delta}\to V_{I'}$ is an isomorphism. By definition,
    $$ V_I=\ker\left( \prod_{\delta\in I}V_{I_\delta}\rightrightarrows \prod_{\delta,\delta'\in I} V_{I_\delta\cap I_{\delta'}} \right). $$

    On the one hand, let $\mathscr{V}$ be the local system of free $K\la \gamma^{-1}s\ra$-modules on the topological space $I\subset \RR^n_{>0}$ such that $\mathscr{V}_{I_\delta}$ is the constant local system $V_{I_\delta}$ and that the gluing is given by 
$$ (\mathscr{V}|_{I_\delta})|_{I_\delta\cap I_{\delta'}}= V_{I_\delta}|_{I_\delta\cap I_{\delta'}}\cong V_{I_\delta\cap I_{\delta'}}\cong V_{I_{\delta'}}|_{I_\delta\cap I_{\delta'}}= (\mathscr{V}|_{I_{\delta'}})|_{I_\delta\cap I_{\delta'}}. $$
On the other hand, since $I$ is contractible, $\mathscr{V}$ is necessarily isomorphic to the constant local system $K\la\gamma^{-1}s\ra^{\oplus r}$ and so $V_I=\mathrm{H}^0(I,\mathscr{V})\cong\mathrm{H}^0(I,K\la\gamma^{-1}s\ra^{\oplus r})=K\la\gamma^{-1}s\ra^{\oplus r}$ is a free $K\la\gamma^{-1}s\ra$-module of rank $r$ such that the restriction maps $V_I\to V_{I_\rho}$ are isomorphisms. So a $K\la\gamma^{-1}s\ra$-basis of $V_I$ gives rise to a basis of $P|_{\BBB^1_{\gamma}\times A_K^n(I_\delta)}$ for any $\delta$, and so it gives a basis of $P$. Thus $P|_{\BBB^1_\gamma\times A_K^n(I)}$ is free and $\Sigma$-semi-constant.
\end{proof}
\begin{lem}\label{descent to smaller fields}
    Assume $K/\QQ_p$ is an extension of complete nonarchimedean fields. Let $I$ be a polysegment, $\Sigma=\prod_{i=1}^n\Sigma_i\subset{K}^n$ and $P$ be a free $\nabla$-module on $\BBB^1_{K,\rho}\times A_K^n(I)$ relative to $\BBB^1_{K,\rho}$ such that $P|_{\eta\times A_K^n(I)}$ is $\Sigma$-semi-constant where $\eta$ is the Shilov boundary of $\BBB^1_{K,\rho}$. Then there exists a subextension of complete nonarchimedean fields $K_0/\QQ_p$ of $K/\QQ_p$ such that the $\QQ$-vector space $\log\sqrt{|K_0^\times|}$ is of at most countable dimension and a free $\nabla$-module $P_0$ on $\BBB^1_{K_0,\rho}\times A_{K_0}^n(I)$ relative to $\BBB^1_{K_0,\rho}$ such that $P=P_0\otimes_{K_0\la\rho^{-1}s\ra_I}K\la\rho^{-1}s\ra_I$ and $P_0|_{\eta_0\times A_{K_0}^n(I)}$ is $\Sigma$-semi-constant, where $\eta_0$ is the Shilov boundary of $\BBB^1_{K_0,\rho}$.
\end{lem}
\begin{proof}
    For $G=(g_{\mu\nu})\in\Mat_{r\times r}(K\la\rho^{-1}s\ra_I)$, write 
    $$g_{\mu\nu}=\sum_{i\in\ZZ_{\geq 0},{j\in\ZZ^n}}g_{\mu\nu ij}s^it^j$$
    and denote the smallest subfield of $K$ containing $g_{\mu\nu ij}$ for $i\in\ZZ_{\geq 0},j\in\ZZ^n,1\leq \mu,\nu\leq r$ by $\QQ_p(G)$. For $G=(g_{\mu\nu})\in\Mat_{r\times r}(K(s)) $, denote the smallest subfield of $K$ containing all coefficients (in $K$) of the rational function $g_{\mu\nu}$ for every $1\leq \mu,\nu\leq r$ also by $\QQ_p(G)$.
    Let $e_1,\dots,e_r$ be a basis of $P$ and let $G_{t_i}$ be the representation matrix of $\tpa$ with respect to this basis. By Lemma \ref{characterization of sigma semiconstancy}, we can find a basis $e'_1,\dots,e'_r$ of $P|_{\eta\times A_K^n(I)}$ on which $\tpa$ acts via a diagonal matrix $G'_{t_i}$ with entries in $\Sigma_i$. Put $\overline{e_i}$ to be the restriction of $e_i$ on $P|_{\eta\times A_K^n(I)}$ and set $H\in\GL_r(K(s)_{\rho,I})$ be the change-of-basis matrix such that
    $$ (\overline{e_1},\dots,\overline{e_r})H=(e'_1,\dots,e'_r). $$
    If we write $H=\sum_{j\in\ZZ^n}H_jt^j$ with $H_j\in\Mat_{r\times r}(K(s)_\rho)$ and write $H_j=\lim_{i\to\infty} H_{ij}$ with $H_{ij}\in\Mat_{r\times r}(K(s))$. Then, by setting
    \begin{align*}
        K_0&:=\QQ_p(G_{t_1},\dots,G_{t_n},\{H_{ij}\}_{i\geq 1,j\in\ZZ^n})^{\wedge} \text{ and} \\
        P_0&:=\mathrm{span}_{K_0\la\rho^{-1}s\ra_I}(e_1,\dots,e_r)
    \end{align*}
    with $\tpa$ acts on $P_0$ via $G_{t_i}$, one easily see that the conditions in the statement of the lemma are satisfied.
\end{proof}

\begin{rmk}
    Keep the notations as in Lemma \ref{generization}. The assumption that $I$ contains some $K$-free element $\rho'$ in Lemma \ref{generization} is superfluous. If $\log \sqrt{|K^\times|}$ is at most countable dimension as a $\QQ$-vector space, then such $\rho'$ always exists. If $\log \sqrt{|K^\times|}$ is of uncountable dimension, we firstly assume that $P$ is free. Then by Lemma \ref{descent to smaller fields}, we can reduce to the case where $\log \sqrt{|K^\times|}$ is at most countable dimension. For general $P$ which is not necessarily free, we can first work with free ones locally, and glue local pieces in the same way as in the proof of Lemma \ref{generization}.
\end{rmk}

Let $X$ be an affinoid dagger space of dimension $1$ and $I\subset\RR_{>0}^n$ be a polysegment. Suppose there exists a finite \'etale morphism $f:X\to\BBB^{1\dagger}$. According to Remark \ref{extension of finite etale morphisms}, for a $\nabla$-module $P$ over $\ran{I}$ relative to $X$, we can always find a presentation $(P',X',J)$ of $P$ and some $\rho>1$ such that $f$ extends to some finite \'etale morphism $X'\to \BBB^{1}_{\rho}$. We call it a presentation compatible with $f$.

\begin{lem}\label{union of weak equivalence multisets}
    Let $A$ be a multisubset of $\ZZ_p^n$ with Liouville partition $\AA_1\cup\cdots\cup \AA_k$. Suppose $B_1,\dots,B_l$ are multisubsets of $\ZZ_p^n$ weakly equivalent to $A$. Then, there exists a Liouville partition $\BB_{i1}\cup\cdots\cup\BB_{ik}$ of $B_i$ for each $1\leq i\leq l$ such that $\BB_{ij}$ is weakly equivalent to $\AA_j$ for $1\leq j\leq k$. Moreover, 
    $$\left( \bigcup_{i=1}^l \BB_{i1} \right)\cup\cdots\cup\left( \bigcup_{i=1}^l \BB_{ik} \right)$$
    is a Liouville partition of $\bigcup_{i=1}^l B_i$ with $\bigcup_{i=1}^l \BB_{ij}$ weakly equivalent to $\underbrace{\AA_j\cup\cdots\cup \AA_j}_{l\text{ times}}$ for $1\leq j\leq k$.
\end{lem}
\begin{proof}
    It suffices to prove the assertion for Liouville partition in every direction $r$. Then, the existence of such Liouville partitions of $B_1,\dots,B_l$ follows from Proposition \ref{weak_equivalence}. To prove the second assertion, we can reduce to the case where $k=l=2$. In this case, it suffices to show that both $\BB_{11}\cup\BB_{22}$ and $\BB_{12}\cup\BB_{21}$ are Liouville partitions in the $r$-th direction, which follows from Proposition \ref{weak_equivalence}.
\end{proof}

\begin{prop}\label{decomposition of semiconstant}
    Suppose $f:X\to \BBB_K^{1\dagger}$ is a finite \'etale morphism of degree $d$ of connected dagger spaces and $I\subset \RR_{>0}^n$ is a closed polysegment. Let $P$ be a free object in $\Robb{X}$ with exponent $A$ having Liouville partition $\AA_1\cup\cdots\cup \AA_k$ and let $(P',X',J)$ be a presentation of $P$ compatible with $f$. If there exists an element $A_\xi$ in $\ZZ_p^{r\times n}$ such that $P'|_{\xi\times A_K^n(J)}$ is $A_\xi$-semi-constant for every $\xi\in\mathcal{S}(X')$, then $P$ admits a decomposition $P=P_1\oplus\cdots\oplus P_k$ with $\AA_i$ being an exponent of $P_i$ for $1\leq i\leq k$. In particular, if $A$ is $(\NLD)$, then $P$ is $A$-semi-constant.
\end{prop}

\begin{proof}
    Denote the extension of $f$ to $X'$ by $f':X'\to \BBB_{\rho}^1$ and denote the Shilov boundary of $\BBB_{\rho}^1$ by $\eta$. Then $(f'_{J*}P')|_{\eta\times A_K^n(J)}=\bigoplus_{\xi\in \cS(X')}f'_{J*}(P'|_{\xi\times A_K^n(J)})$. For each $\xi\in\cS(X')$, since $P'|_{\xi\times A_K^n(J)}$ is $A_\xi$-semi-constant, we can choose a basis $e_{\xi,1},\dots,e_{\xi,r}$ of it such that $\tpa$ acts on this basis via the diagonal matrix
    $G_{\xi,t_i}=\diag(A^i_{\xi,1},\dots,A^i_{\xi,r})$. Since $\HH(\xi)/\HH(\eta)$ is a finite separable extension, we denote its degree by $d_\xi$ and we can find $h_\xi\in\HH(\xi)$ such that $\HH(\xi)=\HH(\eta)[h_\xi]$. Then $\{h_\xi^\mu e_{\xi,\nu}\}_{\overset{1\leq \mu\leq d_\xi}{1\leq \nu\leq r}}$ is a basis of $f'_{J*}(P'|_{\xi\times A_K^n(J)})$ with $\tpa$ acting on this basis via 
    $$
    \begin{pmatrix*}
        G_{\xi,t_i}&\cdots& O\\
        \vdots &\ddots &\vdots \\
        O&\cdots&G_{\xi,t_i}
    \end{pmatrix*}.$$
    Write $\cS(X')=\{\xi_1,\dots,\xi_l\}$, set $_iG=(G_{{\xi_1},t_i},\dots,G_{{\xi_1},t_i}\dots,G_{{\xi_l},t_i}\dots,G_{{\xi_l},t_i})$ with $G_{{\xi_j},t_i}$ appear $d_{\xi_j}$ times and set $G=(_1G,\dots,{_nG})$. Then, $f'_{J*}P'|_{\eta\times A_K^n(J)}$ is $G$-semi-constant, $G$ represents the exponent of $f'_{J*}P'|_{\eta\times A_K^n(J)}$ (which is also the exponent of $f'_{J*}P'$) and hence weakly equivalent to $f'_{J*}A$. Thus we can obtain the corresponding Liouville partition $G_1\cup\cdots\cup G_k$ of $G$ from the Liouville partition $\AA_1,\dots,\AA_k$ of $A$ by Lemma \ref{union of weak equivalence multisets}.
    By Lemma \ref{generization}, $f'_{J*}P'$ is $G$-semi-constant and hence naturally admit a decomposition $f'_{J*}P'=P'_1\oplus\cdots\oplus P'_k$ with $P'_j$ having exponent $f'_{J*}\AA_j$. By Proposition \ref{relative pushforward of decomposition}, this is also a decomposition of $P'$ and hence restricts to a decomposition of $P$ in $\Robb{X}$. 
    
    Moreover, if $A$ is $(\NLD)$, the partition $\AA_1\cup\dots\cup\AA_k$ of $A$ into $\ZZ^n$-cosets is a Liouville partition. In this case, each $P'|_{\xi\times A_K^n(J)}$ is not only $A_\xi$-semi-constant but also $A$-semi-constant because $A_\xi$ and $A$ belong to the same $\ZZ^n$-coset. Then each $P'_i$ is $f'_{J*}\AA_i$-constant as an object in $\mathrm{Rob}_{\BBB^1_{\rho}\times A_K^n(J)/\BBB^1_\rho}$. Let $b=(b_1,\dots,b_n)\in\ZZ_p^n$ be an element representing $\AA_i$. If we write $V=\bigcap_{i=1}^n\ker(\tpa-b_i)$, this is simultaneously an $\OO_{\BBB^1_\rho}$-module and an $\OO_{X'}$-module. Since we have the following commutative diagram:
    \begin{equation*}
        \begin{tikzcd}
            V\otimes_{\OO_{X'}}\OO_{X'\times A_K^n(J)} \arrow{r}{i} \arrow[swap]{d}{\cong} & P \arrow{d}{=} \\%
            V\otimes_{\OO_{\BBB^1_\rho}}\OO_{\BBB^1_\rho\times A_K^n(J)}  \arrow{r}{\cong}& f_{J*}P,
        \end{tikzcd} 
    \end{equation*}
    we conclude that $i$ is an isomorphism. Take some Tate point $x$ of $A_K^n(I)$. Then $V\otimes_{\OO_{X'}}\OO_{X'\times x}\cong P|_{X'\times x}$ is locally free, from which we conclude that $V$ is locally free by faithfully flat descent. This implies that $P'_i$ is also $\AA_i$-constant as an object in $\mathrm{Rob}_{X'\times A_K^n(J)/X'}$. Hence $P$ is $A$-semi-constant.
\end{proof}
By working locally, the following is a consequence of above proposition:

\begin{cor}
    Suppose $X$ is a connected smooth affinoid dagger curve and $I\subset\RR_{>0}^n$ is a closed polysegment. Let $P$ be an object in $\Robb{X}$ with exponent $A$ having Liouville partition $\AA_1\cup\cdots\cup \AA_k$ and $(P',X',J)$ be a presentation of $P$. If for every point $x\in X'$, there exists $A_x\in\ZZ_p^{r\times n}$ such that $P|_{x\times A_K^n(J)}$ is $A_x$-semi-constant, then $P$ admits a unique decomposition $P_1\oplus\cdots\oplus P_k$ with $\AA_i$ being an exponent of $P_i$ for $1\leq i\leq k$.
\end{cor}

\begin{proof}
    By Lemma \ref{localQS} and Corollary \ref{existence_of_finite_etale_morphism}, we can find an affinoid dagger covering $\{U_\lambda\}_{\lambda\in\Lambda}$ and a closed covering $\{J_\delta\}_{\delta\in\Delta}$ of $I$ with nonempty interior such that $P|_{U_\lambda\times A_K^n(J_\delta)}$ is free and that there exists a finite \'etale morphism $f:U_\lambda\to\BBB^{1\dagger}$ for every $\lambda\in\Lambda, \delta\in\Delta$. Then, by Proposition \ref{decomposition of semiconstant}, we have the decomposition $P|_{U_\lambda\times A_K^n(J_\delta)}=\bigoplus_{i=1}^k P_{(\lambda,\delta),i}$ of $\nabla$-module over $U_\lambda\times A_K^n(J_\delta)$ relative to $U_\lambda$. By the uniqueness assertion in Lemma \ref{uniqueness}, we have the canonical equality
    $$ P_{(\lambda,\delta),i}|_{(U_\lambda\cap U_{\lambda'})\times (A_K^n(J_\delta)\cap A_K^n(J_{\delta'}))}=P_{(\lambda',\delta'),i}|_{(U_\lambda\cap U_{\lambda'})\times (A_K^n(J_\delta)\cap A_K^n(J_{\delta'}))}, $$
    for $(\lambda,\delta),(\lambda',\delta')\in \Lambda\times\Delta$. Then, since we have the exact sequence of sheaves
    $$ 0\lra P\lra \prod_{(\lambda,\delta)\in \Lambda\times\Delta}P|_{U_\lambda\times A_K^n(J_\delta)}\rightrightarrows \prod_{(\lambda,\delta),(\lambda',\delta')\in \Lambda\times\Delta}P|_{(U_\lambda\cap U_{\lambda'})\times (A_K^n(J_\delta)\cap A_K^n(J_{\delta'}))},  $$
    if we define 
    $$P_i=\ker\left( \prod_{(\lambda,\delta)\in\Lambda\times\Delta}P_{(\lambda,\delta),i}\rightrightarrows \prod_{(\lambda,\delta),(\lambda',\delta')\in \Lambda\times\Delta}P_{\lambda,i}|_{(U_\lambda\cap U_{\lambda'})\times (A_K^n(J_\delta)\cap A_K^n(J_{\delta'}))} \right),$$
    we obtain the desired decomposition $P=P_1\oplus\cdots\oplus P_k$. The uniqueness follows from Lemma \ref{uniqueness}.
\end{proof}

\subsection{Generalized $p$-adic Fuchs theorem for special $\nabla$-modules}
In this subsection, We prove the generalized $p$-adic Fuchs theorem for objects in $\Rob{\BBB^{1,\dagger}}$, on which derivations from the base acts ``constantly''.

Let $X$ be $\BBB_{\rho}^1$, $\BBB_{\rho}^{1,\dagger}$ or $\Sp(F_\rho)$. For a polysegment $I\subset \RR_{>0}^n$ and $\xi\in K$, we define the $\nabla$-module $(Q_\xi,\nabla_{Q_\xi})$ on $X$ as the $\nabla$-module $(\OO_X,d+\xi  ds)$ and we call $\xi$ its eigenvalue. 

Now we introduce the notion of constancy on the base for $\nabla$-modules over $\ran{I}$:

\begin{defn}\label{semiconstant on the base}
    Suppose $I\subset \RR_{>0}^n$ is an open polysegment. Let $X$ be either $\BBB_{\rho}^1$, $\BBB_{\rho}^{1,\dagger}$ or $\Sp(F_\rho)$  and $P$ be a $\nabla$-module over $X\times A_K^n(I)$. 
    We say $P$ is $\xi$-constant on the base if it has the form $\pi_1^*(Q_\xi,\nabla_{Q_\xi})\otimes \pi_2^*(P_0,\nabla_{P_0})$ for some  $\xi\in {K}$ and $\nabla$-module $P_0$ on $A_K^n(I)$, where $\pi_1:X\times A_K^n(I)\to X$, $\pi_2:\ran{I}\to A_K^n(I)$ denote the projections.
\end{defn}

\begin{lem}\label{equvalence condition of constancy}
    Let $X,I,\xi$ be as in Definition \ref{semiconstant on the base}. Let $P$ be a $\nabla$-module over $\ran{I}$. Then the followings are equivalent:
    \begin{enumerate}
        \item $P$ is $\xi$-constant on the base.
        \item There exists a covering $\{I_\lambda\}_{\lambda\in\Lambda}$ of $I$ by closed subpolysegments such that $P|_{X\times A_K^n(I_\lambda)}$ is free, and such that there exists a basis of $P|_{X\times A_K^n(I_\lambda)}$ on which $\partial_s$ acts via multiplication by $\xi$.
    \end{enumerate}
\end{lem}
\begin{proof}
    1 $\Longrightarrow $ 2:
        If $P$ is $\xi$-constant on the base, by definition, there exists a $\nabla$-module $P_0$ on $A_K^n(I)$ such that $P=\pi_1^*(Q_\xi)\otimes \pi_2^*(P_0)$. We can find a covering $\{I_\lambda\}_{\lambda\in\Lambda}$ of $I$ by closed subpolysegments such that $P_0|_{A_K^n(I_\lambda)}$ is free for all $\lambda\in\Lambda$ by Proposition \ref{QS}. It is obvious that $P|_{X\times A_K^n(I_\lambda)}$ is free. Choose a basis $e_1,\dots,e_r$ of $P_0|_{A_K^n(I_\lambda)}$ and let $v$ be a basis of $Q_\xi$, then $\pi^*_1(v)\otimes\pi^*_2(e_1),\dots,\pi^*_1(v)\otimes\pi^*_2(e_r)$ is the desired basis of $P|_{X\times A_K^n(I_\lambda)}$.

    2 $\Longrightarrow$ 1:
    If such a covering exists, for $\lambda\in\Lambda$, let $e_{\lambda,1},\dots,e_{\lambda,r}$ be the basis of $P|_{X\times A_K^n(I_\lambda)}$ described in 2 and let $G_{t_i}$ denote the representation matrix of $\tpa$ with respect to the basis $e_{\lambda,1},\dots,e_{\lambda,r}$. Then we have
    \begin{align*}
        (e_{\lambda,1},\dots,e_{\lambda,r})[\xi G_{t_i}+\spa(G_{t_i})]&=\spa( t_i\partial_{t_i}(e_{\lambda,1},\dots,e_{\lambda,r}))\\
        & =t_i\partial_{t_i}(\spa (e_{\lambda,1},\dots,e_{\lambda,r}))\\
        & = (e_{\lambda,1},\dots,e_{\lambda,r})\xi G_{t_i}.
    \end{align*}
    Thus we have $\spa(G_{t_i})=0$, which implies that $G_{t_j}\in\Mat_{r\times r}(K_{I_\lambda})$. For $\lambda,\lambda'\in\Lambda$ such that $I_\lambda\cap I_{\lambda'}\neq\emptyset$. Let $H\in\GL_r(\Gamma(X,\OO_X)_{I_\lambda\cap I_{\lambda'}})$ be the change-of-basis matrix such that 
    \begin{equation*}
        (\overline{e_{\lambda',1}},\dots,\overline{e_{\lambda',r}})=(\overline{e_{\lambda,1}},\dots,\overline{e_{\lambda,r}})H, \tag{*}
    \end{equation*}
    where $\overline{e_{\lambda,i}},\overline{e_{\lambda',i}}$ denote the restriction of  $e_{\lambda,i},e_{\lambda',i}$ on $X\times A_K^n(I_{\lambda}\cap I_{\lambda'})$. Then, by action of $\partial_s$ on both sides of (*), we compute that $\partial_s(H)=0$, that is, $H\in\GL_r(K_{I_\lambda\cap I_{\lambda'}})$. If we set $P_{0,\lambda}=\mathrm{span}_{K_{I_\lambda}}(e_{\lambda,1},\dots,e_{\lambda,r})$, we have 
    $$P_{0,\lambda}|_{A_K^n(I_\lambda)\cap A_K^n(I_{\lambda'})}\cong P_{0,\lambda'}|_{A_K^n(I_\lambda)\cap A_K^n(I_{\lambda'})},$$
    where the isomorphim is given via $H$. If we define 
    $$ P_0:=\ker\left( \prod_{\lambda\in\Lambda}P_{0,\lambda}\rightrightarrows\prod_{\lambda,\lambda'\in\Lambda}P_{0,\lambda}|_{A_K^n(I_\lambda)\cap A_K^n(I_{\lambda'})}\right), $$
    we immediately see that  $P=\pi^*_1(Q_\xi)\otimes \pi_2^*(P_0)$. 
\end{proof}

\begin{prop}\label{C on base}
    Suppose that $I=[\alpha,\beta]\subset \RR_{>0}^n$ is a closed polysegment, and $\gamma<\rho\in\RR_{>0}$. Let $P$ be a free $\nabla$-module over $\BBB^1_{\rho}\times A_K^n(I)$ and $\eta$ be the Shilov boundary of $\BBB^1_{\rho}$. If there exists a basis $e_1,\dots,e_r$ of $P|_{\eta\times A_K^n(I)}$ on which $\spa$ acts via a diagonal matrix $G_s\in\Mat_{r\times r}(K)$, then $e_1,\dots,e_r$ is also a basis of $P|_{\BBB^1_{\gamma}\times A_K^n(I)}$.
\end{prop}
\begin{proof}
   
   Let $e'_1,\dots,e'_r$ be a basis of $P$ and suppose that $\spa$ act on this basis via a matrix $G'_{s}\in\mathrm{Mat}_{r\times r}(K\langle\rho^{-1}s\rangle\la\alpha/t,t/\beta\ra)$. Let $H\in\GL_{r}(K(s)_{\rho}\la\alpha/t,t/\beta\ra)$ be the change-of-basis matrix such that 
    $$(e_1,\dots,e_r)H=(e'_1,\dots,e'_r).$$
    Firstly we assume that $I$ contains some $K$-free tuple $\rho'$. Then by an easy computation we have 
    \begin{equation}
        G'_{s}=H^{-1}G_{s}H+H^{-1}\spa(H).\tag{*}
    \end{equation}
    Set $F=K(t)_{\rho'}$. Then we can view $H$ as an invertible matrix with entries in $F(s)_{{\rho}}$.
    By Theorem \ref{decomposition of matrix}, there exists a triple $(L,M,N)$ with $L\in \GL_r(F\la\rho/s\ra)$, $M\in\GL_r(F[[s/\rho]]_\an) $ and $N$ a diagonal matrix with entries in $\ZZ$ such that 
    $$H=s^NLM, \text{ and } L(\infty)=I_r.$$
    Then (*) becomes
    \begin{align*}
        G'_{s} &= M^{-1}L^{-1}s^{-N}G_{s}s^NL^{}M^{}+M^{-1}L^{-1}s^{-N}\spa(s^NLM)\\
        & = M^{-1}L^{-1}G_{s}L^{}M^{}+s^{-1}M^{-1}L^{-1}NLM+M^{-1}L^{-1}\spa(L)M+M^{-1}\spa(M),
    \end{align*}
    because $G_s$ is a diagonal matrix. This implies that
    \begin{equation}
        L^{-1}G_{s}L+L^{-1}\spa(L)+s^{-1}L^{-1}NL=MG'_{s}M^{-1}-\spa(M)M^{-1}. \tag{**}
    \end{equation} 
    Since $L\in \GL_r(F\la\rho/s\ra)$ and $G_{s}$ is constant, left-hand-side of (**) is a matrix with coefficients in $F\la\rho/s\ra$. Moreover, since $M\in\GL_r(F[[s/\rho]]_\an)$ and $G'_{s}\in\mathrm{Mat}_{r\times r}(K\langle \rho^{-1}s\rangle\la \alpha/t,t/\beta\ra)\subset \Mat_{r\times r}(F\la s/\rho\ra)$, right hand side of (**) is a matrix with coefficients in $F[[s/\rho]]_\an$ and so both sides of (**), in particular $ L^{-1}G_{s}L+L^{-1}\spa(L)+s^{-1}L^{-1}NL$, is a matrix with entries in $F$. Namely, we obtain the following equality:
    $$ L^{-1}G_{s}L+s^{-1}L^{-1}NL+L^{-1}\spa(L)=A$$
    where $A\in\Mat_{r\times r}(F)$. Write $L=\sum_{k=-\infty}^0L_k s^k$.
    Since $L_0=I$, we conclude that $A=G_{s}$. By comparing coefficients of $s^k$, we obtain equalities $G_{s}L_k+(N+(k+1)I)L_{k+1}=L_kG_{s}$ for $k\leq -1$.
    When $k=-1$, we have $G_{s}L_{-1}+N=L_{-1}G_{s}$. By Lemma \ref{C=0}, we have $N=0$. Thus above equalities becomes 
    $$ G_{s}L_k+(k+1)L_{k+1}=L_kG_{s},\ \ k\leq -1 .$$
    Again by Lemma \ref{C=0}, we conclude that $L_k=0$ for $k\leq -1$ and so $L=L_0=I$. Thus $H=M\in \GL_r(F[[s/\rho]]_\an)\cap \GL_{r}(K(s)_{\rho}\la\alpha/t,t/\beta\ra)=\GL_{r}(K[[s/\rho]]_{\an}\la\alpha/t,t/\beta\ra)\subset \GL_r(K\la \gamma^{-1}s\ra\la\alpha/t,t/\beta\ra)$ by Lemma \ref{intersection} and hence the conclusion. 
    
    In general, we argue as in the proof of Lemma \ref{descent to smaller fields}. Let $G'_{t_i}\in\break\Mat_{r\times r}(K\la\rho^{-1}s\ra\la\alpha/t,t/\beta\ra)$ be the representation matrix of $\tpa$ with respect to the basis $e'_1,\dots,e'_r$. 
    If we write $H=\sum_{j\in\ZZ^n}H_jt^j$ with $H_j\in\Mat_{r\times r}(K(s)_\rho)$ and write $H_j=\lim_{i\to\infty} H_{ij}$ with $H_{ij}\in\Mat_{r\times r}(K(s))$. Then, by setting
    \begin{align*}
        K_0&:=\QQ_p(G'_s,G'_{t_1},\dots,G'_{t_n},\{H_{ij}\}_{i\geq 1,j\in\ZZ^n})^{\wedge} \text{ and} \\
        P_0&:=\mathrm{span}_{K_0\la\rho^{-1}s\ra\la\alpha/t,t/\beta\ra}(e'_1,\dots,e'_r)
    \end{align*}
    with $\partial_s$ acts on $P_0$ via $G'_s$ and $\tpa$ acts on $P_0$ via $G'_{t_i}$, we reduce to the case where such a $K_0$-free element exists since the $\QQ$-vector space $\log \sqrt{|K_0^\times|}$ has at most countable dimension.
\end{proof}

\begin{cor}\label{constancy}
    Let $I\subset \RR_{>0}^n$ be a polysegment and $P$ be a $\nabla$-module over $\BBB^{1\dagger}\times A_K^n(I)$. Suppose that there exists a presentation $(P',\BBB^{1}_{\rho},J)$ of $P$ such that $P'|_{\eta\times A_K^n(J)}$ is $\xi$-constant on the base for some $\xi\in {K}$, where $\eta$ is the Shilov boundary of $\BBB^1_\rho$. Then $P$ is $\xi$-constant on the base.
\end{cor}
\begin{proof}
    Since $P'|_{\eta\times A_K^n(J)}$ is $\xi$-constant on the base, we firstly apply Lemma \ref{equvalence condition of constancy} for $X=\{\eta\}$. Next apply Proposition \ref{C on base} and then again apply Lemma \ref{equvalence condition of constancy} for the case $X=\BBB^{1\dagger}$, we obtain the desired conclusion.
\end{proof}
\begin{cor}
    Let $I\subset \RR_{>0}^n$ be a polysegment and $P$ be a $\nabla$-module over $\BBB^{1,\dagger}\times A_K^n(I)$ satisfying the Robba condition having exponent $A$ with Liouville partition $A=\AA_1\cup\dots\cup \AA_k$. Suppose that there exists a presentation $(P',\BBB_{\rho}^1,J)$ of $P$ such that $P'|_{\eta\times A_K^n(J)}$ is $\xi$-constant on the base for some $\xi\in\overline{K}$, where $\eta$ is the Shilov boundary of $\BBB_{\rho}^1$. Then $P$ admits a unique direct sum decomposition $P=P_1\oplus\cdots\oplus P_k$ with $\AA_i$ being an exponent of $P_i$ for $1\leq i\leq k$.
\end{cor}

\begin{proof}
    By Corollary \ref{constancy}, $P$ is $\xi$-constant on the base. Thus there exists a $\nabla$-module $P_0$ over $A_K^n(I)$ such that $P=\pi_1^*Q_\xi\otimes \pi_2^* P_0$ where $\pi_1,\pi_2$ are as in Definition \ref{semiconstant on the base}. By restricting $P$ to the fiber $x\times A_K^n(I)$ at a $K$-valued point, we see that $P_0$ satisfies the Robba condition and has exponent $A$. By Theorem \ref{decom_all}, $P_0$ admits a unique decomposition $P_{01}\oplus\cdots\oplus P_{0k}$ with $\AA_i$ being an exponent of $P_{0i}$. Thus $P=\bigoplus_{i=1}^k \pi_1^*Q_\xi\otimes \pi_2^* P_{0i} $ give rise to the desired decomposition. Moreover, uniqueness follows from Lemma \ref{nabla uniqueness}.
\end{proof}

\begin{appendix}
    \renewcommand{\thesection}{\Alph{section}} 
    \makeatletter
    \def\@seccntformat#1{\@ifundefined{#1@cntformat}%
       {\csname the#1\endcsname.\hspace{0.5em}}
       {\csname #1@cntformat\endcsname}}
    \newcommand\section@cntformat{\appendixname\ \thesection.\hspace{0.5em}}
    \makeatother
\section{Exponents of $\nabla$-modules over $p$-adic (poly)annuli}
The $p$-adic Fuchs theorem was first established by Christol and Mebkhout in \cite{CM2}. Christol and Mebkhout's original proof was complicated due to application of Frobenius antecedent developed in \cite{christol1994modules}. However, the theorem was later simplified by Dwork in \cite{dwork1997exponents}, who provided a more accessible proof that retained the theorem's essential insights. The proof that is mostly adopted nowadays and in this paper is that of Dwork's. It is suggested by Kedlaya in \cite{Ked1} Notes of Chapter 13 that there are still two questions remain unsettled, both of which are related to exponents. We will verify them in the following two subsections.

\subsection{Exponents form a weak equivalence class}
Let $I$ be an open polysegment, $P$ be a $\nabla$-module over $A^n_K(I)$ satisfying the Robba condition in the sense of Definition \ref{absoluteRobba} and $A$ be an exponent of $P$.
It is well-known and written in a lot of literatures that if $B$ is another exponent of $P$ then $A$ is weakly equivalent to $B$. Its converse is also true when $A$ has non-Liouville differences (see \cite{Ked1} Proposition 13.4.5 or \cite{dwork1997exponents} Theorem 5.2), however, it is not known whether the following is also true in general (cf. \cite{dwork1997exponents} Remark 4.5):
\begin{conj}\label{weak_equivalence_is_exponent}
With the teminology as in Definition \ref{exp}, if $B$ is a multisubset of $\ZZ_p^n$ weakly equivalent to $A$, then $B$ is also an exponent of $P$.
\end{conj}
In this subsection, we give a positive answer to the conjecture above.

\begin{lem}\label{eqimat}
    Suppose $I$ is an open polysegment and $P$ is a free $\nabla$-module of rank $r$ over $A^n_K(I)$ satisfying the Robba condition. For a multisubset $A$ of $\ZZ_p^n$ , $A$ is an exponent of $P$ if and only if there exists a sequence of invertible matrices $\{S'_{k,A}\}_{k\geq 0}\subset \mathrm{Mat}_{r\times r}(K_I)$ such that $\{S'_{k,A}\}_{k\geq 0}$ satisfy condition 1 of Definition \ref{exp} and that there exists $l'>0$ with $|S'_{k,A}|_\rho,|S'^{-1}_{k,A}|_\rho\leq p^{l'k}$ for all $\rho\in I$ and $k\geq 0$.
\end{lem}
\begin{proof}
    Suppose that $A$ is an exponent of $P$. By \cite{Gac}, page $188$, Lemma 4, for matrices $S_{k,A}$ defined as in Definition \ref{exp}, there exists $k_0>0$ such that $S_{k,A}$ is invertible for $k>k_0$. Then, from the description of the inverse matrix using cofactors, we have $|S_{k,A}^{-1}|_{\rho}\leq p^{(r-1)kl}$ for all $\rho\in(\alpha,\beta)$. Now we set $l'=\max\{r-1,1\}l$, then
    $$
    S'_{k,A}=\left\{
    \begin{aligned}
        & S_{k,A} \text{ if } k>k_0\\
        & S_{k_0,A} \text{ if } k\leq k_0
    \end{aligned}
    \right.
    $$
    give rise to the desired sequence of matrices.

    Next, suppose that there exists a sequence of matrices $\{S'_{k,A}\}_{k\geq 0}$ as in the statement of the lemma. For each $k$, choose $c_k\in K$ such that $|c_k|=p^{l'k}$, set $l=2l'$ and $S_{k,A}=c_kS'_{k,A}$. Then $|S_{k,A}|_\rho=|c_k||S'_{k,A}|_\rho\leq p^{2l'k}=p^{kl}$. From the definition of determinant, $|\det S'^{-1}_{k,A}|_\rho\leq p^{l'rk}$. So $|\det S'_{k,A}|_\rho\geq p^{-l'rk}$ and
    $$ |\det S_{k,A}|_\rho=|\det c_kS'_{k,A}|_\rho=|c_k|^r|\det S'_{k,A}|_\rho\geq p^{l'rk}p^{-l'rk}=1. $$
\end{proof}

\begin{prop}\label{wkexp}
    Assume $I=(\alpha,\beta)\subset \RR^n_{>0}$. Let $P$ be a free $\nabla$-module of rank $r$ over $A^n_K(I)$ satisfying the Robba condition. Let $A$ be an exponent of $P$ and $B$ be a multisubset of $\ZZ_p^n$ which is weakly equivalent to $A$. Then $B$ is also an exponent of $P$.
\end{prop}
\begin{proof}
    Since $A$ and $B$ are weakly equivalent, there exists $c>0$ and a sequence of permutations $\{\sigma_k\}_{k\geq 0}$ of $\{1,\dots,r\}$ such that
    $$ \la A^i_{\sigma_k(j)}-B^i_j \ra_k\leq ck\ \ \text{for all } k\geq 0,1\leq i\leq n, 1\leq j\leq r. $$
    Use $_{k}A$ (resp. $_{k}B$) denote the unique $n$-tuple of $r\times r$ diagonal matrices with entries in $\ZZ\cap [0,p^k)$ such that 
    $_{k}A\equiv A\ (\mathrm{mod}\ p^k) $ (resp. $_{k}B\equiv B\ (\mathrm{mod}\ p^k) $). Moreover, put 
    $$ _{k}N^i_j=\la A^i_{\sigma_k(j)}-B^i_j \ra_k. $$
    Then either $_{k}N^i_j\equiv A^i_{\sigma_k(j)}-B^i_j\ (\mathrm{mod}\ p^k)$ or $_{k}N^i_j\equiv B^i_j-A^i_{\sigma_k(j)}\ (\mathrm{mod}\ p^k)$. Now define
    $$_{k}{\mathrm{sgn}^i_j}=\left\{
    \begin{aligned}
        1\ \ & \text{if }_{k}N^i_j\equiv A^i_{\sigma_k(j)}-B^i_j\ (\mathrm{mod}\ p^k),\\
        -1\ \ &\text{if } _{k}N^i_j\equiv B^i_j-A^i_{\sigma_k(j)}\ (\mathrm{mod}\ p^k) \text{ and }_{k}N^i_j \not\equiv A^i_{\sigma_k(j)}-B^i_j\ (\mathrm{mod}\ p^k),
    \end{aligned}
    \right.
    $$
    put 
    $$_{k}\mathrm{sgn}^i:=
    \begin{pmatrix*}
        _{k}\mathrm{sgn}^i_1& & \\
         & \ddots  & \\
         & & _{k}\mathrm{sgn}^i_r
    \end{pmatrix*}
    ,\ \ _{k}\mathrm{sgn}=(_{k}\mathrm{sgn}^1,\dots, {_{k}}\mathrm{sgn}^n),
    $$
    and put
    $$_{k}N^i:=
    \begin{pmatrix*}
        _{k}N^i_1& & \\
         & \ddots  & \\
         & & _{k}N^i_r
    \end{pmatrix*}
    ,\ \ _{k}N=(_{k}N^1,\dots,{_{k}}N^n).
    $$
    Fix a basis $e_1,\dots,e_r$ of $P$ and
    suppose that $\{S_{k,A}\}_{k\geq 0}$ is the corresponding sequence of invertible matrices defining $A$ as an exponent of $P$ with respect to this basis, in the sense of Lemma \ref{eqimat}. Let $l>0$ be such that $|S_{k,A}|_\rho,|S_{k,A}^{-1}|_\rho\leq p^{lk}$ for all $\rho \in I$ and $k\geq 0$. Fix $R\in\RR_{>0}$ such that $\max_{1\leq i\leq n}\left\{\max\{1/\alpha_i,\beta_i\}\right\}\leq R$ and choose $a>0$ such that $R^{cn}\leq p^a$. Put $l'=l+a$.

    If we write 
    $$ S_{k,A}=(S_{k,A,1},\dots,S_{k,A,r}) $$ as its column vectors, set
    $$ S_{k,A}^{\sigma_k}:=(S_{k,A,\sigma_k(1)},\dots,S_{k,A,\sigma_k(r)}). $$
    Moreover, for each $A^i=\diag(A^i_1,\dots,A^i_r)$, put $A^i_{\sigma_k}:=\diag(A^i_{\sigma_k(1)},\dots,A^i_{\sigma_k(r)})$.
    By permuting columns of following equality by $\sigma_k$ for each $\zeta\in\mupk^n$
    $$\zeta^*[(e_1,\dots,e_r)S_{k,A}]=(e_1,\dots,e_r)S_{k,A}\zeta^{A}, $$
    we obtain the following: 
    $$ \zeta^*[(e_1,\dots,e_r)S_{k,A}^{\sigma_k}]=(e_1,\dots,e_r)S_{k,A}^{\sigma_k}\zeta^{A_{\sigma_k}}.$$
    Now set $S_{k,B}=S_{k,A}^{\sigma_k}t^{-_{k}\mathrm{sgn}_{k}N}$. We firstly verify the bound of $S_{k,B}$ and $S_{k,B}^{-1}$. For $1\leq i\leq n$, $1\leq j\leq r$, we have 
    $$|t_i^{_{k}N^i_j}|_{\rho_i}=\rho_i^{_{k}N^i_j}\leq R^{ck},\text{ and}$$
    $$ |t_i^{-_{k}N^i_j}|_{\rho_i}=\rho_i^{-_{k}N^i_j} \leq R^{ck},$$
    and so
    $$ |t^{-_{k}\mathrm{sgn}_{k}N}|_{\rho}=\prod_{i=1}^n|t_i^{-(_{k}\mathrm{sgn}^i)( _{k}N^i)}|_{\rho_i}\leq R^{cnk}. $$
    Obviously $|S_{k,A}^{\sigma_k}|_\rho=|S_{k,A}|_\rho$ and $|(S_{k,A}^{\sigma_k})^{-1}|_\rho=|S_{k,A}^{-1}|_\rho$. This implies that for all $\rho \in I$ and $k\geq 0$
    $$ |S_{k,B}|_\rho=|S_{k,A}^{\sigma_k}|_\rho|t^{-_{k}\mathrm{sgn}_{k}N}|_\rho\leq p^{lk}R^{cnk}\leq p^{lk}p^{ak}=p^{l'k} \text{ and} $$
    $$ |(S_{k,B})^{-1}|_\rho=|(S_{k,A}^{\sigma_k})^{-1}|_\rho|t^{_{k}\mathrm{sgn}_{k}N}|_\rho\leq p^{lk}R^{cnk}\leq p^{lk}p^{ak}=p^{l'k}. $$

    Then we verify that $\{S_{k,B}\}_{k\geq 0}$ defines $B$ as an exponent of $P$.
    For $\zeta\in\mupk^n$, we have
    \begin{align*}
        \zeta^*[(e_1,\dots,e_r)S_{k,B}]&=\zeta^*[(e_1,\dots,e_r)S_{k,A}^{\sigma_k}t^{-_{k}\mathrm{sgn}_{k}N}]\\
        &=(e_1,\dots,e_r)S_{k,A}^{\sigma_k}\zeta^{A_{\sigma_k}}\zeta^*(t^{-_{k}\mathrm{sgn}_{k}N})\\
        &= (e_1,\dots,e_r)S_{k,A}^{\sigma_k}t^{-_{k}\mathrm{sgn}_{k}N}\zeta^{A_{\sigma_k}-( _{k}\mathrm{sgn})( _{k}N)}\\
        &=(e_1,\dots,e_r)S_{k,B}\zeta^{A_{\sigma_k}-( _{k}\mathrm{sgn})( _{k}N)}
    \end{align*}
    Note that 
    $$ \zeta^{A_{\sigma_k}-( _{k}\mathrm{sgn})( _{k}N)}=\prod_{i=1}^n\zeta_i^{A^i_{\sigma_k}-( _{k}\mathrm{sgn}^i)( _{k}N^i)}, $$
    so it suffices to verify that for all $1\leq i\leq n, 1\leq j\leq r$
    $$ _{k}A^i_{\sigma_k(j)}-( _{k}\mathrm{sgn}^i_j)( _{k}N^i_j)\equiv  {_{k}B^i_j}\ (\mathrm{mod}\ p^k), $$
    which is 
    $$ _{k}A^i_{\sigma_k(j)}- {_{k}B^i_j}\equiv (_{k}\mathrm{sgn}^i_j)( _{k}N^i_j)\ (\mathrm{mod}\ p^k).$$
    It is true due to the definition of $_{k}\mathrm{sgn}^i_j$.
\end{proof}

\subsection{Equivalence of construction of exponents}
There are two definitions of $p$-adic exponent. We will refer the exponent defined in \cite{CM2} D\'efinition 5.3-6 as CM-exponent and the exponent defined in \cite{dwork1997exponents} Section 4 by Dw-exponent in the following. Instead of the complex nature of Frobenius antecedent in \cite{CM2} that is used to define $p$-adic exponents, Dwork's definition is rather computationally simple. 
In this subsection, we show the equivalence of definition of CM-exponent and Dw-exponent in the case of one dimensional annuli.

Use $|\cdot|_\infty$ to denote the usual absolute value on $\RR$.
For $x\in \RR$, let $\langle x \rangle=\min_{s\in\ZZ}|x-s|_\infty$ denote the distance of $x$ to $\ZZ$.
For a sequence $a=\{a_h\}_{h\geq 0}$ in $\RR$, set 
$$\mathcal{N}:=\left\{ \{a_h\}_{h\geq 0}\in\RR^{\ZZ_{\geq 0}}:p^h\langle a_h\rangle = O(h) \right\}\text{ and}$$
$$ \mathcal{E}:= \left\{ a:=\{a_h\}_{h\geq 0}\in\RR^{\ZZ_{\geq 0}}:p^h\la pa_{h+1}-a_h \ra=O(1) \right\}.$$

\begin{defn}[\cite{CM2} D\'efinition 4.2-1]
    For $\alpha\in\ZZ_p$, and $h\geq 0$, let $\alpha_h\in\ZZ$ be an integer such that $\alpha_h\equiv \alpha\ (\mathrm{mod}\ p^h)$.
    The sequence $\{p^{-h}\alpha_h\}_{h\geq 0}$ in $ \RR$ is called an approximation sequence of $\alpha$. 
\end{defn}
It is an easy observation that if $\{p^{-h}\alpha_h\}$ is an approximation sequence of $\alpha$ then $\{-p^{-h}\alpha_h\}$ is an approximation sequence of $-\alpha$.
\begin{convention}
    For  $\textbf{a}=({\textbf{a}} ^1,\dots,\textbf{a}^r)$ a sequence of vectors in $\RR^r$, namely, each $\textbf{a}^i$ is a sequence in $\RR$. The $h$-th term of the sequence $\textbf{a}^i$ is denoted by $\textbf{a}^i_h$ and denote $\textbf{a}_h=(\textbf{a}_h^1,\dots,\textbf{a}_h^r)\in\RR^r$. For a permutation $\sigma$, set $\sigma(\textbf{a}_h)=(\textbf{a}_h^{\sigma(1)},\dots,\textbf{a}_h^{\sigma(r)})$ and $\sigma(\textbf{a})=(\textbf{a}^{\sigma(1)},\dots,\textbf{a}^{\sigma(r)})$.
\end{convention}
\begin{defn}[\cite{CM2} D\'efinition 4.4-1]
    Let $\textbf{a},\textbf{b}$ be sequences of vectors in $\RR^r$. We say $\textbf{a}$ is weakly equivalent to $\textbf{b}$ if there exists a sequence of permutations $\{\sigma_h\}_{h\geq 0}$ such that the sequence $\{\textbf{a}_h-\sigma_h(\textbf{b}_h)\}_{h\geq 0}$ in $\RR^r$ belongs to $\mathcal{N}^r$.
\end{defn}
\begin{thm}[\cite{CM2} Th\'eor\`em 4.2-3]\label{isom}
    The map 
    \begin{align*}
        \iota : \ZZ_p/\ZZ&\to \mathcal{E}/(\mathcal{N}\cap\mathcal{E})\\
        \alpha&\mapsto \{p^{-h}\alpha_h\}_{h\geq 0}
    \end{align*}
    is a well-defined isomorphism of groups.
\end{thm}
\begin{rmk}\label{weak equivalence of sequences}
    Let $A,B$ be elements in $\ZZ_p^r$ and let $\textbf{a},\textbf{b}$ be approximation sequences of $A,B$, respectively. It is easy to verify that if $A$ and $B$ are weakly equivalent, then $\textbf{a}$ and $\textbf{b}$ are weakly equivalent, with respect to the same sequence of permutations.
\end{rmk}
\begin{convention}
    In this subsection, suppose $I=(\alpha,\beta)\subset\RR$ is an open interval. Let $P$ be a $\nabla$-module of rank $r$ over $A^1_K(I)$ satisfying the Robba condition. Let $\varphi:t\mapsto t^p$ be the Frobenius homomorphism on $K_I$ over $K$. For $h\geq 0$, we use $P_h$ to denote the Frobenius antecedent of $P$ with respect to $\varphi^h$. That is, a $\nabla$-module over $A_K^1(I^{p^h})$ such that $\varphi^{h*}P:=P_h\otimes_{K_{I^{p^h}}}K_I=P$. Note that such a $P_h$ always exists uniquely and satisfies the Robba condition (cf. \cite{Ked1} Theorem 10.4.4). 
\end{convention}
\begin{defn}[\cite{CM2} D\'efinition 5.1-1]
    An $F$-basis $\mathcal{B}$ of $P$ is the data $\{\mathcal{B}_h\}_{h\geq 0}$ where each $\mathcal{B}_h$ is a basis of $P_h$.
\end{defn}
\begin{defn}[\cite{CM2} D\'efinition 5.1-2]\label{CMSma}
    Let $\textbf{a}=\{\textbf{a}_h\}_{h\geq 0}$ be a sequence of vectors in $\RR^r$. $\textbf{a}$ is said to be admissible for $P$ if there exists an $F$-basis $\mathcal{B}$ of $P$ for which the sequence of matrices
    $$K_h=x^{p^h\textbf{a}_h}\mathrm{M}(\varphi^{h*}(\mathcal{B}_h),\mathcal{B}_0)$$
    satisfies the following condition:
    there exists $l>0$ such that for all $\rho\in I$ and $h\geq 0$, $|K_h|_\rho\leq p^{hl}$ and $|K_h^{-1}|_\rho\leq p^{hl}$. Here $x^{p^h\textbf{a}_h}$ is the diagonal matrix
    $$
    \begin{pmatrix*}
       x^{{p^h}\textbf{a}^1_h} & & \\
       & \ddots &\\
       & & x^{{p^h}\textbf{a}^r_h}
    \end{pmatrix*}
    $$
    and $ \mathrm{M}(\varphi^{h*}(\mathcal{B}_h),\mathcal{B}_0)$ is the change-of-basis matrix from $\varphi^{h*}(\mathcal{B}_h)$ to $\mathcal{B}_0$.
    A vector $A$ of $\ZZ^r_p$ is said to be admissible if its approximation sequence is admissible.
\end{defn}
\begin{defn}[\cite{CM2} D\'efinition 5.1-7]
    A sequence $\textbf{a}=\{\textbf{a}_h\}_{h\geq 0}$ of vectors in $\RR^r$ is called well-ordered if $\textbf{a}\in \mathcal{E}^r$.    
\end{defn}
\begin{rmk}\label{remark on well-ordered}
    We summarize properties of admissible and well-ordered sequences:
    \begin{enumerate}
        \item  Suppose $\textbf{a}$ and $\textbf{b}$ are sequences of vectors in $\RR^r$ such that $\textbf{a}$ is admissible for $P$ and that $\textbf{b}$ is weakly equivalent to $\textbf{a}$. Then $\textbf{b}$ is also admissible for $P$ (cf. \cite{CM2} Corollaire 5.1-11).
        \item As a consequence of Theorem \ref{isom}, any approximation sequence induced by an element in $\ZZ_p^r$ is well-ordered. 
    \end{enumerate}
\end{rmk}
\begin{defn}[\cite{CM2} D\'efinition 5.3-6]
    An element $A$ of $\ZZ_p^r$ is called a CM-exponent of $P$ if one of its approximation sequences is well-ordered and admissible for $P$.
\end{defn}
An easy observation from Remark \ref{remark on well-ordered} and Remark \ref{weak equivalence of sequences} is that if $A$ is a CM-exponent of $P$, then any element of $\ZZ_p^r$ that is weakly equivalent to $A$ is also a CM-exponent of $P$. Now we verify the equivalence of definiton of $p$-adic exponents.

    \begin{prop}
        Let $A\in\ZZ_p^r$, then $A$ is a Dw-exponent of $P$ if and only if $-A$ is a CM-exponent of $P$.
    \end{prop}
    \begin{proof}
    Let $A$ be a Dw-exponent of $P$. If we fix a basis $e_1,\dots, e_r$ of $P$, then there exist a sequence of matrices $\{S_{h,A}\}_{h\geq 0}$ satisfying conditions in Lemma \ref{eqimat}. For $h\geq 0$, set $\mathcal{B}_h=(e_1,\dots,e_r)S_{h,A}t^{-A_h}$. This is a basis of $P_h$ because for any $\zeta\in\mu_{p^h}$ 
    \begin{align*}
        \zeta^*(\mathcal{B}_h)&=\zeta^*((e_1,\dots,e_r)S_{h,A}t^{-A_h})\\
        & = \zeta^*((e_1,\dots,e_r)S_{h,A}) (\zeta t)^{-A_h}\\
        & = (e_1,\dots,e_r)S_{h,A}\zeta^{A_h}(\zeta t)^{-A_h}\\
        & = (e_1,\dots,e_r)S_{h,A}t^{-A_h}=\mathcal{B}_h.
    \end{align*}
    Also, this implies that $\mathrm{M}(\mathcal{B}_0,\varphi^{h*}(\mathcal{B}_h))=S_{h,A}t^{-A_h}$ and if we set 
    $$K_h=t^{-A_h}\mathrm{M}(\varphi^{h*}(\mathcal{B}_h),\mathcal{B}_0)=S_{h,A}^{-1},$$
    it satisfies the condition required in Definition \ref{CMSma}.
    Hence $-A$ is admissible and so it is a CM-exponent of $P$.

    Conversely, suppose that $-A$ is a CM-exponent of $P$. Set $\mathcal{B}_0=\{e_1,\dots,e_r\}$ and $S_{h,A}=K_h^{-1}$. Then we have the following computation:
    \begin{align*}
        \zeta^*((e_1,\dots,e_r)S_{h,A}) &= \zeta^*((e_1,\dots,e_r)\mathrm{M}(\mathcal{B}_0,\varphi^{h*}(\mathcal{B}_h))t^{A_h})\\
        & = \zeta^*(\varphi^{h*}(\mathcal{B}_h))(\zeta t)^{A_h}\\
        & = \varphi^{h*}(\mathcal{B}_h)t^{A_h}\zeta^{A_h}\\
        & = (e_1,\dots,e_r) \mathrm{M}(\mathcal{B}_0,\varphi^{h*}(\mathcal{B}_h))t^{A_h}\zeta^{A_h}\\
        & = (e_1,\dots,e_r) S_{h,A}\zeta^{A_h}.
    \end{align*}
    This implies that $A$ is a Dw-exponent of $P$. 
\end{proof}
\end{appendix}

\nocite{*}
\normalem
\bibliographystyle{alpha}

\bibliography{references}

\end{document}